\documentclass[reqno,12pt]{amsart}

\usepackage{amssymb, amsfonts}
\usepackage{verbatim}
\usepackage{fullpage}

\usepackage{tikz-cd}

\usepackage{mathtools}
\usepackage{amssymb, amsfonts}
\usepackage[utf8]{inputenc}
\usepackage{amsthm}
\usepackage{amsmath}
\usepackage{amsxtra}
\usepackage{amscd}
\usepackage{amssymb}
\usepackage{mathtools}
\usepackage{xspace}
\usepackage{amsfonts}
\usepackage{todonotes}
\usepackage{enumerate}
\usepackage{multirow}
\usepackage{tikz}
\usepackage{cancel}
\usepackage{tikz-cd}
\usepackage[all]{xy}
\usepackage{amscd}
\usepackage{caption}
\usepackage{epsfig}
\usepackage{graphicx,transparent}
\usepackage{here}
\usepackage{setspace}
\usepackage{wasysym}
\usepackage{xparse}
\usepackage{manfnt}
\usepackage[hidelinks]{hyperref}
\usepackage{url}
\usepackage{mathrsfs}
\usepackage{color}   %May be necessary if you want to color links
\usepackage{cleveref}

\hypersetup{
    colorlinks=true, %set true if you want colored links
    linktoc=all,     %set to all if you want both sections and subsections linked
    linkcolor=blue,  %choose some color if you want links to stand out
}

\usepackage{tikz-cd}

\usepackage{mathtools}
\usepackage{subcaption}

\usepackage{amsmath}
\usepackage{tikz}
\usepackage{mathdots}
\usepackage{yhmath}
\usepackage{cancel}
\usepackage{color}
\usepackage{siunitx}
\usepackage{array}
\usepackage{multirow}
\usepackage{amssymb}
\usepackage{tabularx}
\usepackage{extarrows}
\usepackage{booktabs}
\usetikzlibrary{fadings}
\usetikzlibrary{patterns}
\usetikzlibrary{shadows.blur}
\usetikzlibrary{shapes}

\newtheorem{theorem}[equation]{Theorem}
\newtheorem{lemma}[equation]{Lemma}
\newtheorem{proposition}[equation]{Proposition}
\newtheorem{corollary}[equation]{Corollary}

\theoremstyle{definition}
\newtheorem{definition}[equation]{Definition}

\theoremstyle{remark}
\newtheorem{remark}[equation]{Remark}

\numberwithin{equation}{section}

\DeclareMathOperator{\res}{res}

\newcommand{\cR}{\mathcal{R}}
\newcommand{\cC}{\mathcal{C}}

\newcommand{\cM}{\mathcal{M}}
\newcommand{\cQ}{\mathcal{Q}}

\newcommand{\cH}{\mathcal{H}}

\newcommand{\PP}{\mathbb{P}}

\newcommand{\cO}{\mathcal{O}}
\newcommand{\cA}{\mathcal{A}}
\newcommand{\bP}{\mathbb{P}}
\newcommand{\bQ}{\mathbb{Q}}

\newcommand{\marg}[1]{\normalsize{{\color{red}\footnote{#1}}}{\marginpar[\vskip -.3cm {\color{red}\hfill\tiny{\thefootnote}\normalsize$\implies$}]{\vskip -.3cm{ \color{red}$\impliedby$\tiny\thefootnote}}}}

\newcommand{\Myeongjae}[1]{\marg{\color{magenta}(Myeongjae): #1}}

\newcommand{\prabhat}[1]{\textcolor{red}{[Prabhat: #1]}}

\begin{document}

\title[Cohomology of strata]{Cohomology ring of the strata of holomorphic differentials of genus three}
\author[P. Devkota]{Prabhat Devkota}
\address{Institut f\"ur Mathematik\\
Goethe-Universit\"at Frankfurt\\
Robert-Mayer-Stra{\ss}e 6-8\\
Frankfurt am Main, 60325 Germany}
\email{devkota@math.uni-frankfurt.de}
\author[M. Lee]{Myeongjae Lee}
\address{Department of Mathematics\\
Indiana University\\
831 E.\ Third Street\\
Bloomington, IN 47405\\
USA}
\email{myelee@iu.edu}
\thanks{}

\begin{abstract}
We prove that the rational Chow ring of every connected component of the multi-scale compactifications of the genus three strata $\cH(4), \cH(2,2)$ and $\cH(3,1)$ is generated by boundary-stratum classes. Our proof shows Chow-K\"unneth generation property of boundary strata of the multi-scale compactification by reducing the boundary analysis to lower-genus generalized strata of meromorphic differentials, including residueless loci, described through plane-quartic and plane-cubic models. As a consequence, the rational cycle-class map is an isomorphism and the compactifications have no odd rational cohomology. 
\end{abstract}

\maketitle

\tableofcontents

\section{Introduction} \label{sec:intro}
Given an $n$-tuple of integers $(m_1,\ldots,m_n)$ summing to $2g-2$, the moduli space of meromorphic differentials (or {\em stratum of differentials}) $\Omega\cM_{g,n}(\mu)$ parameterizes meromorphic 1-forms $\omega$ on an $n$-pointed curve $(C;p_1,\ldots,p_n)\in \cM_{g,n}$ of genus $g$ with prescribed orders of zeros and poles at $p_i$, i.e., $\text{div}(\omega)=\sum_im_ip_i$. If all $m_i$'s are non-negative then $\Omega\cM_{g,n}(\mu)$ is naturally a subvariety of the Hodge bundle $\Omega\cM_{g,n}\rightarrow \cM_{g,n}$ -- the vector bundle over $\cM_{g,n}$ whose fiber over $(C;p_1,\ldots,p_n)$ is the vector space $H^0(C,\omega_C)$ of holomorphic 1-forms on the curve $C$. There is a natural action of $\mathbb{C}^*$ on $\Omega\cM_{g,n}(\mu)$ by rescaling the differential. The resulting quotient $\mathbb{P}\Omega\cM_{g,n}(\mu)$ will also be called the {\em moduli space of meromorphic differentials}. For notational simplicity, we will put $\cH_{g,n}(\mu)\coloneqq\mathbb{P}\Omega\cM_{g,n}(\mu)$, and often even drop $g,n$ or even $\mu$ from the notation and denote it simply by $\cH(\mu)$ or simply $\cH$, if $g,n$ and $\mu$ are self-evident from the context. Of course there is a natural forgetful morphism $\cH(\mu)\rightarrow \cM_{g,n}$ that forgets the underlying meromorphic differntial. This morphism is injective, so we will frequently identify $\cH(\mu)$ as a closed substack of $\cM_{g,n}$.

The moduli space $\cH_{g,n}(\mu)$ is a smooth Deligne-Mumford stack, but it may not be connected (though it frequently is). The seminal paper \cite{kozo1} by Kontsevich and Zorich proves that if all $m_i\ge 0$ then $\cH(\mu)$ can have up to three connected components: at most one hyperelliptic component and up to two non-hyperelliptic components. Analogous classification of the components when some $m_i$ is negative is done in \cite{boissymero}. Similarly, $\cH(\mu)$ is not usually proper. However, its na\"ive closure in $\overline\cM_{g,n}$ is not smooth. Nonetheless, a smooth (as a Deligne-Mumford stack) compactification $\overline\cH_{g,n}(\mu)\coloneqq \mathbb{P}\Xi\overline\cM_{g,n}(\mu)$, called the {\em moduli space of multi-scale differentials} (or the {\em multi-scale compactification} of $\cH_{g,n}(\mu))$, was constructed in \cite{BCGGM}. It is a smooth proper Deligne-Mumford stack with projective coarse moduli space and normal crossing boundary divisor $\partial\cH_{g,n}(\mu)\coloneqq \overline\cH_{g,n}(\mu)\setminus\cH_{g,n}(\mu)$. This compactification admits a natural proper morphism $\overline\cH_{g,n}(\mu)\rightarrow \overline\cM_{g,n}$ which is birational onto its image.

While there have been a lot of developments in the study of Chow and cohomology rings of $\cM_{g,n}$ and $\overline\cM_{g,n}$, the analogous questions for $\cH(\mu)$ and $\overline\cH(\mu)$ have been relatively poorly understood. Of course in genus 0, the situation is markedly simpler: $\cH_{0,n}(\mu)=M_{0,n}$, whereas the compactification $\overline\cH_{0,n}(\mu)$ was shown in \cite{Devkotacohomology} to have its cohomology (and Chow) ring generated by the boundary divisors, paralleling the situation for $\overline{M}_{0,n}$ in \cite{keel}. However, already for $g=1$ and $n=2$, the situation is markedly different to that for $\overline\cM_{g,n}$: the (primitive component of the) moduli space $\overline\cH(m,-m)$ is the classical modular curve $X_1(m)$ which is rational only for $m\le 10$ or $m=12$, see e.g. \cite{mazur}.

For $g\ge 2$, the situation is even more mysterious. Of course the hyperelliptic component of a holomorphic stratum (i.e. $\mu=(m_1,\ldots,m_n)$ with $m_i\ge 0$) is easier to study since its image in $\overline\cM_{g,n}$ is finitely covered by $\overline{M}_{0,n}$ (at least on the level of coarse moduli spaces). Additionally, the Chow rings of the interior of several holomorphic strata in lower genera, especially the "non-varying strata" are quite well understood \cite{chen_gorensteinsingularity}. However, even for a stratum of holomorphic differentials $\cH(\mu)$, the boundary $\partial\cH(\mu)$ necessarily parameterizes meromorphic differentials. So, extending the computation of Chow ring to the compactification $\overline\cH(\mu)$ is not necessarily straightforward. In fact, using the geometry of the boundary $\partial\cH(\mu)$, the authors of \cite{devkota_grushevsky} constructed several odd cohomology classes as well as even degree (nontautological) cohomology classes on a large class of the compactified strata.\\

\textbf{Convention:} The Chow rings and cohomology rings will be taken with $\mathbb{Q}$-coefficients, unless explicitly noted otherwise. In particular, for brevity, $H^*(X)$ will mean $H^*(X;\mathbb{Q})$ and $CH_*(X)$ will mean $CH_*(X)\otimes\mathbb{Q}$, unless explicitly noted otherwise. Additionally, all the schemes or stacks in consideration will be over the $\mathbb{C}$, and for $X$ a smooth Deligne-Mumford stack, we will use the convention $CH^k(X)\coloneqq CH_{n-k}(X)$, where $n=\text{dim}_{\mathbb{C}}(X)$.

\subsection{Main results}
Inspired by the recent developments in the computations of Chow and cohomology rings of $\overline\cM_{g,n}$ -- see e.g. \cite{cl22} and \cite{clp23} -- in this paper, we build up on similar techniques to study the cohomology ring of $\overline\cH_{g,n}(\mu)$, especially in genus 3 when $\mu$ is of strictly homolorphic type (i.e. $\mu=(m_1,\ldots,m_n)$ with $m_i>0$).
\begin{theorem}\label[theorem]{thm:main_thm}
    The strata $\overline{\cH}(\mu)$ has CKgP for $\mu=(4),(2,2)$ and $(3,1)$.
\end{theorem}
The main tool is the Chow-K\"unneth generation Property (CKgP), which has been quite successful recently in the computation of the Chow and cohomology rings of $\overline\cM_{g,n}$ for small $g,n$ (see \cite{cl22}). In particular, we will use the stratification of $\overline\cH(\mu))$ into boundary strata, and prove that each of them have the CKgP to prove that $\overline\cH(\mu)$ also has the CKgP. Note that $\cH(\mu)$ has CKgP for $\mu=(4)$ or $(2,2)$ or $(3,1)$ or $(2,1,1)$ by the results of \cite{chen_gorensteinsingularity}.

Since $\overline\cH(\mu)$ is a smooth and proper Deligne-Mumford stack, the Chow-K\"unneth generation Property yields the following:
\begin{corollary}\label[corollary]{cor:main_cor}
    The rational cohomology ring and the Chow ring of $\overline{\cH}(\mu)$ for $\mu=(4),(2,2)$ and $(3,1)$ is generated by the classes of boundary strata. 
\end{corollary}
Of course, in the course of proving \Cref{thm:main_thm} and \Cref{cor:main_cor}, we prove the analogous statements for the strata of (meromorphic) differentials (possibly with linear residue conditions on the poles) in lower genus, that manifest in the boundary of $\overline\cH(\mu)$. In particular, for $\mu=(4),\ (2,2)$ and $(3,1)$, we show that the moduli spaces $\cH(\mu,-2)$ and $\cH(\mu,-2,-2)$ parameterizing meromorphic differentials in genus two and one, respectively, have the CKgP and their Chow rings are generated by their respective fundamental classes. We also study the genus one residueless strata $\cR(4,-2,-2)$, $\cR(2,2,-2,-2)/S_2$, and $\cR(3,1,-2,-2)/S_2$, where $S_2$ permutes the two double poles.

Recall that the moduli spaces of multi-scale differentials also admit the tautological rings $R^*(\overline\cH)\subset CH^*(\overline\cH)$ and $RH^*(\overline\cH)\subset H^*(\overline\cH)$ paralleling the definition for $\overline\cM_{g,n}$ -- see \cite[\S 8]{CoMoZaEU} for precise definition. In particular, the largest version of the tautological ring discussed in {\em loc. cit.} was defined to contain also the fundamental classes each component of the boundary stratum $D_\Gamma$ as well as the horizontal boundary stratum $D_h$ (instead of the classes $[D_\Gamma]$ and $[D_h]$ collectively). Our results then shows that the moduli spaces $\overline\cH(\mu)$ for $\mu=(4),\, (2,2)$ and $(3,1)$ have the Chow (as well as the cohomology) rings isomorphic to the tautological ring (the largest version).

\begin{remark}
The moduli spaces $\cR(\mu,-2,-2)$ parameterizing residueless meromorphic differentials in genus one underlie our main difficulty in extending the results to the signatures $\mu=(2,1,1)$ and $(1,1,1,1)$. Our arguments for $\mu=(4),\, (2,2),\, (3,1)$ are ad-hoc, and exploit the fact that $\cR(\mu,-2,-2)$ is at most two-dimensional, and we do not know how to generalize to the remaining two signatures $\mu=(2,1,1),\, (1,1,1,1)$. Nevertheless, we note that for $\mu=(2,1,1)$ (resp $\mu=(1,1,1,1)$), the $S_2$ quotient $\overline\cH(\mu)/S_2$, symmetrizing any pair of simple zeros, admits a birational morphism onto $\overline{M}_{3,1}$ (resp. $\overline{M}_{3,2}$) simply by forgetting the two symmeterized simple zeros. Then by analyzing the exceptional divisors of this birational morphism, e.g. via an application of \Cref{thm:ckgp_fibration} as in \S5.2, we can prove that $\overline\cH(\mu)/S_2$ also has the CKgP in these two cases.
\end{remark}

\subsection{Idea of the proof and outline of the paper}
Of course, the idea of the proof is to prove that all the (open) boundary strata of $\overline\cH(\mu)$ have the CKgP, with the the Chow ring generated by the fundamental classes of each component. Note that the boundary strata of $\overline\cH(\mu)$ are not necessarily finitely covered by the product of the moduli spaces of lower genus (as is the case for the boundary strata in $\overline\cM_{g,n}$) -- instead, they are only related by a finite correspondence \cite[\S 4]{CoMoZaEU}. So, in addition to showing that all the lower genus moduli spaces of differentials that manifest in the boundary have the desired property, we check the finite correspondence still retains the property that we want.

On the other hand, we cannot expect all the moduli spaces of meromorphic differentials even in low genus (say 1 or 2) with few enough marked points to have the CKgP, as already seen with the modular curves $\cH(m,-m)$. As such, we need to list all the lower genus moduli spaces that appear in the boundary of $\cH(\mu)$ for $\mu=(4),\,(2,2)\,(3,1)$, and show one by one that they have the CKgP. This adds a layer of combinatorial complication, since the number of boundary strata gets untenably large as the codimension increases. However, we argue in \Cref{rem:codimension_up_to_two} that for the boundary strata represented by the level graphs with only vertical edges, we do not need to go beyond those with three levels (i.e. codimension two). The list of codimension two boundary strata will be relegated to Appendix \ref{appendix}. On the other hand, the (closed) horizontal boundary $D_h$ can self-intersect, so we deal with it separately by realizing it as the finite image of $\overline\cH(\mu,-1,-1)/S_2$, which, in turn, admits a birational morphism onto $\overline{M}_{2,n}$ (for $n=1$ or $2$), simply by forgetting the simple poles. We use this birational morphism to deal with the closed horizontal boundary in one fell swoop.

This paper will be organized as follows:
\begin{itemize}
    \item In Section~\ref{sec:background}, we recall the basics of the multi-scale compactification of the strata of differentials that will be necessary in our discussions. Additionally, we will also review the notion of Chow-K\"unneth generation Property (CKgP) and state several of its features and consequences.
    \item In~\Cref{sec:genustwo} and~\Cref{sec:genusone}, we show some of the moduli spaces in genus one and two appearing in the boundary of the genus 3 holomorphic strata $\overline\cH(\mu)$ have the CKgP. In particular, in~\Cref{sec:genustwo}, we deal with the genus two moduli spaces $\cH(\mu,-2)$ and in~\Cref{sec:genusone}, we deal with the genus one moduli spaces $\cH(\mu,-2,-2)$, where $\mu$ is one of $(4),\,(2,2)$ and $(3,1)$.
    \item In Section~\ref{sec:fs}, we prove \Cref{thm:main_thm} and \Cref{cor:main_cor} for the minimal stratum $\overline\cH(4)$.
    \item In \Cref{sec:h22}, we prove \Cref{thm:main_thm} and \Cref{cor:main_cor} for the stratum $\overline\cH(2,2)$. In the process, we also show that the stratum of residueless differentials $\cR(2,2,-2,-2)/S_2$ has the CKgP, with the Chow ring generated by the fundamental classes of the components.
    \item In Section~\ref{sec:h31}, we prove \Cref{thm:main_thm} and \Cref{cor:main_cor} for the stratum $\overline\cH(3,1)$. In the process, we also show that the stratu of residueless differentials $\cR(3,1,-2,-2)/S_2$ has the CKgP, with the Chow ring generated by its fundamental class.
    \item In Appendix~\ref{appendix}, we list all the enhanced level graphs of codimension 2 with only vertical edges in the strata $\overline\cH(\mu)$ for the three cases. We will then see that all the vertices in the level graph represent moduli spaces with CKgP and have the Chow rings generated by the fundamental classes.
    
\end{itemize}
\subsection*{Acknowledgments:} The first author would like to thank his PhD advisor Samuel Grushevsky for his support and several fruitful discussions throughout the project. The first author would also like to thank Shubham Saha for making him aware of the paper \cite{casnati_rationality} by G. Casnati and C. Fontanari. Additionally, the research of the first author was supported in part by the Humboldt Research Fellowship.\\
Some parts of the proof of \Cref{prop:chow_leray_hirsch}, especially the idea to use Voisin's spreading argument \cite[Theorem 3.1]{voisinchowdecomposition}, was inspired by the use of generative AI.

\section{Background} \label{sec:background}

\subsection{Multi-scale compactification}

In \cite{BCGGM1} and \cite{BCGGM}, Bainbridge,
Chen, Gendron, Grushevsky, and M\"{o}ller constructed the moduli space of multi-scale differentials $\overline{\cH}(\mu)$, which is a smooth compact complex orbifold containing $\cH(\mu)$ as a dense subset, and such that the boundary divisor $\partial \overline{\cH}(\mu)$ is normal crossing. 

Similarly for generalized strata $\cH(\mu^{\mathfrak{R}})\subset \cH(\mu)$, which is a sublocus where the residues of poles satisfy certain linear equations denoted by $\mathfrak{R}$, has a smooth compactification $\overline{\cH}(\mu^{\mathfrak{R}})$.

These compactifications are parameter space of multi-scale differentials, which encodes the degeneration of meromorphic differentials.

\begin{definition}
    An {\em enhanced level graph} $\Gamma$ of type $\mu$ on a marked nodal curve $(X,\boldsymbol{z} \sqcup \boldsymbol{p})$ consists of the following data:
    \begin{itemize}
        % More precise language and consistent terminology
        \item A dual graph of $(\overline{X}, \boldsymbol{z} \sqcup \boldsymbol{p})$, where each leg $z$ or $p$ in the set of half-edges $H(\Gamma)$ is decorated with the integer prescribed by $\mu$.
        
        \item A level function on the set of vertices:
        \[
        \ell: V(\Gamma) \longrightarrow \{0, -1, -2, \ldots, -N\}
        \]
        which is surjective.
        
        \item An edge $e \in E(\Gamma)$ is called {\em vertical} if it connects vertices of distinct levels; otherwise, $e$ is called {\em horizontal}.
        
        \item Each vertical edge $e$ is assigned a positive integer $\kappa_e$, called its {\em enhancement}.
    \end{itemize}

    For each $v \in V(\Gamma)$, let $g_v$ denote the genus of the irreducible component $X_v$ of $\overline{X}$. Then the following relation holds:
    \[
    2g_v - 2 = \sum_{z_i \mapsto v} a_i 
    - \sum_{p_j \mapsto v} b_j 
    + \sum_{e \in E^+(v)} (\kappa_e - 1)
    - \sum_{e \in E^-(v)} (\kappa_e + 1)
    - \sum_{e \in E^h(v)} 1
    \]
    where:
    \begin{itemize}
        \item[-] The first (second, respectively) sum is over all zeroes (poles) incident to $v$.
        \item[-] $E^+(v)$, $E^-(v)$, and $E^h(v)$ denote the sets of half-edges incident to $v$ going to lower, upper, and same-level vertices, respectively.
    \end{itemize}
\end{definition}

\begin{definition}
    A {\em twisted differential} compatible with an enhanced level graph $\Gamma$ of type $\mu^\mathfrak{R}$ on $(X, \boldsymbol{z} \sqcup \boldsymbol{p})$ is a collection of differentials $\boldsymbol{\eta} = (\eta_v)_{v \in V(\Gamma)}$ satisfying the following conditions:
    \begin{enumerate}
        \item \textbf{Prescribed orders:} The orders of the singularities of $\eta_v$ are determined by the decorations on the half-edges at $v$.
        
        \item \textbf{Balancing of horizontal nodes:} For every horizontal node $e$,
        \[
        \res_{e^+} \boldsymbol{\eta} + \res_{e^-} \boldsymbol{\eta} = 0.
        \]
        
        \item \textbf{$\mathfrak{R}$-Global Residue Conditions (GRCs):} Let $\widetilde{\Gamma}^\mathfrak{R}$ be an auxiliary level graph obtained by adding a vertex $v_R$ at level $\infty$ for each part $R \in \mathfrak{R}$, and connecting each $p_j$ ($j \in R$) to $v_R$. Let $\widetilde{\Gamma}^\mathfrak{R}_{>L}$ denote the subgraph consisting of half-edges and vertices strictly above level $L$.

        For every finite level $L$, any connected component $Y$ of $\widetilde{\Gamma}^\mathfrak{R}_{>L}$ containing all poles of entire parts $R \in \mathfrak{R}$ must satisfy the residue condition:
        \[
        \sum_{e \in E^+(Y, L)} \res_e \boldsymbol{\eta} = 0,
        \]
        where $E^+(Y, L)$ is the set of edges connecting $Y$ to vertices at level $L$.
    \end{enumerate}
    
    For brevity, we refer to the $\mathfrak{R}$-Global Residue Conditions as {\em $\mathfrak{R}$-GRCs}.
\end{definition}

A {\em multi-scale differential} in $\overline{\cH}(\mu^\mathfrak{R})$ is a tuple
\[
(X, \boldsymbol{z} \sqcup \boldsymbol{p}, \Gamma, \boldsymbol{\eta}, \boldsymbol{\sigma}),
\]
where:
\begin{itemize}
    \item $(X, \boldsymbol{z} \sqcup \boldsymbol{p})$ is a pointed nodal curve;
    
    \item $\Gamma$ is an enhanced level graph of type $\mu^\mathfrak{R}$ on $(X, \boldsymbol{z} \sqcup \boldsymbol{p})$;
    
    \item $\boldsymbol{\eta} = (\eta_v)_v$ is a {\em twisted differential} compatible with the enhanced level graph $\Gamma$;
    
    \item $\boldsymbol{\sigma} = (\sigma_e)_e$ is a collection of cyclic-order-reversing maps identifying the geodesic rays on the two branches at each node. We call $\boldsymbol{\sigma}$ the {\em global prong-matching}, and it can be represented by an element in the {\em prong-matching group}
    \[
    P_\Gamma := \prod_e \mathbb{Z} / \kappa_e \mathbb{Z}.
    \]
\end{itemize}
The multiscale differentials only retain information on each level up to projectivization, so we have to consider the action by rescaling level-by-level of a multiplicative torus, called the level rotation torus $T_\Gamma$, isogeneous to $(\mathbb{C}^*)^{L(\Gamma)+1}$, on the space of twisted differentials compatible with the enhanced level graph $\Gamma$ (here, $L(\Gamma)$ is the number of levels below the top). This action also has the effect of twisting the prong matchings on the vertical edges. The moduli space of multiscale differentials then parametrizes the equivalence classes of multiscale differentials under such action.We refer to \cite{BCGGM} or \cite{CoMoZaEU} for details and recap what we will need about the equivalence classes of prong matchings:

Given an enhanced level graph $\Gamma$ with enhancements $\{\kappa_e\}$ on the vertical edges, there are altogether $\prod_e\kappa_e$ possible choices of prong matchings on a compatible multiscale differential. Then understanding the equivalence classes of prong matchings amounts to understanding the equivalence classes under the rotation of the prongs by the level rotation group $R_\Gamma\cong \mathbb{Z}^{L(\Gamma)}$ whose $i^{\text{th}}$ factor acts on the $i^{\text{th}}$ level passage (that is, all the edges that cross the virtual horizontal line right above the level $-i$) by diagonally turning in a fixed direction the the prong-matchings at each edge crossing this level passage. The stabilizer of all the prong-matchings simultaneously is called the \textit{twist group}. The subgroup \textit{simple twist group} $Tw_\Gamma^s\subset Tw_\Gamma$ consists of elements that can be written as a product of prong rotations that act on one level passage only. The quotient $Tw_\Gamma/Tw_\Gamma^s$ is called the \textit{group of ghost automorphisms}. As explained in \cite[\S 2.2]{ccm}, the isotropy group of a multi-scale differential $(C,\omega)$ is an extension of the group of automorphisms $\text{Aut}(C,\omega)$ by the group of ghost automorphisms.

We will frequently use the term ``codimension of a level graph'' to mean the number of levels below the top (that is, the number of level passages) plus the number of horizontal edges. Indeed, the codimension of a level graph coincides with the codimension of the associated boundary stratum. In particular, a level graph of codimension one is called ``divisorial''. We will also call a boundary stratum given by a divisorial level graph with only one horizontal edge a ``horizontal divisor'' and those with no horizontal edge a ``vertical divisor''. Additionally, for an enhanced level graph $\Gamma$ with enhancements $\{\kappa_e\}$, we will use the following notations: $\ell_\Gamma\coloneq \text{lcm}(\kappa_e)$ and $K_\Gamma=\prod\kappa_e$.

If all poles are residueless then we will denote $\cH(\mu^{\mathfrak{R}})$ and $\overline\cH(\mu^{\mathfrak{R}})$ simply by $\cR(\mu)$ and $\overline\cR(\mu)$, respectively. Similarly, if the signature $\mu$ (or $\mu^{\mathfrak{R}}$) is clear (or not important), we will denote $\cH(\mu)$ and $\overline\cH(\mu)$ simply by $\cH$ and $\overline\cH$.

\subsection{Boundary strata of $\overline{\cH}(\mu)$}\label{sec:boundary_correspondence}
The boundary $\partial\cH \coloneq \overline\cH\setminus\cH$ of the multi-scale compactification is stratified by the enhanced level graphs. Given an enhanced level graph $\Gamma$, we will denote by $D_\Gamma$ the (closed) boundary stratum given by the closure of the locus of multi-scale differentials compatible with the enhanced level graph $\Gamma$.

For each level $i$ of an enhanced level graph $\Gamma$ with $L+1$ levels, we denote by $\Gamma^{(i)}$ the subgraph of $\Gamma$ consisting of all the vertices and half-edges in the level $i$ of $\Gamma$ (along with the horizontal edges). Then we consider the product $B_\Gamma^{[i]}$ of the unprojectivized generalized strata of differentials at all vertices $v\in V(\Gamma^{(i)})$, subject to all the relevant residue conditions, and up to the action of $\mathbb{C}^*$ that rescales all the differentials on each of the components simultaneously. If $i=0$ or $i=-L$, we will also denote $B_\Gamma^{[i]}$ by $B_\Gamma^\top$ or $B_\Gamma^\bot$, respectively. Put $B_\Gamma\coloneqq\prod_{i=-L}^0 B_\Gamma^{[i]}$ then by \cite[Prop. 4.4]{CoMoZaEU}, there is a proper stack $D_{\Gamma}^s$ (called a simple boundary stratum) and finite morphisms \[c_\Gamma:D_\Gamma^s\rightarrow D_\Gamma\quad \text{and}\quad p_\Gamma:D_\Gamma^s\rightarrow B_\Gamma,\]
ramified only along the boundary, such that $\frac{\text{deg}(p_\Gamma)}{\text{deg}(c_\Gamma)}=\frac{K_\Gamma}{\ell_\Gamma\cdot|\text{Aut}(\Gamma)|}$. In fact, the discussions of \cite[\S 4]{CoMoZaEU} shows something more. Suppose $\Delta$ is any degeneration of $\Gamma$ and denote by $D_{\Gamma,\Delta}$ (resp. $B_{\Gamma,\Delta}$) the subset of $D_\Gamma$ (resp. $B_\Gamma$) parameterizing multi-scale differentials compatible with an undegeneration of $\Delta$. Then there analogously is a stack $U_{\Gamma,\Delta}$ and finite morphisms $c_{\Gamma,\Delta}:U_{\Gamma,\Delta}\rightarrow D_{\Gamma,\Delta}$ and $p_{\Gamma,\Delta}:U_{\Gamma,\Delta}\rightarrow B_{\Gamma,\Delta}$. The aforementioned simple boundary stratum $D_\Gamma^s$ is then constructed as the normalization of a cover that dominates all $U_{\Gamma,\Delta}$'s.

In particular, if $\Gamma=\Delta$ then we will frequently denote the open subsets $D_{\Gamma,\Gamma}$ and $B_{\Gamma,\Gamma}$ by $D_\Gamma^o$ and $B_\Gamma^o$, respectively. And the finite morphisms (which are now \'etale covers) $c_{\Gamma,\Gamma}$ and $p_{\Gamma,\Gamma}$ by $c_\Gamma^o$ and $p_\Gamma^o$, respectively. The degrees of these morphisms are explicit: $\text{deg}(p_\Gamma^o)=K_\Gamma/[R_\Gamma:Tw_\Gamma]$ (i.e. the number of prong-matching equivalence classes) and $\text{deg}(c_\Gamma^o)=[Tw_\Gamma:Tw_\Gamma^s]\cdot |\text{Aut}(\Gamma)|$, so that $\text{deg}(p_\Gamma)/\text{deg}(c_\Gamma)=\text{deg}(p_\Gamma^o)/\text{deg}(c_\Gamma^o)$.

\subsection{Chow-K\"unneth generation Property}\label{subsection:ckgp_intro}
In this section, we will summarize the basics about Chow-K\"unneth generation property (CKgP) which will play an important role in our discussion later. For more details, we refer to \cite[\S 3]{cl22} and \cite[\S 4]{clp23}.
\begin{definition}
  A space (scheme or an algebraic stack) $X$ is said to have the Chow-K\"unneth generation Property (abbreviated CKgP) if for all spaces $Y$ (of finite type and admitting a stratification by global quotient stacks), the cup product map on the rational Chow rings:
  \[CH_*(X)\otimes CH_*(Y)\rightarrow CH_*(X\times Y)\]
  is surjective.
\end{definition}
Now, we enumerate few important features of CKgP that will be useful later:
\begin{enumerate}
\item If $X$ has the CKgP and $U\subset X$ is open then $U$ also has the CKgP \cite[Prop. 4.2(1)]{clp23}.
\item If $X\rightarrow Y$ is a proper and surjective morphism of Deligne-Mumford stacks and $X$ has the CKgP then so does $Y$ \cite[Prop. 4.2(2)]{clp23}.
\item If $X$ admits a finite stratification $X=\bigsqcup_{S\in \Sigma}$ such that each $S$ is locally closed and has the CKgP then so does $X$ \cite[Lemma 3.4]{cl22}.
\item If $\pi:V\rightarrow X$ is an affine bundle then $V$ has the CKgP if and only if $X$ does \cite[Lemma 3.5]{cl22}.
\item Suppose $X$ is a stack that admits a coarse moduli space $X\rightarrow M$ then $X$ has the CKgP if and only if $M$ does \cite[Lemma 3.9]{cl22}.
\item If $X_1$ and $X_2$ are spaces with the CKgP then $X_1\times X_2$ also has the CKgP \cite[Lemma 3.2]{cl22}.
\item If $X$ is a smooth, proper Deligne-Mumford stack that has the CKgP, then the cycle class map $\text{cl}:CH^*(X)\rightarrow H^*(X,\mathbb{Q})$ is a ring isomorphism. In particular, $X$ has no odd rational cohomology \cite[Lemma 3.11]{cl22}.
\end{enumerate}
\begin{remark}
    If $X$ is a rational surface then $X$ has the CKgP. This can be easily seen as an easy consequence of the feature (iii) above. Indeed, resolving the birational map $X\dashrightarrow\mathbb{P}^2$, we can assume $X$ is proper and admits a birational morphism onto $\mathbb{P}^2$. Then $X=U\cup E$, where $U\subset \mathbb{P}^2$ and $E$ is a union of rational curves, each of which has the CKgP.
\end{remark}
In addition to these features of CKgP, we will also need to understand the following scenario: given $X\rightarrow Y$ a birational morphism with exceptional divisor $E\subset X$ centered over $Z\subset Y$, if $Z$ and $Y$ both have the CKgP, when can we conclude $X$ also has the CKgP? To this end, the following result will play important role in our discussion.

\begin{proposition}[Chow-theoretic Leray-Hirsch]\label[proposition]{prop:chow_leray_hirsch}
    Let $f:X\rightarrow Y$ be a proper morphism between smooth Deligne-Mumford stacks (not necessarily compact). Suppose that:
    \begin{enumerate}
    \item Each fiber $F_y$ is a smooth, proper and irreducible Deligne-Mumford stack with finitely generated rational Chow ring.
    \item There are classes $\xi_1,\ldots,\xi_\ell\in CH^*(X)$ that descend to independent classes in $H^*(X)$, also denoted $\xi_1,\ldots,\xi_\ell$, such that for every fiber $F_y\subset X$, both the cohomology ring $H^*(F_y,\bQ)$ and the Chow ring $CH^*(F_y)$ are $\bQ$-linearly generated by $\xi_1|_{F_y},\ldots,\xi_\ell|_{F_y}$.
    \end{enumerate}
    Then the morphism $\phi:CH_*(Y)\otimes CH_*(F)\cong \bigoplus_{i=1}^\ell \xi_i\cdot CH_*(Y)\rightarrow CH_*(X)$ is surjective, where $F$ is a fiber of $f$ (which is unambiguous since they all have isomorphic rational Chow ring).
\end{proposition}
Assuming $f$ is a pure Tate fiber bundle, a stronger statement for {\em integral} Chow ring was proven in \cite[Theorem 2.1]{motivic_leray_hirsch}.
\begin{proof}
    Since $CH^*(F)$ is finitely generated, by \cite[Theorem 7.1]{clpeleventhcohomology}, $CH^*(F)\cong H^*(F,\bQ)$ has the Poincar\'e duality. We denote by $\zeta_1,\ldots\zeta_\ell$ the fiberwise Poincar\'e dual basis of $\xi_1,\ldots,\xi_\ell$, that is, $\xi_i|_y\cdot \zeta_j|_y=\delta_{i,j}[F_y]$ for each $y$. Suppose $p_i$ (for $i=1,2$) are the projections of $X\times_Y X$ onto the two factors. Then the cycle class
    \[\sum_i p_1^*\xi_i\cdot p_2^*\zeta_j\in CH^*(X\times_Y X)\] restricts on each fiber $F_y\times F_y$ to the class of the diagonal $\Delta_{F_y}$ (see, e.g. \cite[Theorem 11.11]{milnorstasheff}). So if $\Delta_{X/Y}$ is the relative diagonal of $X$ over $Y$, the cycle class
    \[\mathcal{Z}\coloneqq \Delta_{X/Y}-\sum_i p_1^*\xi_i\cdot p_2^*\zeta_j\in CH^*(X\times_Y X)\] restricts to 0 on each fiber $F_y\times F_y$. We will prove that $\mathcal{Z}$ is contained in the image of $CH_*(Y)\otimes CH_*(F\times F)$ in $CH(X\times_Y X)$. Note that the fibration $g:X\times_Y X\rightarrow Y$ also satisfies the assumptions of the Proposition, except the Chow ring of all fibers are generated by $p_1^*\xi_i$'s and $p_2^*\xi_j$'s.
    
    Let us choose some representatives for the classes $p_1^*\xi$ and $p_2^*\zeta_j$ (which we denote by the same notation, for brevity) so that $\mathcal{Z}$ is now an honest cycle. Then, by \cite[Theorem 3.1]{voisinchowdecomposition} applied to the morphism $g:X\times_Y X\rightarrow Y$, the cycle class $\mathcal{Z}$ is rationally equivalent to a cycle $\mathcal{Z}'$ supported in $g^{-1}(Y')$ for a proper subvariety $Y'\subset Y$. (Note that we apply \cite[Theorem 3.1]{voisinchowdecomposition} to the cycle $\mathcal{Z}$ which satisfies the condition (*) preceding the theorem of {\em loc. cit.} simply for the empty subvariety $X'=\emptyset\subset X$.)

    If $U=Y\setminus Y'$ then this proves $\mathcal{Z}|_{g^{-1}(U)}=0$. Now, we proceed with the noetherian induction on $\text{dim}(Y')$ to conclude that $\mathcal{Z}'$ is contained in the image of 
    \[CH_*(Y')\otimes CH_*(F\times F)\cong \bigoplus_{i,j}p_1^*\xi_i\cdot p_2^*\xi_j\cdot CH_*(Y)\]
    in $CH_*(g^{-1}(U))$. So, let us put $\mathcal{Z}'=\sum_{i,j}p_1^*\xi_i\cdot p_2^*\xi_j\cdot g^*\alpha_{i,j}$ then we have:
    \[\Delta_{X/Y}=\sum_i p_1^*\xi_i\cdot p_2^*\zeta_j +\sum_{i,j}p_1^*\xi_i\cdot p_2^*\xi_j\cdot g^*\alpha_{i,j}.\]
    Now we can use the diagonal correspondence $\Delta_{X/Y}$ to readily prove $\phi:CH_*(Y)\otimes CH_*(F)\rightarrow CH_*(X)$ is surjective. Indeed, given $\alpha\in CH_*(X)$, noting $g=f\circ p_1=f\circ p_2$, we have
    \begin{align*}
        \alpha &= (p_1)_*(p_2^*\alpha\cdot \Delta_{X/Y})=\sum_i(p_1)_*(p_2^*\alpha\cdot p_1^*\xi_i\cdot p_2^*\zeta_i)+\sum_{i,j}(p_1)_*\left(p_1^*\xi_i\cdot p_2^*\xi_j\cdot g^*\alpha_{i,j}\cdot p_2^*\alpha\right)\\
        &=\sum_i \xi_i\cdot (p_1)_*p_2^*(\alpha\cdot \zeta_i)+\sum_{i,j}(p_1)_*\left(p_1^*(\xi_i\cdot f^*\alpha_{i,j})\cdot p_2^*(\xi_j\cdot \alpha)\right)\\
        &= \sum_i \xi_i\cdot (p_1)_*p_2^*(\alpha\cdot \zeta_i)+\sum_{i,j}\xi_i\cdot f^*\alpha_{i,j}\cdot (p_1)_*p_2^*(\xi_j\cdot \alpha),
    \end{align*} so by flat base change \cite[Proposition 1.7]{fultonintersection}, we have:
    \[\alpha=\sum_i \xi_i\cdot f^*f_*(\alpha\cdot \zeta_i)+\sum_{i,j}\xi_i\cdot f^*(\alpha_{i,j}\cdot f_*(\xi_j\cdot \alpha)).\]
\end{proof}
While not strictly necessary for us in future, let us remark that the surjective morphism $\phi$ in the above proposition is actually an isomorphism: the injectivity part of the argument in \cite[Theorem 2.1]{motivic_leray_hirsch} adapts to our setting almost verbatim.

We will use this proposition to prove that if $Y$ additionally has the CKgP then so does X:
\begin{theorem}\label[theorem]{thm:ckgp_fibration}
Let $f: X \rightarrow Y$ be a proper morphism between smooth Deligne-Mumford stacks (not necessarily compact). Suppose that:
\begin{enumerate}
    \item $Y$ is unirational and has the Chow--K\"unneth generation Property (CKgP).
    \item Every fiber $F_y$ is smooth, irreducible and has CKgP with rational cohomology (including the Hodge structures) isomorphic to a fixed simplicial toric variety $V$.
    \item There are classes $\xi_1,\ldots,\xi_\ell\in CH^*(X)$ that descend to independent classes in $H^*(X)$, also denoted $\xi_1,\ldots,\xi_\ell$, such that for every fiber $F_y\subset X$, both the cohomology ring $H^*(F_y,\bQ)$ and the Chow ring $CH^*(F_y)$ are $\bQ$-linearly generated by $\xi_1|_{F_y},\ldots,\xi_\ell|_{F_y}$.
\end{enumerate}
Then $X$ also has the CKgP.
\end{theorem}
The second and third hypotheses imply all the direct images $R^qf_*\bQ$ are constant sheaves.
\begin{proof}
First, we suppose $T$ is a smooth Deligne-Mumford stack and prove that the morphism $CH_*(X)\otimes CH_*(T)\rightarrow CH_*(X\times T)$ is surjective. For this we apply \Cref{prop:chow_leray_hirsch} to the fibration $X\times T\rightarrow Y\times T$, which clearly satisfies all the assumptions of the proposition. So there is a surjection
\[\bigoplus_{i=1}^\ell \xi_i\cdot CH_*(Y\times T)\rightarrow CH_*(X\times T).\]
Consequently, by the Chow-K\"unneth generation Property of Y, we obtain a surjective morphism
\[\bigoplus_{i=1}^\ell \xi_i\cdot CH_*(Y)\otimes CH_*(T)\rightarrow CH_*(X\times T).\]
This morphism naturally factors through $CH_*(X)\otimes CH_*(T)$, so the composition
\[\bigoplus_{i=1}^\ell \xi_i\cdot CH_*(Y)\otimes CH_*(T)\rightarrow CH_*(X)\otimes CH_*(T)\rightarrow CH_*(X\times T)\]
is surjective. Consequently, the exterior product morphism $CH_*(X)\otimes CH_*(T)\rightarrow H_*(X\times T)$ is surjective, as we wanted.

Now, we prove the same for possibly singular T. In this case, we take the proper birational desingularization $\tilde{T}\rightarrow T$. Then we have the commutative diagram:
\[
\begin{array}{ccc}
CH_*(X) \otimes CH_*(\tilde{T}) & \longrightarrow & CH_*(X\times \tilde{T}) \\
\Big\downarrow \scriptstyle & & \Big\downarrow \scriptstyle \\
CH_*(X) \otimes CH_*(T) & \longrightarrow & CH_*(X\times T)\ .
\end{array}
\]
The horizontal arrows are exterior product morphisms, whereas the vertical arrows are proper pushforward. Since the top horizontal morphism and the two proper pushforwards are surjective, it follows that $CH_*(X)\otimes CH_*(T)\rightarrow CH_*(X\times T)$ is also surjective. Thus, $X$ has the CKgP.
\end{proof}

\begin{remark}\label[remark]{rem:ckgp_bdry}
  We will apply \Cref{thm:ckgp_fibration} to the exceptional divisors of the morphism $f:\overline\cH_{g,n}(\mu)\eqcolon\overline\cH \rightarrow \overline\cM_{g,n}$, which is birational onto its image. Indeed suppose $D_\Gamma\subset \overline\cH$ is an exceptional divisor then each component of a fiber of $f:D_\Gamma\rightarrow \delta_\Gamma\coloneqq f(D_\Gamma)$ is finitely covered by a simplicial toric variety \cite[Prop. 3.3]{ccm}, which, in particular, implies the fibers all have the CKgP. For simplicity, let us assume the fibers are connected (if not, this is rectified by replacing $f$ by morphism to the normalization of $f(\overline\cH)$) If we restrict $f$ to the open subset $\delta_\Gamma^o\subset\delta_\Gamma$ where the curves do not degenerate further then the fibers of $f^{-1}(\delta_\Gamma^o)\rightarrow \delta_\Gamma^o$ are all irreducible and covered by simplicial toric varieties of uniform combinatorial type. This is because the torus invariant subvarieties all correspond to the degenerations of the level graph $\Gamma$ that keeps the underlying dual graph fixed (see the discussion in \cite[\S 3]{ccm}) (in particular, they are all pulled back from $f^{-1}(\delta_\Gamma^o)$), so the combinatorial type of the toric variety only depends on $\Gamma$. Since the cohomology ring of a simplicial toric variety is generated by the torus invariant subvarieties \cite{cls}, we see that $f:f^{-1}(\delta_\Gamma^o)\rightarrow \delta_\Gamma^o$ satisfies the hypotheses of \Cref{thm:ckgp_fibration} if $\delta_\Gamma^o$ has the CKgP (since the Chow and cohomology rings of each fiber are generated by pullback of fixed collection of cycle classes in $f^{-1}(\delta_\Gamma^o)$, the local systems $R^1f_*\mathbb{Q}$ are all trivial). This allows us to conclude $f^{-1}(\delta_\Gamma^o)$ also has the CKgP. Additionally, by stratifying $D_\Gamma$ so that for each strata the hypotheses of \Cref{thm:ckgp_fibration} are satisfied, we can conclude $D_\Gamma$ itself also has the CKgP.
\end{remark}
\begin{remark}\label[remark]{rem:ckgp_bdry_chow}
  Following the notations from \Cref{rem:ckgp_bdry}, note that Leray-Hirsch type theorem gives the rational cohomology of $f^{-1}(\delta_\Gamma^o)$ as module in terms of the rational cohomology of the base $\delta_\Gamma^o$ and the rational cohomology of a fiber $F$ (since all fibers have isomorphic rational cohomology, the choice of a fiber does not matter): $H^*(f^{-1}(\delta_\Gamma^o))\cong H^*(\delta_\Gamma^o)\otimes H^*(F)$. If $\delta_\Gamma^o$ has the CKgP then the analogous statement holds also for Chow ring. On the other hand, if $D_\Gamma^o\subset f^{-1}(\delta_\Gamma^o)$ is the open subset obtained by further excluding all the degenerations of the level structure of $\Gamma$ (so that $D_\Gamma^o\rightarrow \delta_\Gamma^o$ is now a torus fibration) then, assuming $\delta_\Gamma^o$ has the CKgP, it follows by excision that $CH^*(D_\Gamma^o)\cong CH^*(\delta_\Gamma^o)$.
\end{remark}

\section{Plane quartics and strata of differentials in genus 2} \label{sec:genustwo}
In this section, we will prove that the stratum of meromorphic differentials $\cH(\mu,-2)\subset\cM_{2,n+1}$ in genus 2 has the CKgP. Given an element $(C;p_1,\ldots,p_n,q)\in \cH(\mu,-2)$ such that $\sum_im_ip_i-2q\sim \omega_C$, we look at the divisor $D=\sum_im_ip_i$ on the genus two curve $C$. We have $h^0(C,\cO_C(D))=3$ by the Riemann-Roch theorem, the complete linear series $|D|$ (which is base point free) furnishes a morphism $\phi_D:C\rightarrow \bP^2$. If $D\not\sim 2K_C$ then $\phi_D(C)$ is a quartic curve of geometric genus 2 (so has exactly one singular point). Since $\text{dim}\,|D-2q|=1$, the singular point is then a cusp at $\phi_D(q)$. On the other hand, if $D\sim 2K_C$ then the image of $\phi_D$ is a double conic. In this case, $(C;p_1,\ldots,p_n,q)\in \cH(\mu,-2)$ implies $q$ is a Weierstra{\ss} point. If $\mu=(4)$ then $p_1$ is a Weierstra{\ss} point (so we obtain the hyperelliptic component); if $\mu=(2,2)$, we have either $p_1$ and $p_2$ are hyperelliptic conjugates (which constitutes the hyperelliptic component of $\mathcal{H}(2,2,-2)$) or they are both Weierstra{\ss} points; if $\mu=(3,1)$, we must have $p_1=p_2$ is a Weierstra{\ss} point (so can only occur in the boundary); if $\mu=(2,1,1)$ then $p_1$ is a Weierstra{\ss} point and $p_2,p_3$ are hyperelliptic conjugates; and if $\mu=(1,1,1,1)$ then $p_i$'s give two pairs of hyperelliptic conjugates. Denote by $Z_\mu\subset \cH(\mu,-2)$ this subspace parameterizing $(C;p_1,\ldots,p_n,q)\in \cH(\mu,-2)$ such that $\sum_im_ip_i\sim 2K_C$.

First, we will prove that for each $\mu$, the open subset $\cH(\mu,-2)\setminus Z_\mu$ has the CKgP. For this, we utilize the locus of plane quartics with one cusp and no other singularities (so that the geometric genus becomes 2).
\subsection{The stratum $\cH(4,-2)$}
When $\mu=(4)$, $\cH(4,-2)\setminus Z_\mu=\cH^{nh}(4,-2)$ is the non-hyperelliptic component. Then for $(C;p,q)\in \mathcal{H}^{nh}(4,-2)$, $C_D\coloneqq \phi_D(C)$ is a plane quartic of geometric genus 2 with a cusp (at $\phi_D(q)$) and a hyperflex (at $\phi_D(p)$). Conversely, given a plane quartic $C_D$ of geometric genus 2 with a cusp at $A_1\in \mathbb{P}^2$ and a hyperflex at $A_2\in \mathbb{P}^2$, we can associate $(C;p,q)\in \mathcal{H}^{nh}(4,-2)$, where $C$ is the normalization of $C_D$ and $q,p$ are the preimages of $A_1, A_2$ in the normalization. To make this correspondence formal, considerthe vector space $V\subset \mathbb{C}[x_1,x_2,x_3]_4=H^0(\mathbb{P}^2,\cO_{\mathbb{P}^2}(4))$ given by:
\[V\coloneqq x_1^2\mathbb{C}[x_2,x_3]_2\oplus x_1\mathbb{C}[x_2,x_3]_3\oplus \mathbb{C}x_3^4.\]
As discussed in \cite[Section 2.2]{casnati_rationality}, $V$ consists of plane quartics $C_D$ with a singular point at $A_1\coloneqq [1,0,0]$ and hyperflex at $A_2\coloneqq [0,1,0]$ with the inflectional tangent line $\mathfrak{h}\coloneqq \{x_1=0\}$. The singular point generically is an ordinary node, so the subvariety of interest for us will be the locus in $V$ where the singular point $A_1$ of $C_D$ is a cusp (or worse). Suppose $f\in V$ then in the affine chart $x_1\neq 0$, putting $x=x_2/x_1$ and $y=x_3/x_1$, we have $\frac{\partial f}{\partial x}=\frac{\partial f}{\partial y}=0$ at $(0,0)$. We can further look into the Hessian of $f$:
\[H(f)=\begin{pmatrix}
\frac{\partial^2f}{\partial x^2} & \frac{\partial^2f}{\partial x\partial y} \\
\frac{\partial^2f}{\partial x\partial y} & \frac{\partial^2f}{\partial y^2} 
\end{pmatrix}.\]
If $C_D\coloneqq \{f=0\}$ has cusp then the determinant of $H(f)(0,0)=0$. This condition $\det H(f)(0,0)=0$ defines a quadric hypersurface $\cQ$ in $\mathbb{P}(V)$, and a generic curve in $\cQ$ will have a cuspidal singularity at $[1,0,0]$. We can write down the equation for this quadric explicitly: suppose $f=\sum_{i,j,k}a_{ijk}x_1^ix_2^jx_3^k\in V$. Then on a local coordinate $x_1=1$, $x_2=x$ and $x_3=y$, we have $\frac{\partial^2f}{\partial x^2}(0,0)=2a_{220}$, $\frac{\partial^2f}{\partial y^2}(0,0)=2a_{202}$ and $\frac{\partial^2f}{\partial x \partial y}(0,0)=a_{211}$. Consequently, $\cQ=V(a_{211}^2-4a_{220}a_{202})\subset \mathbb{P}(V)$. In particular, $\cQ$ is a cone over a conic in $\mathbb{P}(V)$ with vertex given by four-dimensional linear subspace $\{a_{211}=a_{220}=a_{202}=0\}$ (which we denote by $\cQ^{\text{sing}}$ as it is also the singular locus of $\cQ$).

Thus, by the discussion before, we obtain a dominant rational map $\cQ\dashrightarrow \cH^{nh}(4,-2)$ whose image is all of $\cH^{nh}(4,-2)$. In fact this gives a quotient stack presentation of $\cH^{nh}(4,-2)$ as $[\cQ^\circ/PG]$, where
\[G=\left\{
  \begin{pmatrix}
a_{11} & 0 & 0 \\
0 & a_{22} & a_{23}\\
0 & 0 & a_{33}
\end{pmatrix}\in GL_3
\right\},\]$\cQ^\circ\subset\cQ$ is the open subset consisting of plane quartics of geometric genus 2 and $PG$ is its image in $PGL_3$ (see \cite[Section 2.2]{casnati_rationality}). We will use this quotient stack presentation to show that $\cH^{nh}(4,-2)$ has CKgP:

\begin{proposition}\label[proposition]{proposition: genus2_mu_4}
  The stratum $\cH^{nh}(4,-2)$ has the CKgP and its rational Chow ring is isomorphic to $\mathbb{Q}$, generated by the fundamental class.
\end{proposition}
\begin{proof}
Denote by $T\subset G\subset GL_3$ be the maximal torus and $PT\subset PG$ be its image in $PG$. Then $[\cQ^\circ/PT]\rightarrow [\cQ^\circ/PG]$ is an $\mathbb{A}^1$-bundle. So by \cite[Lemma 3.5]{cl22}, $[\cQ^\circ/PT]$ has the CKgP if and only if $[\cQ^\circ/PG]$ does. Additionally, the pullback induces isomorphism of the Chow rings of the two stacks $[\cQ^\circ/PT]$ and $[\cQ^\circ/PG]$. So we will now analyze the rational Chow ring of $[\cQ^\circ/PT]$ and prove that it has the CKgP.

Note that $\cQ^\circ$ is an open subset of $\cQ\setminus\cQ^{\text{sing}}$ (since quartics in $\cQ^{\text{sing}}$ has strictly worse singularities), and the complement is a non-trivial divisor (given by discriminant). Since $\cQ\setminus\cQ^{\text{sing}}$ is an $\mathbb{A}^5$-bundle over the conic ($\cong\mathbb{P}^1$) -- this morphism is simply given by mapping $f\in \cQ\setminus\cQ^{\text{sing}}$ to $[a_{220},a_{202},a_{211}]\in \mathbb{P}^2$ -- it follows that $\cQ^\circ$ has the rational Chow ring isomorphic to $\mathbb{Q}$, generated by its fundamental class. Additionally, for any given point $z$ on the conic, the fiber $\cQ_z\cong\mathbb{A}^5$ of this $\mathbb{A}^5$-bundle in $\cQ\setminus\cQ^{\text{sing}}$ is invariant under the action of $PT$ (can be checked by direct computation) and the complement $\cQ\setminus(\cQ^{\text{sing}}\cup \cQ_z)$ is isomorphic to the affine space $\mathbb{A}^6$. Thus $[\cQ^\circ/PT]$ admits a stratification $\cA_1\sqcup \cA_2$, where $\cA_1$ is an open subset of an $\mathbb{A}^6$-bundle over $BPT$ (the classifying space of $PT$) and $\cA_2$ is an open subset of an $\mathbb{A}^5$-bundle over $BPT$. Thus, by the features (i), (iii) and (iv) of CkgP stated in \Cref{subsection:ckgp_intro}, the stack $[\cQ^\circ/PT]$ (and thus also $\cH(4,-2)$) has the CKgP.

To compute the (rational) Chow ring of $[\cQ^\circ/PT]$, we can use Edidin and Graham's equivariant Chow ring \cite{edidingraham_equivariantchow}. Since $\cQ^\circ$ has trivial rational Chow groups (except in the degree equal to its dimension), by \cite[Corollary 2.3]{brion_torusaction}, the $PT$-equivariant rational Chow ring of $\cQ^\circ$ (and thus the rational Chow ring of the quotient) is generated by the pullback classes from $CH^*(BPT)$. Note that $PT\cong (\mathbb{C}^*)^2$, so $CH^*(BPT)\cong\mathbb{Q}[t_1,t_2]$ for divisorial cycle classes $t_1,t_2$. We will now show that pullbacks of both these divisorial classes to $[\cQ^\circ/PT]$ are trivial. For this, we look at the equivariant closed embedding $\cQ\subset \mathbb{P}(V)$. If $\mathbb{P}(V)^\circ\subset\mathbb{P}(V)$ is the locus of plane quartics of geometric genus 2 then we have the closed embedding of stacks $[\cQ^\circ/PG]\hookrightarrow[\mathbb{P}(V)^\circ/PG]$. By the discussion above and that in \cite{casnati_rationality}, the latter is exactly $\mathcal{M}_{2,1}\setminus W_{2,1}$ (where $W_{2,1}$ is the locus of genus 2 curve with a Weierstra{\ss} point) and the embedding is exactly $\cH^{nh}(4,-2)\subset \mathcal{M}_{2,1}\setminus W_{2,1}$ obtained by forgetting the double pole. Note that the map $\mathbb{P}(V)\dashrightarrow\mathcal{M}_{2,1}$ factors through the stratum $\cH(4,-1,-1)/S_2$ where the preimage of the singular point under normalization of the plane quartic gives the symmetrized double poles of the genus 2 curve. It is then clear that the morphism $\cH(4,-1,-1)/S_2\rightarrow \mathcal{M}_{2,1}$, given by forgetting the simple poles, is of degree 1, so since $\cH^{nh}(4,-2)$ manifests in the boundary of $\cH(4,-1,-1)/S_2$, the induced morphism is also of degree 1.

As an upshot, we see that the morphism $[\cQ^\circ/PT]\rightarrow BPT$ factors through $[\mathbb{P}(V)^\circ/PT]$. Since the rational Picard group of $\mathcal{M}_{2,1}$ is generated by the divisor class $[W_{2,1}]$, we see that the rational Picard group of $[\mathbb{P}(V)^\circ/PT]$ is trivial. So the classes $t_1,t_2\in CH^*(BPT)\otimes\mathbb{Q}\cong\mathbb{Q}[t_1,t_2]$ must pullback to 0 in the rational Chow ring of $[\cQ^\circ/PT]$. This completes the proof of the proposition.
\end{proof}
\subsection{The strata $\cH(2,2,-2)$ and $\cH(3,1,-2)$}
\begin{enumerate}
\item When $\mu=(2,2)$, $\mathcal{H}(\mu,-2)\setminus Z_\mu$ is contained in $\mathcal{H}^{nh}(\mu,-2)$. Then for $(C;p_1,p_2,q)\in \mathcal{H}(\mu,-2)\setminus Z_\mu$, $C_D\coloneqq \phi_D(C)$ is a plane quartic of geometric genus 2 with a cusp at $\phi_D(q)$ and a bitangent line $\ell$ with tangency along $\phi_D(p_1)$ and $\phi_D(p_2)$. Conversely, given a plane quartic $C_D$ of geometric genus 2 with a cusp at $A_1\in \mathbb{P}^2$ and a bitangent line $\ell\subset \mathbb{P}^2$ with tangency at $A_2$ and $A_3$ then we can associate $(C;p_1,p_2,q)\in \mathcal{H}(\mu,-2)\setminus Z_\mu$, where $C$ is the normalization of $C_D$ and $p_1,p_2,q$ are preimages of $A_2,A_3,A_1$, respectively, in $C$. To make this correspondence precise, fix $A_1=[1,0,0]$, $A_2=[0,1,0]$ and $A_3=[0,0,1]$ as well as the line $\ell=\{x_1=0\}$ (where $x_1,x_2,x_3$ are the homogeneous coordinates in $\mathbb{P}^2$). Then $V\subset \mathbb{C}[x_1,x_2,x_3]_4$ consisting of quartics $C_D$ with a singular point at $A_1$ and a bitangent line $\ell$ with tangency at $A_2$ and $A_3$ is a vector subspace. In fact:
  \[V\coloneqq x_1^2\mathbb{C}[x_2,x_3]_2\oplus x_1\mathbb{C}[x_2,x_3]_3\oplus \mathbb{C}x_2^2x_3^2.\]
  Then as in the previous case, imposing the singular point $A_1$ to be a cusp defines a quadric hypersurface $\mathcal{Q}\subset \mathbb{P}(V)$ which admits a dominant rational map $\mathcal{Q}\dashrightarrow \mathcal{H}^{nh}(\mu,-2)$ whose image is $\mathcal{H}(\mu,-2)\setminus Z_\mu$. This then gives a quotient stack presentation of $\mathcal{H}(\mu,-2)\setminus Z_\mu$ as $[\mathcal{Q}^\circ/PG]$ where \[G=\left\{
  \begin{pmatrix}
a_{11} & 0 & 0 \\
0 & a_{22} & 0\\
0 & 0 & a_{33}
\end{pmatrix}\in GL_3
\right\},\]$\cQ^\circ\subset\cQ$ is the open subset consisting of plane quartics of geometric genus 2 and $PG$ is its image in $PGL_3$. Then retreading the argument in the proof of \Cref{proposition: genus2_mu_4} in this context proves that $\cH(2,2,-2)\setminus Z_\mu$ has the CKgP. That its Chow ring is isomorphic to $\mathbb{Q}$ generated by the fundamental class also follows via a similar argument -- forgetting the double pole injects $\cH(\mu,-2)\setminus Z_\mu$ into $\mathcal{M}_{2,2}\setminus W_{2,2}$, where $W_{2,2}\subset \mathcal{M}_{2,2}$ is the locus of 2-pointed genus 2 curves such that one of the marked points is a Weierstra{\ss} point. And as in \Cref{proposition: genus2_mu_4}, the Chow ring of $\cH(\mu,-2)\setminus Z_\mu$ is generated by the classes pulled back from the classifying stack of $PG\cong (\mathbb{C}^*)^2$. This pullback factors through $\mathcal{M}_{2,2}\setminus W_{2,2}$, whose Chow ring is tautological, generated by $\psi$-classes. But the $\psi$-classes are proportional to $\lambda$ (and thus trivial) in $\cH(2,2,-2)\setminus Z_\mu$. Consequently, its Chow ring is isomorphic to $\mathbb{Q}$, generated by the fundamental class.

Now, to prove that $\cH^{nh}(\mu,-2)$ has the CKgP, it is enough to prove that $Z_\mu^{nh}\coloneqq Z_\mu\cap \cH^{nh}(\mu,-2)$, which consists of $(C;p_1,p_2,q)\in \cH^{nh}(\mu,-2)$ such that $p_1,p_2,q$ are all Weierstra{\ss} points on $C$, also has the CKgP. But this is clear since $M_{0,6}$ surjects onto $Z_\mu^{nh}$ via the map $(\mathbb{P}^1;z_1,\ldots,z_6)\mapsto (C;p_1,p_2,q)$ where $C$ is the $2:1$ cover of $\mathbb{P}^1$ branched along $z_1,\ldots,z_6$ and $p_1,p_2,q$ are preimages of $z_4,z_5,z_6$, respectively.

Finally, we argue that fundamental class of $Z_\mu^{nh}$ is trivial in the Chow ring of $\cH^{nh}(\mu,-2)$, which would then imply the Chow ring of $\cH^{nh}(\mu,-2)$ is also generated by its fundamental class by excision. But this is clear since $Z_\mu^{nh}$ is set-theoretically the intersection of $W_{2,2}$ with $\cH(\mu,-2)$ in $\mathcal{M}_{2,2}$, so $[Z_\mu^{nh}]$ is proportional to the pullback of $[W_{2,2}]$ from $\mathcal{M}_{2,2}$ to $\cH(\mu,-2)$. In particular, $[Z_\mu^{nh}]$ is tautological, and thus trivial in $\cH(\mu,-2)$.
\item When $\mu=(3,1)$, $Z_\mu$ is entirely contained in the boundary of $\overline\cH(\mu,-2)$. Then we can follow the similar argument as in the previous two cases to get a correspondence between $(C;p_1,p_2,q)\in \cH(\mu,-2)$ and a plane curve $C_D\subset\mathbb{P}^2$ with a double point singularity at $A_1=[1,0,0]$ and inflectional tangent line $\ell=\{x_1=0\}$ with flex point $A_2=[0,1,0]$ such that $\ell$ also intersects $C_D$ at $A_3=[0,0,1]$. Consequently, the vector space $V\subset \mathbb{C}[x_1,x_2,x_3]_4$ in this case will be given by:
  \[V\coloneqq x_1^2\mathbb{C}[x_2,x_3]_2\oplus x_1\mathbb{C}[x_2,x_3]_3\oplus \mathbb{C}x_2x_3^3,\]
  see the discussion in \cite[Section 2.3]{casnati_rationality}. Thus as before, we get a quotient stack presentation $\cH(3,1,-2)\cong [\cQ^\circ/PG]$ for an open subset $\cQ^\circ$ of a quadric hypersurface $\cQ\subset \mathbb{P}(V)$, where $G$ is again the group of invertible diagonal matrices in $GL_3$ as in previous case. And thus, following the argument of $\mu=(2,2)$ case, we can again conclude that $\cH(3,1,-2)$ has the CKgP and its Chow ring is generated by the fundamental class.
\end{enumerate}
Thus, we have proven:
\begin{proposition}
  The strata of abelian differentials $\cH(2,2,-2)$ and $\cH(3,1,-2)$ both have the CKgP and both their Chow rings are generated by their fundamental classes.
\end{proposition}

\begin{remark}
  Additionally, we will also need that $\cH(\mu)$ has the CKgP for $\mu=(2,0)$ or $(1,1,0)$ or $(2,0,0)$. This is well-known, see e.g. \cite{canning_larson_hyperelliptic}. Nonetheless, we will outline a proof showing these strata also have the CKgP. The first two cases are simple. Indeed, $\cH(2,0)$ is, on the level of coarse moduli space, finitely dominated by $M_{0,7}$ with the image the open subset of $\cH(2,0)$ comprising the complement of the locus where both marked points are Weierstra\ss. On the other hand, forgetting a simple zero defines an open immersion $\cH(1,1,0)\rightarrow M_{2,2}$ of coarse moduli spaces.

  For $\mu=(2,0,0)$, we can give a rational parameterization $\cH(2,0,0)$ using the plane curve model $y^2=f(x)$, where $f(x)=x^5+a_4x^4+\ldots+a_0$ is a polynomial of degree five with distinct zeros. Suppose $(C;w,p_1,p_2)\in \cH(2,0,0)$ then $w$ will correspond to the unique point at infinity in the plane curve. 
  On the locus where $p_2\neq\iota(p_1)$, with $\iota$ the hyperelliptic involution, the distinct marked points have distinct images under the hyperelliptic map. We may take $p_1=(0,y_1)$ and $p_2=(1,y_2)$ for $y_1,y_2\in \mathbb{C}$. So $a_0=y_1^2$ and $a_1=y_2^2-y_1^2-1-a_2-a_3-a_4$, so that $(y_1,y_2,a_2,a_3,a_4)\in \mathbb{C}^5$ defines a rational parameterization of $\cH(2,0,0)\setminus\{(C;w,p_1,p_2)\mid p_2=\iota(p_1)\}$. On the other hand, On the other hand, if $p_2=\iota(p_1)$, forgetting $p_2$ identifies this locus with the open subset of $\cH(2,0)$.
\end{remark}

\section{Plane cubics and strata of differentials in genus 1} \label{sec:genusone}
In this section, analyze the Chow ring of certain strata of differentials in genus 1. First, we review the rational parameterization of $M_{1,4}$ constructed by Belorousski in his PhD thesis \cite{belorousski}.

Suppose $Q_1,Q_2,Q_3,Q_4$ are four points in $\mathbb{P}^2$ in linearly general position (say $Q_1=[1,0,0],\ Q_2=[0,1,0],\ Q_3=[0,0,1]$ and $Q_4=[1,1,1]$) and $L_{ij}$ is the line joining $Q_i$ and $Q_j$ in $\mathbb{P}^2$. Then we define $\Phi_4\subset \mathbb{P}^9\cong \mathbb{P}H^0(\mathbb{P}^2,\mathcal{O}_{\mathbb{P}^2}(3))$ to be the collection of cubics in $\mathbb{P}^2$ passing through $Q_1,\ldots,Q_4$ and tangent to $L_{12}$ at $Q_1$. So $\Phi_4\cong\mathbb{P}^4$ and the dense open subset $U_4\subset\Phi_4$ parameterizing smooth cubics admits a dominant morphism $f_4:U_4\rightarrow\mathcal{M}_{1,4}$. Belorousski observes that $f_4$ is injective and $f_4(U_4)\subset\mathcal{M}_{1,4}$ is the complement of $D_1\cup D_2$ in $\mathcal{M}_{1,4}$, where:
\begin{enumerate}
\item $D_1=\{p_1+p_2\sim p_3+p_4\}\subset\mathcal{M}_{1,4}$. That is, $D_1\cong \mathcal{H}_{1,4}(1,1,-1,-1)$.
\item $D_2=\{2p_1\sim p_3+p_4\}\subset\mathcal{M}_{1,4}$. That is, $D_2\cong \mathcal{H}_{1,4}(2,0,-1,-1)$.
\end{enumerate}
The inverse morphism $\mathcal{M}_{1,4}\setminus (D_1\cup D_2)\rightarrow U_4$ is given by the embedding $\varphi:C\rightarrow \mathbb{P}^2$ given by the complete linear series $|2p_1+p_2|$, where $(C;p_1,p_2,p_3,p_4)\in \mathcal{M}_{1,4}\setminus (D_1\cup D_2)$ and coordinates in $\mathbb{P}^2$ are chosen such that $\varphi(p_i) = Q_i$.
\subsection{The stratum $\mathcal{H}(2,0,0,-2)\subset \mathcal{M}_{1,4}$}
Using the Belorousski's parameterization of $\mathcal{M}_{1,4}$, let us analyze the stratum $\mathcal{H}(2,0,0,-2)$. For brevity, let us denote $\mathcal{H}\coloneqq \mathcal{H}(2,0,0,-2)$ throughout this subsection. If $(C;p_1,\ldots,p_4)\in \cH$ then $2p_1+p_2\sim 2p_4+p_2$ in $\mathrm{Jac}(C)$, so under the embedding $\varphi:C\hookrightarrow\mathbb{P}^2$, we have $C$ is tangent to $L_{2,4}$ at $Q_4$ (and conversely -- a cubic passing through $Q_1,\ldots,Q_4$ that is tangent to $L_{1,2}$ at $Q_1$ and $L_{2,4}$ at $Q_4$ gives a point in $\mathcal{H}$). This is further a linear condition on $\mathbb{P}^4$, so we get an identification between an open subset of $\mathbb{P}^3$ and $\mathcal{H}\setminus (D_1\cup D_2)$. Since the open subset of $\mathbb{P}^3$ is contained in the complement of the discriminant hypersurface, $\mathcal{H}\setminus (D_1\cup D_2)$ has the Chow ring isomorphic to $\mathbb{Q}$, generated by the fundamental class.

Next we look at $\mathcal{H}\cap D_1$ and $\mathcal{H}\cap D_2$ and examine their Chow rings. First note that $(C;p_1,p_2,p_3,p_4)\in \mathcal{H}\cap D_2$ implies $2p_1\sim 2p_4\sim p_3+p_4$, which implies $p_4\sim p_3$. But this is impossible unless $p_3=p_4$, so $\mathcal{H}\cap D_2=\emptyset$. On the other hand, $D_1\cong \mathcal{H}(1,1,-1,-1)$, so forgetting $p_3$ injects $\mathcal{H}\cap D_1$ into $\mathcal{M}_{1,3}$ and the image is exactly the stratum $\mathcal{H}(2,0,-2)$, so $\cH\cap D_1\cong \mathcal{H}(2,0,-2)$. We claim that (the coarse moduli space of) $\mathcal{H}(2,0,-2)$ is a rational surface. Indeed forgetting the order 0 marked point gives a fibration over the modular curve $\mathcal{H}(2,-2)$, which is isomorphic to an open subset of $\mathbb{P}^1$. Since a generic element $(E;p,q)\in \mathcal{H}(2,-2)$ has an involution (associated to the 2-to-1 morphism $E\rightarrow \mathbb{P}^1$ given by the pencil $|2p|=|2q|$), the fiber over $(E;p,q)$ of this fibration $\mathcal{H}(2,0,-2)\rightarrow\mathcal{H}(2,-2)$ is again a rational curve (quotient of the elliptic curve $E\setminus\{p,q\}$ by the involution). In particular, this implies $\cH\cap D_1$ has the CKgP. Furthermore if $\mathcal{H}(2,0,-2)\subset \mathcal{H}^{rt}(2,0,-2)$ is a partial compactification by rational tails then $\mathcal{H}^{rt}(2,0,-2)\rightarrow \mathcal{H}(2,-2)$ is a proper $\mathbb{P}^1$-fibration, so the rational Chow ring is generated by $\mathcal{O}(1)$ class of this projective bundle (recall that $\mathcal{H}(2,-2)$, being a strict open subset of $\mathbb{P}^1$ as a coarse moduli space, has the Chow ring generated by its fundamental class), so $\mathcal{H}(2,0,-2)\cong \cH\cap D_1$ has the rational Chow ring generated by its fundamental class.

Observe that the fundamental class of $\cH\cap D_1$ pushes forward to 0 in the rational Chow ring of $\mathcal{H}(2,0,0,-2)$. Indeed if $\iota$ is the inclusion of $\mathcal{H}(2,0,0,-2)$ in $\mathcal{M}_{1,4}$ then the the divisor class of $\cH\cap D_1$ is a multiple of the class $\iota^*[\mathcal{H}(1,1,-1,-1)]$. Since the rational Picard group of $\mathcal{\overline{M}}_{1,n}$ is generated by the boundary divisors (see \cite[Ch. 19 Thm. 4.1]{acg2}), it follows that $[\mathcal{H}(1,1,-1,-1)]=0$ in $\text{Pic}_{\mathbb{Q}}(\mathcal{M}_{1,4})$. Consequently we have proven:
\begin{proposition} \label{prop:h200-2}
  The rational Chow ring of $\mathcal{H}(2,0,0,-2)$ is isomorphic to $\mathbb{Q}$, generated by the fundamental class and $\mathcal{H}(2,0,0,-2)$ has the CKgP.
\end{proposition}
\subsection{The stratum $\mathcal{H}(2,2,-2,-2)\subset\mathcal{M}_{1,4}$}
Note that $\mathcal{H}(2,2,-2,-2)$ has two components -- one is given by squaring the differentials in $\mathcal{H}(1,1,-1,-1)$ and the other is primitive. The former is isomorphic to an open subset of $\mathcal{M}_{1,3}$, so it has the CKgP and its rational Chow ring is generated by its fundamental class. So for brevity, we abuse the notation to denote by $\mathcal{H}(2,2,-2,-2)$ the primitive component. We then claim that there is a birational map $\mathcal{H}(2,2,-2,-2)\dashrightarrow \mathcal{H}(2,0,0,-2)$. Indeed given general $(E;p_1,p_2,p_3,p_4)\in \mathcal{H}(2,2,-2,-2)$, we can map it to $(E;p_0,p_2,p_3,p_4)\in \mathcal{H}(2,0,0,-2)$, where $p_0\sim p_1+p_2-p_3$. Note that this map is well-defined for general $(E;p_1,p_2,p_3,p_4)$ since $h^0(E,\mathcal{O}_E(p_1+p_2-p_3)=1$, so there is the unique point $p_0\in E$ such that $p_0\sim p_1+p_2-p_3$ (this also shows that the map is one-to-one). This map is not well defined exactly when $p_0=p_i$ for $i\in \{2,3,4\}$, that is, if $p_1\sim p_3$ or $p_1+p_2\sim 2p_3$ or $p_1+p_2\sim p_3+p_4$. Of course $p_1\sim p_3$ is impossible and so is $p_1+p_2\sim p_3+p_4$ by our choice of primitive component of the stratum. Consequently we have an open embedding of $\mathcal{H}(2,2,-2,-2)\setminus \mathcal{H}(1,1,-2,0)$ into $\mathcal{H}(2,0,0,-2)$. Consequently $\mathcal{H}(2,2,-2,-2)\setminus \mathcal{H}(1,1,-2,0)$ has the CKgP and its Chow ring is isomorphic to $\mathbb{Q}$ generated by the fundamental class.

On the other hand, forgetting the first marked point gives an inclusion of \[D_0 :=
\left\{
(E;p_1,p_2,p_3,p_4)\in\mathcal{H}(2,2,-2,-2)
\;\middle|\; p_1+p_2\sim 2p_3\ \text{and}\ 2p_1+2p_2\sim 2p_3+2p_4
\right\}.\] into $\mathcal{M}_{1,3}$ and the image is exactly $\mathcal{H}(0,2,-2)$ (indeed $2p_1+2p_2\sim 2p_3+2p_4$ and $p_1+p_2\sim 2p_3$ in conjunction imply $2p_3\sim 2p_4$). We have already seen that $\cH(0,2,-2)\cong \mathcal{H}(2,0,-2)$ has the CKgP and has Chow ring isomorphic to $\mathbb{Q}$ generated by the fundamental class. On the other hand, by the similar argument as in the previous subsection, we can see that the fundamental class of $D_0$ has trivial pushforward in the rational Picard group of $\mathcal{H}(2,2,-2,-2)$. Thus we have proven:
\begin{proposition}
  The rational Chow ring of $\mathcal{H}(2,2,-2,-2)$ is isomorphic to $\mathbb{Q}$, generated by the fundamental class, and $\mathcal{H}(2,2,-2,-2)$ has the CKgP.
\end{proposition}
\begin{remark}
  Note that analogously the morphism $(E;p_1,p_2,p_3)\mapsto (E;p_1,p_0,p_3)$ for $p_0\sim 2p_1-p_2$ gives an identification of the primitive component of $\mathcal{H}(4,-2,-2)$ with $\mathcal{H}(0,2,-2)$, so analogously we also have that each component of $\mathcal{H}(4,-2,-2)$ has the CKgP, with the Chow ring generated by its fundamental class.
\end{remark}
\begin{remark}
    Note that in genus 1, we have the "inversion" morphism $fdz\mapsto \frac{1}{f}dz$ which identifies $\cH(\mu)$ with $\cH(-\mu)$. So showing one has the CKgP is enough to conclude the same for the other. However, this morphism does not extend to the boundary in general.
\end{remark}

\section{Minimal stratum $\overline{\mathcal{H}}(4)$} \label{sec:fs}
In this section, we will prove that the multi-scale compactification $\overline\cH(4)$ of the minimal stratum has the CKgP. We will first prove that each (generalized) stratum of (possibly meromorphic) differentials that appears in the boundary of $\overline\cH(4)$ has the CKgP.

\subsection{Enumerating the boundary divisors}
The boundary divisors of $\overline\cH(4)$ are  either horizontal boundary $D_h$ or given by one of the following (enhanced) level graphs.

\begin{figure}[H]
  \centering
  \begin{subfigure}{0.2\textwidth}
    \centering
\begin{tikzpicture}[main/.style = {draw, circle}]
    \node[main] at (0,0) {2};
    %\fill[black] (0,0) circle (3pt);
    \fill[black] (0, -1.5) circle (3pt);

    \draw[thick] (0,-.35) to [out=225, in=135] (0, -1.5);
    \draw[thick] (0,-.35) to [out=315, in=45] (0, -1.5);
    \draw (0,-1.5) -- (0,-2);

    \draw node at  (0,-2.2) {$4$};
    \draw node at  (-0.4,-1) {$3$};
    \draw node at  (0.4,-1) {$1$};
  \end{tikzpicture}
  \caption{$\Gamma_1$}
\end{subfigure}
\begin{subfigure}{0.2\textwidth}
    \centering
\begin{tikzpicture}[main/.style = {draw, circle}]
    \node[main] at (0,0) {2};
    %\fill[black] (0,0) circle (3pt);
    \fill[black] (0, -1.5) circle (3pt);

    \draw[thick] (0,-.35) to [out=225, in=135] (0, -1.5);
    \draw[thick] (0,-.35) to [out=315, in=45] (0, -1.5);
    \draw (0,-1.5) -- (0,-2);

    \draw node at  (0,-2.2) {$4$};
    \draw node at  (-0.4,-1) {$2$};
    \draw node at  (0.4,-1) {$2$};
  \end{tikzpicture}
  \caption{$\Gamma_2$}
\end{subfigure}
\begin{subfigure}{0.2\textwidth}
    \centering
\begin{tikzpicture}[main/.style = {draw, circle}]
    \node[main] at (0,0) {1};
    %\fill[black] (0,0) circle (3pt);
    \node[main] at (0,-1.85) {1};

    \draw[thick] (0,-.35) to [out=225, in=135] (0, -1.5);
    \draw[thick] (0,-.35) to [out=315, in=45] (0, -1.5);
    \draw (0,-2.2) -- (0,-2.7);

    \draw node at  (0,-2.9) {$4$};
    \draw node at  (-0.4,-1) {$1$};
    \draw node at  (0.4,-1) {$1$};
  \end{tikzpicture}
  \caption{$\Gamma_3$}
\end{subfigure}
\begin{subfigure}{0.2\textwidth}
    \centering
\begin{tikzpicture}[main/.style = {draw, circle}]
    \node[main] at (0,0) {1};
    %\fill[black] (0,0) circle (3pt);
    \fill[black] (0, -1.5) circle (3pt);

    \draw[thick] (0,-.35) to [out=210, in=150] (0, -1.5);
    \draw[thick] (0,-.35) to [out=330, in=30] (0, -1.5);
    \draw[thick] (0,-.35) to [out=270, in=90] (0, -1.5);
    \draw (0,-1.5) -- (0,-2);

    \draw node at  (0,-2.2) {$4$};
    \draw node at  (-0.5,-1) {$1$};
    \draw node at  (-0.1,-1) {$1$};
    \draw node at  (0.5,-1) {$1$};
  \end{tikzpicture}
  \caption{$\Gamma_4$}
\end{subfigure}
\begin{subfigure}{0.2\textwidth}
    \centering
\begin{tikzpicture}[main/.style = {draw, circle}]
    \node[main] at (0,0) {2};
    %\fill[black] (0,0) circle (3pt);
    \node[main] at (0,-1.85) {1};

    \draw[thick] (0,-.35) to [out=270, in=90] (0, -1.5);
    \draw (0,-2.2) -- (0,-2.7);

    \draw node at  (0,-2.9) {$4$};
    \draw node at  (0.1,-1) {$3$};
  \end{tikzpicture}
  \caption{$\Gamma_5$}
\end{subfigure}
\begin{subfigure}{0.2\textwidth}
    \centering
\begin{tikzpicture}[main/.style = {draw, circle}]
    \node[main] at (0,0) {1};
    %\fill[black] (0,0) circle (3pt);
    \node[main] at (0,-1.85) {2};

    \draw[thick] (0,-.35) to [out=270, in=90] (0, -1.5);
    \draw (0,-2.2) -- (0,-2.7);

    \draw node at  (0,-2.9) {$4$};
    \draw node at  (0.1,-1) {$1$};
  \end{tikzpicture}
  \caption{$\Gamma_6$}
\end{subfigure}
\begin{subfigure}{0.2\textwidth}
    \centering
\begin{tikzpicture}[main/.style = {draw, circle}]
    \node[main] at (0,0) {1};
    %\fill[black] (0,0) circle (3pt);
    \node[main] at (-.8,1.85) {1};
    \node[main] at (.8,1.85) {1};

    \draw[thick] (0,.35) to [out=125, in=-55] (-.8, 1.5);
    \draw[thick] (0,.35) to [out=55, in=-125] (.8, 1.5);
    \draw (0,-.35) -- (0,-.7);

    \draw node at  (0,-.9) {$4$};
    \draw node at  (-0.6,1) {$1$};
    \draw node at  (0.6,1) {$1$};
  \end{tikzpicture}
  \caption{$\Gamma_7$}
\end{subfigure}
\begin{subfigure}{0.2\textwidth}
    \centering
\begin{tikzpicture}[main/.style = {draw, circle}]
    %\node[main] at (0,0) {1};
    \fill[black] (0,0) circle (3pt);
    \node[main] at (-.8,1.85) {1};
    \node[main] at (.8,1.85) {1};

    \draw[thick] (0,0) to [out=120, in=-60] (-.8, 1.5);
    \draw[thick] (0,0) to [out=70, in=200] (.8, 1.5);
    \draw[thick] (0,0) to [out=30, in=-90] (.8, 1.5);
    \draw (0,0) -- (0,-.4);

    \draw node at  (0,-.6) {$4$};
    \draw node at  (-0.6,0.9) {$1$};
    \draw node at  (0.8,0.8) {$1$};
    \draw node at  (0.2,1) {$1$};
  \end{tikzpicture}
  \caption{$\Gamma_8$}
\end{subfigure}
\caption{Vertical boundary divisors in $\overline{\cH}(4)$}
\label{fig:vertical_graphs_mu_(4)}
\end{figure}
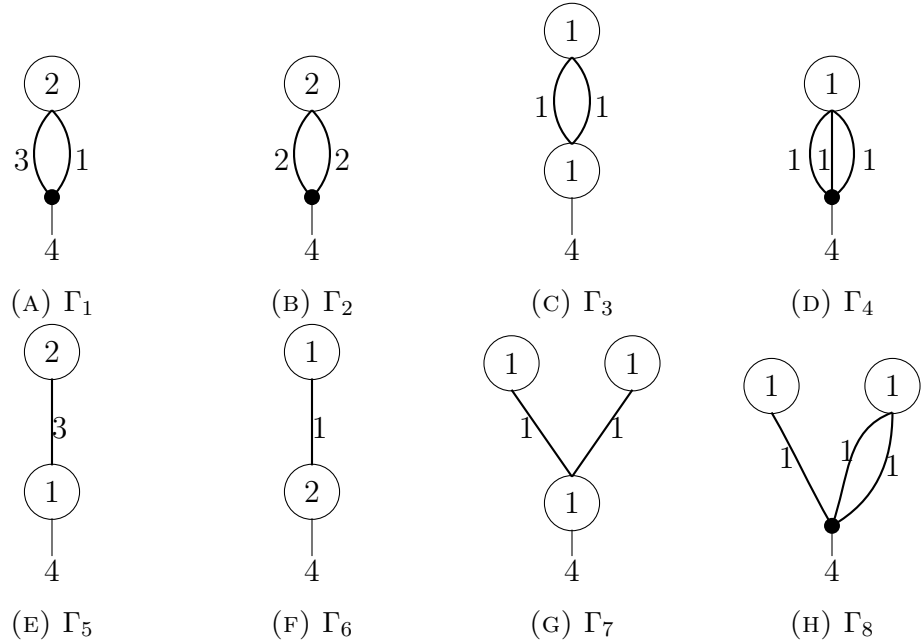% Preamble:

\tikzset{
  vgraph/.style={
    x=1cm,
    y=1cm,
    every node/.style={font=\small},
    thick
  },
  component/.style={
    draw,
    circle,
    fill=white,
    minimum size=7mm,
    inner sep=0pt,
    font=\small
  },
  blacknode/.style={
    circle,
    fill=black,
    inner sep=1.8pt
  },
  edgelabel/.style={
    font=\scriptsize,
    fill=white,
    inner sep=1pt
  },
  leglabel/.style={
    font=\scriptsize
  },
  edge/.style={
    thick,
    line cap=round
  }
}
The level-wise product model for (the interior of) each of these boundary divisors can be enumerated as follows. We refer to \cite{boissymero} for the discussion of the connected components of the strata of meromorphic differentials, and to \cite{lee,leewong} for that of the strata of meromorphic differentials with linear residue conditions that appear in the boundary stratification below.
\begin{enumerate}
\item $D_h^o\sim \cH(4,-1,-1)/S_2$ is the horizontal divisor. It has two components, one in each of $\overline\cH^{h}(4)$ and $\overline\cH^{nh}(4)$.
\item $D_{\Gamma_1}^o\sim \cH(2,0)\times \cH(4,-4,-2)$. The prong-matchings on each compatible collection of twisted differentials belong to the unique equivalence class, so $D_{\Gamma_1}$ has exactly one component contained in $\overline\cH^{nh}(4)$.
\item $D_{\Gamma_2}^o$ is isomorphic to two disjoint copies of $(\cH(1,1)\times \cH(4,-3,-3))/S_2$ (where $S_2\cong\text{Aut}\,(\Gamma_2)$). Each compatible collection of twisted differentials has two equivalence classs of prong-matchings, which leads to $D_{\Gamma_2}$ having two components, one in each of $\overline\cH^h(4)$ and $\overline\cH^{nh}(4)$.
\item $D_{\Gamma_3}^o\sim (\cH(0,0)\times\cH(4,-2,-2))/S_2$ (where $S_2\cong\text{Aut}\,(\Gamma_3)$). The stratum $\cH(4,-2,-2)$ has two components distinguished by the rotation number. Consequently, $D_{\Gamma_3}$ has two components, one in each of $\overline\cH^h(4)$ and $\overline\cH^{nh}(4)$. Note that $\cH(0^n)\cong \cM_{1,n}$, which has the CKgP for $n\le 10$ \cite[\S 5.2]{cl22}.
\item $D_{\Gamma_4}^o\sim (\cH(0,0,0)\times \cH(4,-2,-2,-2))/S_3$ (where $S_3\cong\text{Aut}\,(\Gamma_4)$). This boundary divisor is irreducible contained in $\overline\cH^{nh}(4)$.
\item $D_{\Gamma_5}^o\sim \cH(2)\times \cH(4,-4)$. The stratum $\cH(4,-4)$ has two components distinguished by the rotation number. Consequently, $D_{\Gamma_5}$ has two components, one in each of $\overline\cH^{h}(4)$ and $\overline\cH^{nh}(4)$.
\item $D_{\Gamma_6}^o\sim \cH(0)\times \cH(4,-2)$. The stratum $\cH(4,-2)$ has two components distinguished by hyperellipticity. Consequently $D_{\Gamma_6}$ has to components, one in each of $\overline\cH^{h}(4)$ and $\overline\cH^{nh}(4)$.
\item $D_{\Gamma_7}^o\sim (\mathbb{P}(\Omega\cM_{1,1}\times\Omega\cM_{1,1})\times \cR(4,-2,-2))/S_2$ (where $S_2\cong\text{Aut}\,(\Gamma_7)$). The residueless stratum $\cR(4,-2,-2)$ has two components distinguished by ramification profile. Consequently, $D_{\Gamma_7}$ has two components, one in each of $\overline\cH^{h}(4)$ and $\overline\cH^{nh}(4)$.
\item $D_{\Gamma_8}^o\sim \mathbb{P}(\Omega\cM_{1,1}\times\Omega\cM_{1,2})\times \cH(4,|-2|-2,-2|)/S_2$ (where $S_2\cong\text{Aut}\,(\Gamma_8)$). This divisor is irreducible, contained in $\overline\cH^{nh}(4)$. Note that $\cH(4,|-2\,|\,-2,-2|)$ is 0-dimensional.
\end{enumerate}

\subsection*{The residueless locus $\cR(4,-2,-2)$}
The residueless locus $\PP\cR(4,-2,-2)$ also has two connected components distinguished by the rotation number. Both components are hyperelliptic. The Euler characteristic of $\PP\overline{\cR}(4,-2,-2)$ is $4$ by \cite{LT} (see also \cite[Prop. 2.1]{devkota_grushevsky}). That is, both connected components are rational.\\

With this, we see that each vertex appearing in all the enhanced level graphs in \Cref{fig:vertical_graphs_mu_(4)} represents a moduli space of differentials that has the CKgP. This proves $D_{\Gamma_i}^o$ has the CKgP, except for $i=7,8$ -- for these two cases, we will give an argument based on \Cref{thm:ckgp_fibration}.

\begin{remark}\label[remark]{rem:codimension_up_to_two}
  Note that while we have depicted all the lower genus strata of differentials that manifest as a vertex of a divisorial enhanced level graphs, it is possible that a higher codimension enhanced level graph manifests additional cases. However, we don't need to go beyond enhanced level graphs with three levels. Indeed, suppose $\Lambda$ is an enhanced level graph and $v$ is a vertex on $\Lambda$ representing the stratum $\cH_v$. Then the new enhanced level graph $\Gamma$ obtained by consolidating all levels above $i$ to one and all levels below $i$ to one will have exactly three levels, and still contains the vertex $v$ representing the stratum $\cH_v$. In \Cref{appendix}, we have listed all the enhanced level graphs with three levels in Figure~\ref{fig:h4}. Looking at the list, it is clear that all vertices represent (open) stratum of differentials that has the CKgP.
\end{remark}

\subsection{The vertical boundary}
We observed in the previous subsection that for each divisorial level graph $\Gamma_i$, unless $i=2$, each compatible twisted differential has unique prong-matching equivalence class. And for $i=2$, the two prong-matching equivalence classes lead to two distinct components of the boundary stratum, one hyperelliptic and the other hyperelliptic. Thus, as discussed in \Cref{sec:boundary_correspondence}, there is a finite covering $(B_{\Gamma_i}^{[0]})^o\times (B_{\Gamma_i}^{[-1]})^o\rightarrow D_{\Gamma_i}^o$, which proves $D_{\Gamma_i}^o$ has the CKgP for $1\le i\le 6$. The remaining two divisors, $D_{\Gamma_7}$ and $D_{\Gamma_8}$, are exceptional divisors for the morphism $\overline\cH\rightarrow \overline\cM_{3,1}$. So an application of \Cref{thm:ckgp_fibration} then implies the $D_{\Gamma_i}^o$ also have the CKgP for $i=7,8$. Note that for $i=8$, the fibration $\pi:D_{\Gamma_8}\rightarrow \pi(D_{\Gamma_8})$ can have reducible fibers. So, in order to apply \Cref{thm:ckgp_fibration}, we need to stratify $D_{\Gamma_8}=(D_{\Gamma_8})^1\cup (D_{\Gamma_8})^2\cup (D_{\Gamma_8})^3$ so that the restriction of $\pi$ to $(D_{\Gamma_8})^i$ has irreducible fibers. For this, we take $(D_{\Gamma_1})^1$ and $(D_{\Gamma_2})^2$ to be the (closed) boundary strata given, respectively, by the following two enhanced level graphs (of codimension two):
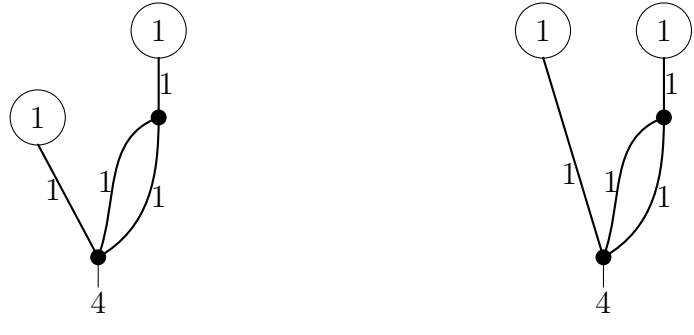
\begin{figure}[H]
  \centering
\begin{subfigure}{0.4\textwidth}
    \centering
\begin{tikzpicture}[main/.style = {draw, circle}]
    %\node[main] at (0,0) {1};
    \fill[black] (0,0) circle (3pt);
    %\fill[black] (-.8,1.85) circle (3pt);
    \node[main] at (-.8,1.85) {1};
    \fill[black] (.8,1.85) circle (3pt);
    \node[main] at (.8,3) {1};

    \draw[thick] (0,0) to (-.8, 1.5);
    \draw[thick] (0,0) to [out=70, in=200] (.8, 1.85);
    \draw[thick] (0,0) to [out=30, in=-90] (.8, 1.85);
    \draw[thick] (0.8,2.65) to (.8, 1.85);
    \draw (0,0) -- (0,-.4);

    \draw node at  (0,-.6) {$4$};
    \draw node at  (-0.6,0.9) {$1$};
    \draw node at  (0.8,0.8) {$1$};
    \draw node at  (0.1,1) {$1$};
    \draw node at  (0.9,2.3) {$1$};
  \end{tikzpicture}
\end{subfigure}
\begin{subfigure}{0.4\textwidth}
    \centering
\begin{tikzpicture}[main/.style = {draw, circle}]
    %\node[main] at (0,0) {1};
    \fill[black] (0,0) circle (3pt);
    %\fill[black] (-.8,3) circle (3pt);
    \node[main] at (-.8,3) {1};
    \fill[black] (.8,1.85) circle (3pt);
    \node[main] at (.8,3) {1};

    \draw[thick] (0,0) to (-.8, 2.65);
    \draw[thick] (0,0) to [out=70, in=200] (.8, 1.85);
    \draw[thick] (0,0) to [out=30, in=-90] (.8, 1.85);
    \draw[thick] (0.8,2.65) to (.8, 1.85);
    \draw (0,0) -- (0,-.4);

    \draw node at  (0,-.6) {$4$};
    \draw node at  (-0.45,1.1) {$1$};
    \draw node at  (0.8,0.8) {$1$};
    \draw node at  (0.1,1) {$1$};
    \draw node at  (0.9,2.3) {$1$};
  \end{tikzpicture},
\end{subfigure}
\end{figure}
and $(D_{\Gamma_8})^3$ to be the closure of the complement of $(D_{\Gamma_8})^1\cup (D_{\Gamma_8})^2$ in $D_{\Gamma_8}$.\\

Regarding higher codimensional boundary strata, note that only $\Gamma_1,\Gamma_2$ and $\Gamma_5$ have edges with enhancements $\ge 2$. Among them, while $D_{\Gamma_1}\cap D_{\Gamma_5}\ne \emptyset$ and $D_{\Gamma_2}\cap D_{\Gamma_5}\ne \emptyset$, the corresponding 3-level graphs (denote $\Gamma_{1,5}$ and $\Gamma_{2,5}$, respectively) do not acquire a long edge (that is, the vertical edge that does not connect vertices in consecutive levels). Consequently, the number of prong-matching equivalence classes do not increase, so the boundary strata will still be covered by product of the generalized strata for each levels. Thus, for similar reason as before, each vertical boundary stratum also has the CKgP.

\subsection{The horizontal boundary}
Denote by $D_h$ the horizontal boundary divisor of $\overline\cH(4)$, then it is finitely covered by $\overline\cH(4,-1,-1)/S_2$. We will focus on the non-hyperelliptic component of $\cH(4)$ and denote by $\tilde{D}_h$ the non-hyperelliptic component of $\overline\cH(4,-1,-1)/S_2$. Then we will prove that $\tilde{D}_h$ (and thus also $D_h$) has the CKgP.

The vertical boundary divisors in $\tilde{D}_h$ are given by the pullback of $D_{\Gamma_i}$ to $\tilde{D}_h$, which we will denote by $D_{\Gamma_i,h}$ for simplicity (where $\Gamma_i$ are the level graphs depicted in \Cref{fig:vertical_graphs_mu_(4)}. Of course $D_{\Gamma_i,h}$ may have multiple irreducible components, possibly given by different level graphs (depending on which vertex of $\Gamma_i$ the horizontal edge manifests).

Let us denote by $f:\tilde{D}_h\rightarrow \overline\cM_{2,3}/S_2$ the forgetful morphism and by $X_h$ the normalization of $f(\tilde{D}_h)$. For brevity, we will also denote by morphism $\tilde{D}_h\rightarrow X_h$ also by $f$. Similarly, denote by $g:X_h\rightarrow \overline{\cM}_{2,1}$ the morphism obtained by forgetting the simple poles. Then $f$ and $g$ are both birational morphisms. Since $\overline\cM_{2,1}$ has the CKgP, to prove the same for $\tilde{D}_h$, we will analyze the exceptional divisors of $f$ and $g$. The $f$-exceptional divisors will be contained in $D_{\Gamma_7,h}$ or $D_{\Gamma_8,h}$, whereas the $g$-exceptional divisors will be contained in $f(D_{\Gamma_i,h})$ for $i=1,\ldots,5$ (it is easy to see that both components of $g\circ f(D_{\Gamma_6,h})$ will remain divisorial in $\overline\cM_{2,1}$).

Note that the discussion of prong-matchings from above still holds for enhanced level graphs in $\tilde{D}_h$, so each of $f(D_{\Gamma_i,h})$ can be finitely covered by the product of lower dimensional moduli spaces of differentials (since the boundary strata in $\overline\cM_{2,3}$ have such a product structure). The lower dimensional moduli spaces that appear in this context have the CKgP, so $X_h$ also has the CKgP. Consequently, to prove the same for $\tilde{D}_h$, it remains to verify $D_{\Gamma_7,h}$ and $D_{\Gamma_8,h}$ also have the CKgP.

So, now we look at these $f$-exceptional divisors, which can be represented by one of the following level graphs:
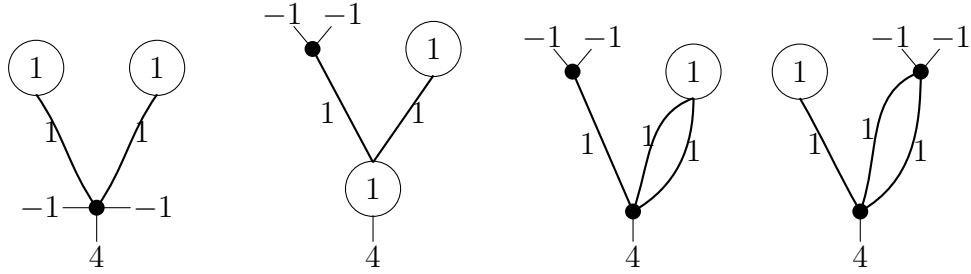
\begin{figure}[H]
  \centering
\begin{subfigure}{0.2\textwidth}
    \centering
\begin{tikzpicture}[main/.style = {draw, circle}]
    %\node[main] at (0,0) {1};
    \fill[black] (0,0) circle (3pt);
    \node[main] at (-.8,1.85) {1};
    \node[main] at (.8,1.85) {1};

    \draw[thick] (0,0) to [out=125, in=-55] (-.8, 1.5);
    \draw[thick] (0,0) to [out=55, in=-125] (.8, 1.5);
    \draw (0,0) -- (0,-.45);
    \draw (0,0) -- (-0.45,0);
    \draw (0,0) -- (0.45,0);

    \draw node at  (-0.75,0) {$-1$};
    \draw node at  (0.75,0) {$-1$};
    \draw node at  (0,-.65) {$4$};
    \draw node at  (-0.6,1) {$1$};
    \draw node at  (0.6,1) {$1$};
  \end{tikzpicture}
\end{subfigure}
\begin{subfigure}{0.2\textwidth}
    \centering
\begin{tikzpicture}[main/.style = {draw, circle}]
    \node[main] at (0,0) {1};
    \fill[black] (-.8,1.85) circle (3pt);
    %\node[main] at (-.8,1.85) {1};
    \node[main] at (.8,1.85) {1};

    \draw[thick] (0,.35) to (-.8, 1.85);
    \draw[thick] (0,.35) to [out=55, in=-125] (.8, 1.5);
    \draw (0,-.35) -- (0,-.7);
    \draw (-.8,1.85) -- (-1.05,2.15);
    \draw (-.8,1.85) -- (-.55,2.15);

    \draw node at  (0,-.9) {$4$};
    \draw node at  (-1.2,2.3) {$-1$};
    \draw node at  (-.4,2.3) {$-1$};
    \draw node at  (-0.6,1) {$1$};
    \draw node at  (0.6,1) {$1$};
  \end{tikzpicture}
\end{subfigure}
\begin{subfigure}{0.2\textwidth}
    \centering
\begin{tikzpicture}[main/.style = {draw, circle}]
    %\node[main] at (0,0) {1};
    \fill[black] (0,0) circle (3pt);
    \fill[black] (-.8,1.85) circle (3pt);
    %\node[main] at (-.8,1.85) {1};
    \node[main] at (.8,1.85) {1};

    \draw[thick] (0,0) to (-.8, 1.85);
    \draw[thick] (0,0) to [out=70, in=200] (.8, 1.5);
    \draw[thick] (0,0) to [out=30, in=-90] (.8, 1.5);
    \draw (0,0) -- (0,-.4);
    \draw (-.8,1.85) -- (-1.05,2.15);
    \draw (-.8,1.85) -- (-.55,2.15);

    \draw node at  (0,-.6) {$4$};
    \draw node at  (-1.2,2.3) {$-1$};
    \draw node at  (-.4,2.3) {$-1$};
    \draw node at  (-0.6,0.9) {$1$};
    \draw node at  (0.8,0.8) {$1$};
    \draw node at  (0.2,1) {$1$};
  \end{tikzpicture}
\end{subfigure}
\begin{subfigure}{0.2\textwidth}
    \centering
\begin{tikzpicture}[main/.style = {draw, circle}]
    %\node[main] at (0,0) {1};
    \fill[black] (0,0) circle (3pt);
    \node[main] at (-.8,1.85) {1};
    %\node[main] at (.8,1.85) {1};
    \fill[black] (0.8,1.85) circle (3pt);

    \draw[thick] (0,0) to [out=120, in=-60] (-.8, 1.5);
    \draw[thick] (0,0) to [out=70, in=200] (.8, 1.85);
    \draw[thick] (0,0) to [out=30, in=-90] (.8, 1.85);
    \draw (0,0) -- (0,-.4);
    \draw (.8,1.85) -- (1.05,2.15);
    \draw (.8,1.85) -- (.55,2.15);

    \draw node at  (0,-.6) {$4$};
    \draw node at  (1.2,2.3) {$-1$};
    \draw node at  (.4,2.3) {$-1$};
    \draw node at  (-0.6,0.9) {$1$};
    \draw node at  (0.8,0.8) {$1$};
    \draw node at  (0.1,1.1) {$1$};
  \end{tikzpicture}
\end{subfigure}.
\caption{$f$-exceptional divisors on $\tilde{D}_h=\overline\cH(4,-1,-1)/S_2$}
\label{fig:f_exceptional_mu_4}
\end{figure}
Of course if $\Lambda$ is any of the four level graphs from above then $f(D_\Lambda)$ is a product of lower dimensional moduli spaces of differentials (potentially with residue conditions), all of which have the CKgP. So, that each of these boundary divisors have the CKgP then follows from \Cref{thm:ckgp_fibration}. Note that while in each cases, the fiber of $D_\Lambda\rightarrow f(D_\Lambda)$ is one dimensional, in the third and fourth cases, there can be reducible fibers. In such cases, to apply \Cref{thm:ckgp_fibration}, we need to stratify $D_\Lambda$ to assure the fibers remain irreducible, as we did above in previous subsection.\\

With this, we can prove:
\begin{proposition}\label[proposition]{prop:chow_ring_mu_(4)}
  Both components of the stratum $\overline\cH(4)$ have the CKgP and the cohomology (as well as the Chow) ring is tautological, generated by the boundary strata.
\end{proposition}
\begin{proof}
  Our discussions above have usually focused on the non-hyperelliptic component of $\overline\cH(4)$. In case of the hyperelliptic component, there is a natural morphism $\overline{M}_{0,8}\rightarrow \overline{M}_{3,1}$ on the level of coarse moduli spaces, taking $(R;p_1,\ldots,p_8)\in \overline{M}_{0,8}$ to $(C,q)\in \overline{M}_{3,1}$, where $C$ is the (stabilization of the) admissible double cover of $R$ branched along $p_1,\ldots,p_8$ and $q$ is the pre-image of $p_1$ under this double cover. The image of this morphism is exactly the image of the hyperelliptic component of $\overline\cH(4)$ in $\overline{M}_{3,1}$. The exceptional divisors can be dealt with analogously to before. This proves the first part of the proposition.

  The second part, that is the computation of the cohomology and Chow rings, is a consequence of \Cref{rem:ckgp_bdry_chow}. Indeed, the Chow ring of $\cH(4)$ is generated by the fundamental classes of the two components, and same can be seen to be true for each (open) boundary stratum $D_\Gamma^o$. Thus, by excision on Chow, the (rational) Chow ring is generated by the boundary strata (and thus is tautological). Since $\overline\cH(4)$ is a smooth and proper Deligne-Mumford stack and has the CKgP, its (rational) cohomology ring is isomorphic to its (rational) Chow ring.
\end{proof}
\section{The stratum $\overline{\mathcal{H}}(2,2)$} \label{sec:h22}

In this section, we will prove that the multi-scale compactification $\overline\cH(2,2)$ has the CKgP and the cohomology (as well as the Chow) ring is generated by the boundary strata. Note that $\overline\cH(2,2)$ has two components: the hyperelliptic $\overline\cH^{h}(2,2)$ and the non-hyperelliptic $\overline\cH^{nh}(2,2)$ ones. By considering the morphism from $\overline{M}_{0,9}$ to the image of $\overline\cH^{h}(2,2)$ in $\overline{M}_{3,2}$ analogus to the $\mu=(4)$ case in \Cref{prop:chow_ring_mu_(4)}, we can easily show that the hyperelliptic component has the CKgP with the Chow ring generated by the boundary divisors. So, in our discussion, we will put more emphasis on the non-hyperelliptic component $\overline\cH^{nh}(2,2)$.

\subsection{Enumerating the boundary divisors}
The boundary divisors of $\overline\cH(2,2)$ are either horizontal boundary $D_h$ or given by one of the following (enhanced) level graphs:
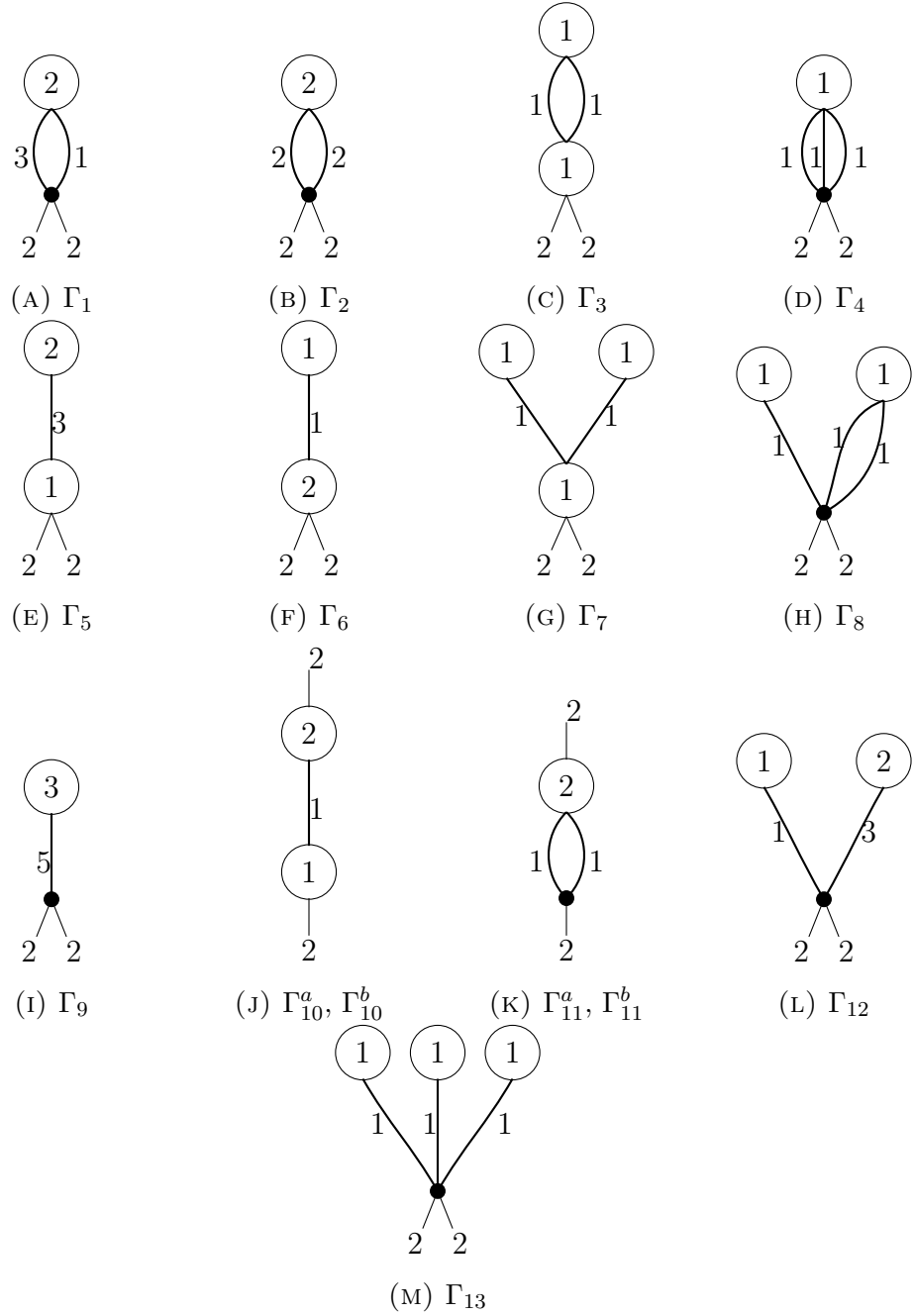
\begin{figure}[H]
  \centering
  \begin{subfigure}{0.2\textwidth}
    \centering
\begin{tikzpicture}[main/.style = {draw, circle}]
    \node[main] at (0,0) {2};
    %\fill[black] (0,0) circle (3pt);
    \fill[black] (0, -1.5) circle (3pt);

    \draw[thick] (0,-.35) to [out=225, in=135] (0, -1.5);
    \draw[thick] (0,-.35) to [out=315, in=45] (0, -1.5);
    \draw (0,-1.5) -- (0.2,-2);
    \draw (0,-1.5) -- (-0.2,-2);

    \draw node at  (0.3,-2.2) {$2$};
    \draw node at  (-0.3,-2.2) {$2$};
    \draw node at  (-0.4,-1) {$3$};
    \draw node at  (0.4,-1) {$1$};
  \end{tikzpicture}
  \caption{$\Gamma_1$}
\end{subfigure}
\begin{subfigure}{0.2\textwidth}
    \centering
\begin{tikzpicture}[main/.style = {draw, circle}]
    \node[main] at (0,0) {2};
    %\fill[black] (0,0) circle (3pt);
    \fill[black] (0, -1.5) circle (3pt);

    \draw[thick] (0,-.35) to [out=225, in=135] (0, -1.5);
    \draw[thick] (0,-.35) to [out=315, in=45] (0, -1.5);
    \draw (0,-1.5) -- (0.2,-2);
    \draw (0,-1.5) -- (-0.2,-2);

    \draw node at  (0.3,-2.2) {$2$};
    \draw node at  (-0.3,-2.2) {$2$};
    \draw node at  (-0.4,-1) {$2$};
    \draw node at  (0.4,-1) {$2$};
  \end{tikzpicture}
  \caption{$\Gamma_2$}
\end{subfigure}
\begin{subfigure}{0.2\textwidth}
    \centering
\begin{tikzpicture}[main/.style = {draw, circle}]
    \node[main] at (0,0) {1};
    %\fill[black] (0,0) circle (3pt);
    \node[main] at (0,-1.85) {1};

    \draw[thick] (0,-.35) to [out=225, in=135] (0, -1.5);
    \draw[thick] (0,-.35) to [out=315, in=45] (0, -1.5);
    \draw (0,-2.2) -- (0.2,-2.7);
    \draw (0,-2.2) -- (-0.2,-2.7);

    \draw node at  (0.3,-2.9) {$2$};
    \draw node at  (-0.3,-2.9) {$2$};
    \draw node at  (-0.4,-1) {$1$};
    \draw node at  (0.4,-1) {$1$};
  \end{tikzpicture}
  \caption{$\Gamma_3$}
\end{subfigure}
\begin{subfigure}{0.2\textwidth}
    \centering
\begin{tikzpicture}[main/.style = {draw, circle}]
    \node[main] at (0,0) {1};
    %\fill[black] (0,0) circle (3pt);
    \fill[black] (0, -1.5) circle (3pt);

    \draw[thick] (0,-.35) to [out=210, in=150] (0, -1.5);
    \draw[thick] (0,-.35) to [out=330, in=30] (0, -1.5);
    \draw[thick] (0,-.35) to [out=270, in=90] (0, -1.5);
    \draw (0,-1.5) -- (0.2,-2);
    \draw (0,-1.5) -- (-0.2,-2);

    \draw node at  (0.3,-2.2) {$2$};
    \draw node at  (-0.3,-2.2) {$2$};
    \draw node at  (-0.5,-1) {$1$};
    \draw node at  (-0.1,-1) {$1$};
    \draw node at  (0.5,-1) {$1$};
  \end{tikzpicture}
  \caption{$\Gamma_4$}
\end{subfigure}
\begin{subfigure}{0.2\textwidth}
    \centering
\begin{tikzpicture}[main/.style = {draw, circle}]
    \node[main] at (0,0) {2};
    %\fill[black] (0,0) circle (3pt);
    \node[main] at (0,-1.85) {1};

    \draw[thick] (0,-.35) to [out=270, in=90] (0, -1.5);
    \draw (0,-2.2) -- (0.2,-2.7);
    \draw (0,-2.2) -- (-0.2,-2.7);

    \draw node at  (0.3,-2.9) {$2$};
    \draw node at  (-0.3,-2.9) {$2$};
    \draw node at  (0.1,-1) {$3$};
  \end{tikzpicture}
  \caption{$\Gamma_5$}
\end{subfigure}
\begin{subfigure}{0.2\textwidth}
    \centering
\begin{tikzpicture}[main/.style = {draw, circle}]
    \node[main] at (0,0) {1};
    %\fill[black] (0,0) circle (3pt);
    \node[main] at (0,-1.85) {2};

    \draw[thick] (0,-.35) to [out=270, in=90] (0, -1.5);
    \draw (0,-2.2) -- (0.2,-2.7);
    \draw (0,-2.2) -- (-0.2,-2.7);

    \draw node at  (0.3,-2.9) {$2$};
    \draw node at  (-0.3,-2.9) {$2$};
    \draw node at  (0.1,-1) {$1$};
  \end{tikzpicture}
  \caption{$\Gamma_6$}
\end{subfigure}
\begin{subfigure}{0.2\textwidth}
    \centering
\begin{tikzpicture}[main/.style = {draw, circle}]
    \node[main] at (0,0) {1};
    %\fill[black] (0,0) circle (3pt);
    \node[main] at (-.8,1.85) {1};
    \node[main] at (.8,1.85) {1};

    \draw[thick] (0,.35) to [out=125, in=-55] (-.8, 1.5);
    \draw[thick] (0,.35) to [out=55, in=-125] (.8, 1.5);
    \draw (0,-.35) -- (0.2,-.8);
    \draw (0,-.35) -- (-0.2,-.8);

    \draw node at  (0.3,-1) {$2$};
    \draw node at  (-0.3,-1) {$2$};
    \draw node at  (-0.6,1) {$1$};
    \draw node at  (0.6,1) {$1$};
  \end{tikzpicture}
  \caption{$\Gamma_7$}
\end{subfigure}
\begin{subfigure}{0.2\textwidth}
    \centering
\begin{tikzpicture}[main/.style = {draw, circle}]
    %\node[main] at (0,0) {1};
    \fill[black] (0,0) circle (3pt);
    \node[main] at (-.8,1.85) {1};
    \node[main] at (.8,1.85) {1};

    \draw[thick] (0,0) to [out=120, in=-60] (-.8, 1.5);
    \draw[thick] (0,0) to [out=70, in=200] (.8, 1.5);
    \draw[thick] (0,0) to [out=30, in=-90] (.8, 1.5);
    \draw (0,0) -- (0.2,-.5);
    \draw (0,0) -- (-0.2,-.5);

    \draw node at  (0.3,-.7) {$2$};
    \draw node at  (-0.3,-.7) {$2$};
    \draw node at  (-0.6,0.9) {$1$};
    \draw node at  (0.8,0.8) {$1$};
    \draw node at  (0.2,1) {$1$};
  \end{tikzpicture}
  \caption{$\Gamma_8$}
\end{subfigure}
\begin{subfigure}{0.2\textwidth}
    \centering
\begin{tikzpicture}[main/.style = {draw, circle}]
    \node[main] at (0,0) {3};
    %\fill[black] (0,0) circle (3pt);
    \fill[black] (0, -1.5) circle (3pt);

    \draw[thick] (0,-.35) to [out=270, in=90] (0, -1.5);
    \draw (0,-1.5) -- (0.2,-2);
    \draw (0,-1.5) -- (-0.2,-2);

    \draw node at  (0.3,-2.2) {$2$};
    \draw node at  (-0.3,-2.2) {$2$};
    \draw node at  (-0.1,-1) {$5$};
  \end{tikzpicture}
  \caption{$\Gamma_{9}$}
\end{subfigure}
\begin{subfigure}{0.2\textwidth}
    \centering
\begin{tikzpicture}[main/.style = {draw, circle}]
    \node[main] at (0,0) {2};
    %\fill[black] (0,0) circle (3pt);
    \node[main] at (0,-1.85) {1};

    \draw[thick] (0,-.35) to [out=270, in=90] (0, -1.5);
    \draw (0,-2.2) -- (0,-2.7);
    \draw (0,0.35) -- (0,0.85);

    \draw node at  (0,-2.9) {$2$};
    \draw node at  (0.1,1) {$2$};
    \draw node at  (0.1,-1) {$1$};
  \end{tikzpicture}
  \caption{$\Gamma_{10}^a$, $\Gamma_{10}^b$}
\end{subfigure}
\begin{subfigure}{0.2\textwidth}
    \centering
\begin{tikzpicture}[main/.style = {draw, circle}]
    \node[main] at (0,0) {2};
    %\fill[black] (0,0) circle (3pt);
    \fill[black] (0, -1.5) circle (3pt);

    \draw[thick] (0,-.35) to [out=225, in=135] (0, -1.5);
    \draw[thick] (0,-.35) to [out=315, in=45] (0, -1.5);
    \draw (0,-1.5) -- (0,-2);
    \draw (0,0.35) -- (0,0.85);

    \draw node at  (0,-2.2) {$2$};
    \draw node at  (0.1,1) {$2$};
    \draw node at  (-0.4,-1) {$1$};
    \draw node at  (0.4,-1) {$1$};
  \end{tikzpicture}
  \caption{$\Gamma_{11}^a$, $\Gamma_{11}^b$}
\end{subfigure}
\begin{subfigure}{0.2\textwidth}
    \centering
\begin{tikzpicture}[main/.style = {draw, circle}]
    %\node[main] at (0,0) {1};
    \fill[black] (0,0) circle (3pt);
    \node[main] at (-.8,1.85) {1};
    \node[main] at (.8,1.85) {2};

    \draw[thick] (0,0) to [out=120, in=-60] (-.8, 1.5);
    \draw[thick] (0,0) to [out=60, in=-120] (.8, 1.5);
    \draw (0,0) -- (0.2,-.5);
    \draw (0,0) -- (-0.2,-.5);

    \draw node at  (0.3,-.7) {$2$};
    \draw node at  (-0.3,-.7) {$2$};
    \draw node at  (-0.6,0.9) {$1$};
    \draw node at  (0.6,0.9) {$3$};
  \end{tikzpicture}
  \caption{$\Gamma_{12}$}
\end{subfigure}
\begin{subfigure}{0.2\textwidth}
    \centering
\begin{tikzpicture}[main/.style = {draw, circle}]
    %\node[main] at (0,0) {1};
    \fill[black] (0,0) circle (3pt);
    \node[main] at (-1,1.85) {1};
    \node[main] at (0,1.85) {1};
    \node[main] at (1,1.85) {1};

    \draw[thick] (0,0) to [out=120, in=-60] (-1, 1.5);
    \draw[thick] (0,0) to [out=90, in=-90] (0, 1.5);
    \draw[thick] (0,0) to [out=60, in=-120] (1, 1.5);
    \draw (0,0) -- (0.2,-.5);
    \draw (0,0) -- (-0.2,-.5);

    \draw node at  (0.3,-.7) {$2$};
    \draw node at  (-0.3,-.7) {$2$};
    \draw node at  (-0.8,0.9) {$1$};
    \draw node at  (-0.1,0.9) {$1$};
    \draw node at  (0.9,0.9) {$1$};
  \end{tikzpicture}
  \caption{$\Gamma_{13}$}
\end{subfigure}
\caption{Vertical divisors in $\overline\cH(2,2)$}
\label{fig:vertical_graphs_mu_(22)}
\end{figure}
Depending on which marked double zero remains in the bottom level and which in the top level, the tenth and eleventh graphs above will have two possibilities each. So, we have distinguished the two possibilities by $\Gamma_i^a$ and $\Gamma_i^b$, respectively.

The level-wise product model for (the interior of) each of these boundary divisors can be enumerated as follows:
\begin{enumerate}
\item $D_h^o\sim \cH(2,2,-1,-1)/S_2$ is the horizontal divisor. It has two components, one in each of $\overline\cH^{h}(2,2)$ and $\overline\cH^{nh}(2,2)$.
\item $D_{\Gamma_1}^o\sim \cH(2,0)\times \cH(2,2,-4,-2)$. Analogus to the $\mu=(4)$ case, $D_{\Gamma_1}$ again has exactly one component contained in $\overline\cH^{nh}(2,2)$.
\item $D_{\Gamma_2}^o\sim (\cH(1,1)\times \cH(2,2,-3,-3))/S_2$ (where $S_2\cong\text{Aut}\,(\Gamma_2)$). Again, similarly to the $\mu=(4)$ case, each compatible collection of twisted differentials has two equivalence classs of prong-matchings, which leads to $D_{\Gamma_2}$ having two components, one in each of $\overline\cH^h(2,2)$ and $\overline\cH^{nh}(2,2)$.
\item $D_{\Gamma_3}^o\sim (\cH(0,0)\times\cH(2,2,-2,-2))/S_2$ (where $S_2\cong\text{Aut}\,(\Gamma_3)$). The stratum $\cH(2,2,-2,-2)$ has two components distinguished by the rotation number. Consequently, $D_{\Gamma_3}$ has two components, one in each of $\overline\cH^h(2,2)$ and $\overline\cH^{nh}(2,2)$.
\item $D_{\Gamma_4}^o\sim (\cH(0,0,0)\times \cH(2,2,-2,-2,-2))/S_3$ (where $S_3\cong\text{Aut}\,(\Gamma_4)$). This boundary divisor is irreducible contained in $\overline\cH^{nh}(2,2)$.
\item $D_{\Gamma_5}^o\sim \cH(2)\times \cH(2,2,-4)$. The stratum $\cH(2,2,-4)$ has two components distinguished by the rotation number. Consequently, $D_{\Gamma_5}$ has two components, one in each of $\overline\cH^{h}(2,2)$ and $\overline\cH^{nh}(2,2)$.
\item $D_{\Gamma_6}^o\sim \cH(0)\times \cH(2,2,-2)$. The stratum $\cH(2,2,-2)$ has two components distinguished by hyperellipticity. Consequently $D_{\Gamma_6}$ has two components, one in each of $\overline\cH^{h}(2,2)$ and $\overline\cH^{nh}(2,2)$.
\item $D_{\Gamma_7}^o\sim (\bP(\Omega\cM_{1,1}\times\Omega\cM_{1,1})\times \cR(2,2,-2,-2))/S_2$ (where $S_2\cong\text{Aut}\,(\Gamma_7)$). The residueless stratum $\cR(2,2,-2,-2)$ has two components distinguished by ramification profile. Consequently, $D_{\Gamma_7}$ has two components, one in each of $\overline\cH^{h}(2,2)$ and $\overline\cH^{nh}(2,2)$. We will show below that $\cR(2,2,-2,-2)/S_2$ has the CKgP.
\item $D_{\Gamma_8}^o\sim \bP(\Omega\cM_{1,1}\times\Omega\cM_{1,2})\times \cH(2,2\mid-2
\mid-2,-2)/S_2$ (where $S_2\cong\text{Aut}\,(\Gamma_8)$). This divisor is irreducible, contained in $\overline\cH^{nh}(2,2)$. We will see below that $\cH(2,2\mid-2
\mid-2,-2)$ has the CKgP.
\item $D_{\Gamma_9}^o\sim \cH(4)\times \cH(2,2,-6)\sim \cH(4)$. Since $\cH(4)$ has two components distinguished by hyperellipticity, $D_{\Gamma_9}$ also has two components, one in each of $\overline\cH^{h}(2,2)$ and $\overline\cH^{nh}(2,2)$.
\item $(D_{\Gamma_{10}}^a)^o\sim \cH(2,0)\times \cH(2,-2)$ is irreducible, contained in $\overline\cH^{nh}(2,2)$. Same for $D_{\Gamma_{10}}^b$.
\item $(D_{\Gamma_{11}}^a)^o\sim \cH(2,0,0)\times \cH(2,-2,-2)$ is also irreducible, contained in $\overline\cH^{nh}(2,2)$. Same for $D_{\Gamma_{11}}^b$.
\item $D_{\Gamma_{12}}^o\sim (\bP(\Omega\cM_{1,1}\times\Omega\cM_{2,1}(2))\times \cR(2,2,-2,-4))$. This divisor is irreducible, contained in $\overline\cH^{nh}(2,2)$. Note that $\cR(2,2,-2,-4)$ is 0-dimensional.
\item $D_{\Gamma_{13}}^o\sim (\bP(\Omega\cM_{1,1}\times\Omega\cM_{1,1}\times \Omega\cM_{1,1})\times \cR(2,2,-2,-2,-2))$. This divisor is irreducible, contained in $\overline\cH^{nh}(2,2)$. Note that $\cR(2,2,-2,-2,-2)$ is 0-dimensional.
\end{enumerate}
\subsection*{Strata with residue conditions:} From the above enumeration, we see that $\cR(2,2,-2,-2)/S_2$ (where $S_2$ symmetrizes double poles) and $\cH(2,2\mid-2\mid-2,-2)$ are the only two vertices that we have not yet shown to have the CKgP, which we do now. The case of $\cH(2,2\mid-2\mid-2,-2)$ is straightforward:
\begin{proposition}
    The stratum $\cH(2,2 \mid -2 \mid -2,-2)$ is a rational curve.     
\end{proposition}
\begin{proof}
    A differential in $\cH(2,2 \mid -2 \mid -2,-2)$ can be written as \[\omega=\frac{(z-a)^2(z-b)^2}{z^2(z-1)^2}dz\] for distinct $a,b\in \mathbb{P}^1\setminus\{0,\infty,1\}$. At the pole $z=0$, \[\operatorname{res}_{z=0}\omega=2a^2b^2\left(1-\frac{1}{a}-\frac{1}{b}\right) = 0.\]

    So $\cH(2,2 \mid -2 \mid -2,-2)$ is isomorphic to a dense subset of the plane conic curve $V(ab-b-a)\subset \operatorname{Spec}\mathbb{C}[a,b]$. 
\end{proof}

The case of the residueless stratum $\cR(2,2,-2,-2)/S_2$ requires a bit more work. Recall that $\cR(2,2,-2,-2)$ has two connected components, both hyperelliptic. They are also distinguished by rotation number ($r=1$ or $2$). Let $\cC_r$ be the connected component with rotation number $r$.

\begin{proposition}
    Each connected component of the stratum $\cR^u(2,2,-2,-2)\coloneqq \cR(2,2,-2,-2)/S_2$ (where $S_2$ acts by symmetrizing the double poles) is a rational surface.
\end{proposition}
For the component of rotation number 1, we prove a stronger statement for the stratum without the $S_2$ action.
\begin{proof}
    Let us denote by $\overline\cR(2,2,-2,-2)$ the multi-scale compactification of $\overline\cR(2,2,-2,-2)$ (and by $\overline\cR^u(2,2,-2,-2)$ its $S_2$ quotient). First consider the component $\cR_2^u\subset \overline\cR^u(2,2,-2,-2)$ of rotation number 2. Forgetting the double zeros gives a generically finite morphism $p:\cR_2^u\rightarrow \overline{M}_{1,2}/S_2$. We claim that on the level of coarse moduli space, this morphism is generically one-to-one. Let $\delta_{\text{irr}}\subset \overline{M}_{1,2}/S_2$ be the boundary divisor parameterizing rational nodal curves. Then $p^{-1}(\delta_{\text{irr}})$ consists of the horizontal boundary divisor $D_h^2\subset \cR_2^u$ and $D_\Gamma^2$ corresponding to the following enhanced level graph:
    
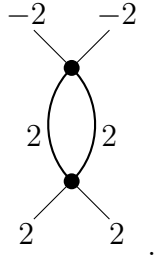
\begin{figure}[H]
        \centering
        \begin{tikzpicture}[main/.style = {draw, circle}]
    \fill[black] (0,0) circle (3pt);
    \fill[black] (0, -1.5) circle (3pt);

    \draw[thick] (0,0) to [out=225, in=135] (0, -1.5);
    \draw[thick] (0,0) to [out=315, in=45] (0, -1.5);
    \draw (0,-1.5) -- (0.5,-2);
    \draw (0,-1.5) -- (-0.5,-2);
    \draw (0,0) -- (0.5,0.5);
    \draw (0,0) -- (-0.5,0.5);

    \draw node at  (0.6,-2.2) {$2$};
    \draw node at  (-0.6,-2.2) {$2$};
    \draw node at  (-0.5,-0.9) {$2$};
    \draw node at  (0.5,-0.9) {$2$};
    \draw node at  (0.6,0.7) {$-2$};
    \draw node at  (-0.6,0.7) {$-2$};
  \end{tikzpicture}.
    \end{figure}
    Of course $D_\Gamma^2$ is a $p$-exceptional divisor, so for a generic $(C;q_1+q_2)\in \delta_{\text{irr}}$, the fiber $p^{-1}(C;q_1+q_2)$ is contained in $D_h^2$. Since $p$ is unramified along $D_h^2$, to show it is generically one-to-one, it suffices to prove that $p|_{D_h^2}$ is one-to-one. Note that $D_h^2$ is birational to the component of $\overline\cH(2,2\ |-2\ |\ -2\ |\ -1,-1)/(S_2\times S_2)$ of rotation number 2, where the first factor of $S_2$ symmetrizes the two double poles, and the second factor of $S_2$ symmetrizes the two simple poles (and the vertical lines "|" delineate the poles whose residues sum to 0). Then the morphism $D_h^2\rightarrow \delta_{\text{irr}}$ then corresponds to $f:\overline\cH(2,2\ |-2\ |\ -2\ |\ -1,-1)/(S_2\times S_2)\rightarrow \overline{M}_{0,4}/(S_2\times S_2)$. Suppose the two double poles be 0 and $\infty$ and the two simple poles be 1 and $a\in \mathbb{P}^1\setminus\{0,1,\infty\}$, then a differential in its $f$-fiber can be written as:
    \[\omega=\frac{(z-b)^2(z-c)^2}{z^2(z-1)(z-a)}dz\] for $b,c\in \mathbb{P}^1\setminus\{0,\infty,1,a\}$. Imposing $\text{Res}_0(\omega)=0$ gives:
    \[bc\left(1+\frac{1}{a}\right)=2b+2c.\]On the other hand, imposing $\text{Res}_\infty(\omega)=0$ gives:
    \[b+c=\frac{a+1}{2}.\] Solving these two equations yield two components: $H_1\coloneqq\{a=-1,\ b+c=0\}$ and $H_2=\{bc=a,\ b+c=(a+1)/2\}$. Since the hyperelliptic involution swaps each of the two double zeros, the two double poles and the simple poles in the component of rotation number 2, it has to be $H_2$ (with the involution $\iota:z\mapsto a/z)$. Thus, the morphism $f$ is given by $H_2\rightarrow \overline{M}_{0,4}/(S_2\times S_2)$. Then the above computation shows that above general point in $M_{0,4}/(S_2\times S_2)$, there is a unique differential in $H_2$, whose zeros are conjugate under the hyperelliptic involution. Since a general point in $M_{0,4}/(S_2\times S_2)$ has a non-trivial isotropy containing the hyperelliptic involution, $f$ is one-to-one on the level of coarse moduli space, which implies $p|_{D_{\text{irr}}}$ (and thus also $p$) is generically one-to-one. Since $\overline{M}_{1,2}/S_2$ is rational, we conclude that $\cR_2^u$ is also rational.\\

    Next, let us look at the component $\cR_1\subset \overline\cR(2,2,-2,-2)$, without symmetrizing the double poles. Then, unlike the rotation number 2 case above, the morphism $p:\cR_1\rightarrow\overline{M}_{1,2}$ defined by forgetting the double zeros is not dominant. Indeed, from the above computations, we saw that the horizontal boundary $D_h^1$ is birational to $H_1=\{a=-1,\ b+c=0\}$ (where we symmetrize the simple poles $1$ and $a=-1$). So, forgetting the two double zeros $b$ and $c$ gives a morphism to $\overline{M}_{0,4}/S_2$ whose image is just a point $(\mathbb{P}^1;0,\infty,1,-1)$. We can analogously list all other boundary divisors of non-compact type and see that they all get contracted by $p$.

    In fact the image $p(\cR_1)\subset \overline{M}_{1,2}$ is contained in $\overline{\cH}(2,-2)$. Consequently, if we denote by $\cR_1^o\subset\cR_1$ its interior then $\cR_1^o\subset \cH(2,2,-2,-2)\cap \cH(0,0,2,-2)$ in $M_{1,4}$. That is, $(C;p_1,p_2,p_3,p_4)\in\cR_1^o$ satisfies $2p_1+2p_2\sim 2p_3+2p_4$ and $2p_3\sim 2p_4$. The hyperelliptic involution $\iota$ is then the involution associated to the pencil $|2p_3|=|2p_4|$. Since the double zeros $p_1$ and $p_2$ are swapped by the involution, the fiber of $p:\cR_1\rightarrow \cH(2,-2)\subset M_{1,2}$ is exactly $\{(p_1,p_2)\,|\,p_1+p_2\sim 2p_3\sim 2p_4\}/\langle\iota\rangle$, which is exactly the linear series $|2p_3|=|2p_4|$. That is, $\cR_1$ is birational to a $\mathbb{P}^1$-bundle over $\cH(2,-2)\cong\mathbb{P}^1$, and thus is rational.
\end{proof}
\begin{corollary}
    The Chow ring of both components of $\cR^u(2,2,-2,-2)$ is generated by boundary divisors.
\end{corollary}
\begin{proof}
    This is the consequence of the rational parameterization constructed in the preceding proof. For the component of rotation number 1, the interior $\cR_1^o\subset\cR_1$ is an open subset of the $\mathbb{P}^1$-bundle over $\cH(2,-2)$, so it has trivial Chow ring generated by the fundamental class. Consequently, $CH^*(\cR_1)$ is generated by the components of the boundary divisors.

    For the component of rotation number 2, the interior $(\cR_2^u)^o\subset \cR_2^u$ admits a generically one-to-one morphism to $M_{1,2}/S_2$. So, it suffices to prove that it has no exceptional divisor, so that $(\cR_2^u)^o$ embeds into $M_{1,2}/S_2$ as an open subset. Note that $\cR_2^u\subset \overline\cH(1,1,-1,-1)/S_2$ and the fiber over the point $(E;q_1+q_2)$ of the morphism $\overline\cH(1,1,-1,-1)/S_2\rightarrow \overline{M}_{1,2}/S_2$ is isomorphic to the pencil $|q_1+q_2|$. The global sections of the invertible sheaf $\mathcal{O}_E(q_1+q_2)$ can be viewed as meromorphic functions on $E$ with at worst simple poles at $q_1$ and $q_2$ but holomorphic otherwise. So take a meromorphic function $f\in H^0(E,\mathcal{O}_E(q_1+q_2))$ with simple poles at $q_1$ and $q_2$ such that $f^2dz$ is a differential of second kind (i.e. poles are residueless). Then $(f+c)^2dz = f^2dz+2cfdz+c^2dz$ cannot have residueless poles unless $c=0$, so the morphism $(\cR_2^u)^o\rightarrow M_{1,2}/S_2$ is injective. Then the conclusion on Chow ring of $\cR_2^u$ follows since the Chow ring of $\overline{M}_{1,2}$ is generated by the boundary divisors.
    \end{proof}

\subsection*{Certifying CKgP} With this, we see that all the vertices appearing in the each of the divisorial level graphs represent a lower genus stratum of differentials that has the CKgP. Same can be seen to be true also for higher codimension level graphs (as argued before, we only have to look up to codimension 2, and these level graphs have been included in the appendix, particularly in Figure~\ref{fig:h4}, Figure~\ref{fig:h13} and Figure~\ref{fig:h22}). Thus, the situation is quite parallel to that of $\mu=(4)$ case: the only divisorial level graph with at least two non-trivial prong equivalence classes is $\Gamma_2$, and this causes $D_{\Gamma_2}$ to have two distinct components, one in each of $\overline\cH^{h}(2,2)$ and $\overline\cH^{nh}(2,2)$. So, we can argue as in the $\mu=(4)$ case to show that all the vertical boundary strata to show that $\overline\cH(2,2)$ also has the CKgP. Indeed, for $i=7,8,12,13$, the argument will involve the application of \Cref{thm:ckgp_fibration}, whereas for the remainig boundary strata, they will be finite covered by the product of the generalized strata for each levels.

The horizontal boundary divisor $D_h$ of $\overline\cH^{nh}(2,2)$, on the other hand, is finitely covered by $\overline\cH^{nh}(2,2,-1,-1)/S_2$, which, in turn, admits a birational morphism onto $\overline\cM_{2,2}$. So, showing $D_h$ has the CKgP amounts to analyzing the exceptional divisors of this birational morphism, as in $\mu=(4)$ case.

As in $\mu=(4)$ case, the hyperelliptic component is easier to analyze because the image of $\overline\cH^h(2,2)$ in $\overline\cM_{3,2}$ is finitely covered by $\overline{M}_{0,9}$. Thus we have:
\begin{proposition}
    Both components of the compactified stratum $\overline\cH(2,2)$ have the CKgP and each of their Chow (as well as the cohomology) ring is tautological, generated by the boundary strata.
\end{proposition}

\section{The stratum $\overline{\mathcal{H}}(3,1)$} \label{sec:h31}
In this section, we will prove that the multi-scale compactification $\overline\cH(3,1)$ has the CKgP and the cohomology (as well as the Chow) ring is generated by the boundary strata.

\subsection{Enumerating the boundary divisors}
The boundary divisors of $\overline\cH(3,1)$ are either horizontal boundary $D_h$ or given by one of the following (enhanced) level graphs:
\begin{figure}[H]
  \centering
  \begin{subfigure}{0.2\textwidth}
    \centering
\begin{tikzpicture}[main/.style = {draw, circle}]
    \node[main] at (0,0) {2};
    %\fill[black] (0,0) circle (3pt);
    \fill[black] (0, -1.5) circle (3pt);

    \draw[thick] (0,-.35) to [out=225, in=135] (0, -1.5);
    \draw[thick] (0,-.35) to [out=315, in=45] (0, -1.5);
    \draw (0,-1.5) -- (0.2,-2);
    \draw (0,-1.5) -- (-0.2,-2);

    \draw node at  (0.3,-2.2) {$1$};
    \draw node at  (-0.3,-2.2) {$3$};
    \draw node at  (-0.4,-1) {$1$};
    \draw node at  (0.4,-1) {$3$};
  \end{tikzpicture}
  \caption{$\Gamma_1$}
\end{subfigure}
\begin{subfigure}{0.2\textwidth}
    \centering
\begin{tikzpicture}[main/.style = {draw, circle}]
    \node[main] at (0,0) {2};
    %\fill[black] (0,0) circle (3pt);
    \fill[black] (0, -1.5) circle (3pt);

    \draw[thick] (0,-.35) to [out=225, in=135] (0, -1.5);
    \draw[thick] (0,-.35) to [out=315, in=45] (0, -1.5);
    \draw (0,-1.5) -- (0.2,-2);
    \draw (0,-1.5) -- (-0.2,-2);

    \draw node at  (0.3,-2.2) {$1$};
    \draw node at  (-0.3,-2.2) {$3$};
    \draw node at  (-0.4,-1) {$2$};
    \draw node at  (0.4,-1) {$2$};
  \end{tikzpicture}
  \caption{$\Gamma_2$}
\end{subfigure}
\begin{subfigure}{0.2\textwidth}
    \centering
\begin{tikzpicture}[main/.style = {draw, circle}]
    \node[main] at (0,0) {1};
    %\fill[black] (0,0) circle (3pt);
    \node[main] at (0,-1.85) {1};

    \draw[thick] (0,-.35) to [out=225, in=135] (0, -1.5);
    \draw[thick] (0,-.35) to [out=315, in=45] (0, -1.5);
    \draw (0,-2.2) -- (0.2,-2.7);
    \draw (0,-2.2) -- (-0.2,-2.7);

    \draw node at  (0.3,-2.9) {$1$};
    \draw node at  (-0.3,-2.9) {$3$};
    \draw node at  (-0.4,-1) {$1$};
    \draw node at  (0.4,-1) {$1$};
  \end{tikzpicture}
  \caption{$\Gamma_3$}
\end{subfigure}
\begin{subfigure}{0.2\textwidth}
    \centering
\begin{tikzpicture}[main/.style = {draw, circle}]
    \node[main] at (0,0) {1};
    %\fill[black] (0,0) circle (3pt);
    \fill[black] (0, -1.5) circle (3pt);

    \draw[thick] (0,-.35) to [out=210, in=150] (0, -1.5);
    \draw[thick] (0,-.35) to [out=330, in=30] (0, -1.5);
    \draw[thick] (0,-.35) to [out=270, in=90] (0, -1.5);
    \draw (0,-1.5) -- (0.2,-2);
    \draw (0,-1.5) -- (-0.2,-2);

    \draw node at  (0.3,-2.2) {$1$};
    \draw node at  (-0.3,-2.2) {$3$};
    \draw node at  (-0.5,-1) {$1$};
    \draw node at  (-0.1,-1) {$1$};
    \draw node at  (0.5,-1) {$1$};
  \end{tikzpicture}
  \caption{$\Gamma_4$}
\end{subfigure}
\begin{subfigure}{0.2\textwidth}
    \centering
\begin{tikzpicture}[main/.style = {draw, circle}]
    \node[main] at (0,0) {2};
    %\fill[black] (0,0) circle (3pt);
    \node[main] at (0,-1.85) {1};

    \draw[thick] (0,-.35) to [out=270, in=90] (0, -1.5);
    \draw (0,-2.2) -- (0.2,-2.7);
    \draw (0,-2.2) -- (-0.2,-2.7);

    \draw node at  (0.3,-2.9) {$1$};
    \draw node at  (-0.3,-2.9) {$3$};
    \draw node at  (0.1,-1) {$3$};
  \end{tikzpicture}
  \caption{$\Gamma_5$}
\end{subfigure}
\begin{subfigure}{0.2\textwidth}
    \centering
\begin{tikzpicture}[main/.style = {draw, circle}]
    \node[main] at (0,0) {1};
    %\fill[black] (0,0) circle (3pt);
    \node[main] at (0,-1.85) {2};

    \draw[thick] (0,-.35) to [out=270, in=90] (0, -1.5);
    \draw (0,-2.2) -- (0.2,-2.7);
    \draw (0,-2.2) -- (-0.2,-2.7);

    \draw node at  (0.3,-2.9) {$1$};
    \draw node at  (-0.3,-2.9) {$3$};
    \draw node at  (0.1,-1) {$1$};
  \end{tikzpicture}
  \caption{$\Gamma_6$}
\end{subfigure}
\begin{subfigure}{0.2\textwidth}
    \centering
\begin{tikzpicture}[main/.style = {draw, circle}]
    \node[main] at (0,0) {1};
    %\fill[black] (0,0) circle (3pt);
    \node[main] at (-.8,1.85) {1};
    \node[main] at (.8,1.85) {1};

    \draw[thick] (0,.35) to [out=125, in=-55] (-.8, 1.5);
    \draw[thick] (0,.35) to [out=55, in=-125] (.8, 1.5);
    \draw (0,-.35) -- (0.2,-.8);
    \draw (0,-.35) -- (-0.2,-.8);

    \draw node at  (0.3,-1) {$1$};
    \draw node at  (-0.3,-1) {$3$};
    \draw node at  (-0.6,1) {$1$};
    \draw node at  (0.6,1) {$1$};
  \end{tikzpicture}
  \caption{$\Gamma_7$}
\end{subfigure}
\begin{subfigure}{0.2\textwidth}
    \centering
\begin{tikzpicture}[main/.style = {draw, circle}]
    %\node[main] at (0,0) {1};
    \fill[black] (0,0) circle (3pt);
    \node[main] at (-.8,1.85) {1};
    \node[main] at (.8,1.85) {1};

    \draw[thick] (0,0) to [out=120, in=-60] (-.8, 1.5);
    \draw[thick] (0,0) to [out=70, in=200] (.8, 1.5);
    \draw[thick] (0,0) to [out=30, in=-90] (.8, 1.5);
    \draw (0,0) -- (0.2,-.5);
    \draw (0,0) -- (-0.2,-.5);

    \draw node at  (0.3,-.7) {$1$};
    \draw node at  (-0.3,-.7) {$3$};
    \draw node at  (-0.6,0.9) {$1$};
    \draw node at  (0.8,0.8) {$1$};
    \draw node at  (0.2,1) {$1$};
  \end{tikzpicture}
  \caption{$\Gamma_8$}
\end{subfigure}
\begin{subfigure}{0.2\textwidth}
    \centering
\begin{tikzpicture}[main/.style = {draw, circle}]
    \node[main] at (0,0) {3};
    %\fill[black] (0,0) circle (3pt);
    \fill[black] (0, -1.5) circle (3pt);

    \draw[thick] (0,-.35) to [out=270, in=90] (0, -1.5);
    \draw (0,-1.5) -- (0.2,-2);
    \draw (0,-1.5) -- (-0.2,-2);

    \draw node at  (0.3,-2.2) {$1$};
    \draw node at  (-0.3,-2.2) {$3$};
    \draw node at  (-0.1,-1) {$5$};
  \end{tikzpicture}
  \caption{$\Gamma_{9}$}
\end{subfigure}
\begin{subfigure}{0.2\textwidth}
    \centering
\begin{tikzpicture}[main/.style = {draw, circle}]
    \node[main] at (0,0) {2};
    %\fill[black] (0,0) circle (3pt);
    \node[main] at (0,-1.85) {1};

    \draw[thick] (0,-.35) to [out=270, in=90] (0, -1.5);
    \draw (0,-2.2) -- (0,-2.7);
    \draw (0,0.35) -- (0,0.85);

    \draw node at  (0,-2.9) {$3$};
    \draw node at  (0.1,1) {$1$};
    \draw node at  (0.1,-1) {$2$};
  \end{tikzpicture}
  \caption{$\Gamma_{10}$}
\end{subfigure}
\begin{subfigure}{0.2\textwidth}
    \centering
\begin{tikzpicture}[main/.style = {draw, circle}]
    \node[main] at (0,0) {2};
    %\fill[black] (0,0) circle (3pt);
    \fill[black] (0, -1.5) circle (3pt);

    \draw[thick] (0,-.35) to [out=225, in=135] (0, -1.5);
    \draw[thick] (0,-.35) to [out=315, in=45] (0, -1.5);
    \draw (0,-1.5) -- (0,-2);
    \draw (0,0.35) -- (0,0.85);

    \draw node at  (0,-2.2) {$3$};
    \draw node at  (0.1,1) {$1$};
    \draw node at  (-0.4,-1) {$2$};
    \draw node at  (0.4,-1) {$1$};
  \end{tikzpicture}
  \caption{$\Gamma_{11}$}
\end{subfigure}
\begin{subfigure}{0.2\textwidth}
    \centering
\begin{tikzpicture}[main/.style = {draw, circle}]
    %\node[main] at (0,0) {1};
    \fill[black] (0,0) circle (3pt);
    \node[main] at (-.8,1.85) {1};
    \node[main] at (.8,1.85) {2};

    \draw[thick] (0,0) to [out=120, in=-60] (-.8, 1.5);
    \draw[thick] (0,0) to [out=60, in=-120] (.8, 1.5);
    \draw (0,0) -- (0.2,-.5);
    \draw (0,0) -- (-0.2,-.5);

    \draw node at  (0.3,-.7) {$1$};
    \draw node at  (-0.3,-.7) {$3$};
    \draw node at  (-0.6,0.9) {$1$};
    \draw node at  (0.6,0.9) {$3$};
  \end{tikzpicture}
  \caption{$\Gamma_{12}$}
\end{subfigure}
\caption{Vertical divisors in $\overline\cH(3,1)$}
\label{fig:vertical_graphs_mu_(31)}
\end{figure}
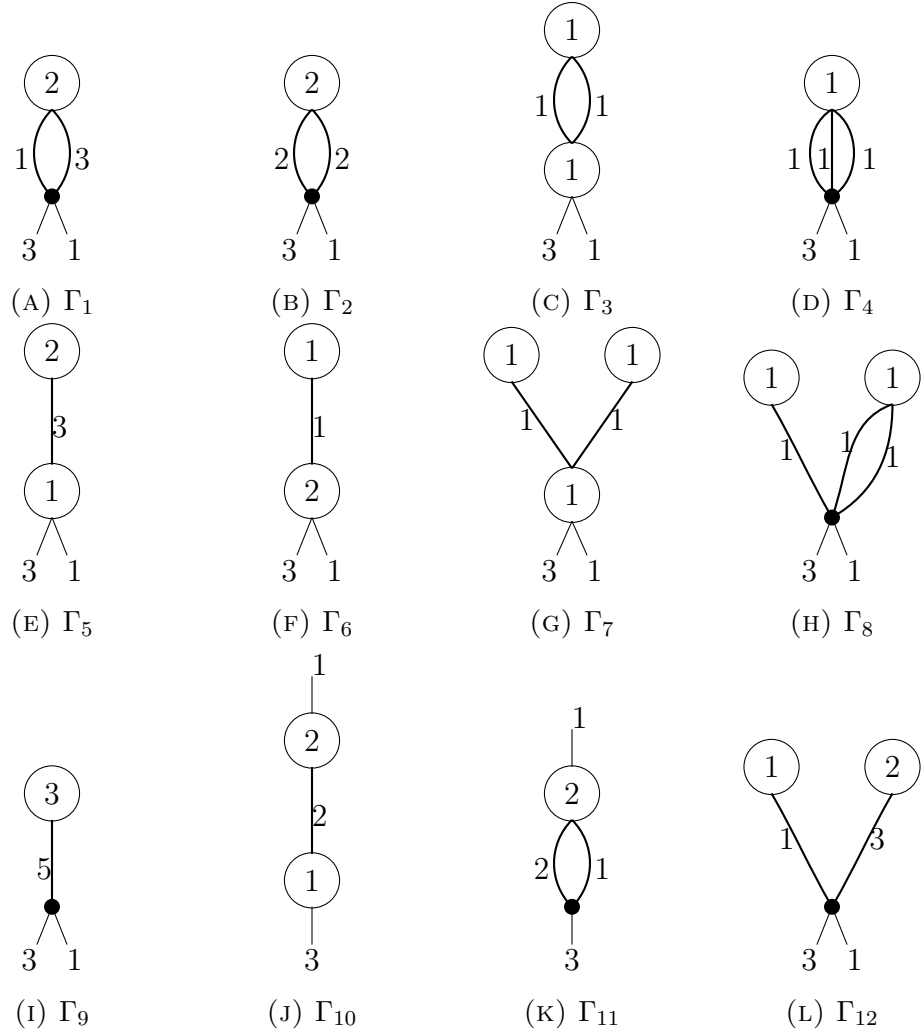
\begin{remark}
    The genus zero residueless stratum $\cR(3,1,-2,-2,-2)$ is empty. This can be seen, e.g., by writing the differential as $\omega = \frac{(z-a)(z-b)^3}{z^2(z-1)^2}dz$ and imposing the conditions of vanishing residue at the double poles 0 and $\infty$. Analyzing the resulting equations in $a,b$, we would then see that all such differentials will be contained in the boundary of $\overline\cH(3,1,-2,-2,-2)$, thereby illustrating our claim. As a result, the enhanced level graph with three elliptic tails in the top level and the stratum $\cR(3,1,-2,-2,-2)$ in the bottom level does not appear in the boundary of $\overline\cH(3,1)$.
\end{remark}
Note that $D_{\Gamma_9}$ is the closure of the natural inclusion of $\cH(4)$ in $\overline\cH(3,1)$, such that  $D_{\Gamma_i}$ for $i\le 8$ intersect $D_{\Gamma_9}$ exactly along the corresponding boundary divisors in \Cref{fig:vertical_graphs_mu_(4)}.

As in the previous section, we now discuss the level-wise product model for (the interior of) each of the boundary divisors in $\overline\cH(3,1)$.
\begin{enumerate}
\item $D_h^o\sim \cH(3,1,-1,-1)/S_2$ is the horizontal divisor. It is irreducible.
\item $D_{\Gamma_1}^o\sim \cH(2,0)\times \cH(3,1,-2,-4)$. The prong-matchings on each compatible collection of twisted differentials belong to the unique equivalence class, so $D_{\Gamma_1}$ is irreducible.
\item $D_{\Gamma_2}^o$ is related to $(\cH(1,1)\times\cH(3,1,-3,-3))/S_2$ via finite correspondence. Each comatible collection of twisted differentials has two equivalence classes of prong-matchings. However, $D_{\Gamma_2}^o$ remains connected (so the finite correspondence is non-trivial).
  \item $D_{\Gamma_3}^o\sim (\cH(0,0)\times\cH(3,1,-2,-2))/S_2$ (where $S_2\cong\text{Aut}\,(\Gamma_3)$) is irreducible. Of course, forgetting the simple zero defines an open immersion of $\cH(3,1,-2,-2)$ into $\cM_{1,3}$, so $\cH(3,1,-2,-2)$ has the CKgP.
\item $D_{\Gamma_4}^o\sim (\cH(0,0,0)\times \cH(3,1,-2,-2,-2))/S_3$ (where $S_3\cong\text{Aut}\,(\Gamma_4)$) is irreducible.
\item $D_{\Gamma_5}^o\sim \cH(2)\times \cH(3,1,-4)$ is also irreducible. Again, forgetting the simple zero defines an open immersion $\cH(3,1,-4)\subset \cM_{1,2}$.
\item $D_{\Gamma_6}^o\sim \cH(0)\times \cH(3,1,-2)$ is irreducible.
\item $D_{\Gamma_7}^o\sim (\bP(\Omega\cM_{1,1}\times\Omega\cM_{1,1})\times \cR(3,1,-2,-2))/S_2$ (where $S_2\cong\text{Aut}\,(\Gamma_7)$). It is irreducible. That $\cR(3,1,-2,-2)/S_2$ has the CKgP will be show below.
\item $D_{\Gamma_8}^o\sim \bP(\Omega\cM_{1,1}\times\Omega\cM_{1,2})\times \cH(3,1\mid-2\mid-2,-2)/S_2$ (where $S_2\cong\text{Aut}\,(\Gamma_8)$). This divisor is also irreducible. That $\cH(3,1\mid-2\mid-2,-2)$ has the CKgP will be shown below.
\item $D_{\Gamma_9}^o\sim \cH(4)\times \cH(3,1,-6)\sim \cH(4)$. Since $\cH(4)$ has two components distinguished by hyperellipticity, $D_{\Gamma_9}$ also has two components.
\item $D_{\Gamma_{10}}^o\sim \cH(1,1)\times \cH(3,-3)$ is irreducible. Of course $\cH(3,-3)$ is the classical modular curve $X_1(3)$, which is rational.
\item $D_{\Gamma_{11}}^o\sim \cH(1,1,0)\times \cH(3,-2,-3)$ is also irreducible.
\item $D_{\Gamma_{12}}^o\sim (\bP(\Omega\cM_{1,1}\times\Omega\cM_{2,1}(2))\times \cR(3,1,-2,-4))$. This divisor is irreducible. Note that $\cR(3,1,-2,-4)$ is 0-dimensional.
\end{enumerate}
\subsection*{Strata with residue conditions}
From above enumeration, we see that the only vertices of the vertical divisor that we have not yet shown to have the CKgP are $\cR(3,1,-2,-2)/S_2$ and $\cH(3,1\mid -2\mid -2,-2)$, which we do now. The latter is one-dimensional, so can be dealt with easily:
\begin{proposition}
    The stratum $\cH(3,1 \mid -2 \mid -2,-2)$ is a rational curve.     
\end{proposition}
\begin{proof}
    A differential in $\cH(3,1 \mid -2 \mid -2,-2)$ can be written as \[\omega=\frac{(z-a)(z-b)^3}{z^2(z-1)^2}dz\] for distinct $a,b\in \mathbb{P}^1\setminus\{0,\infty,1\}$. At the pole $z=0$, \[\operatorname{res}_{z=0}\omega=ab^3\left(2-\frac{1}{a}-\frac{3}{b}\right) = 0.\]

    So $\cH(3,1 \mid -2 \mid -2,-2)$ is isomorphic to a dense subset of the plane conic curve $V(2ab-b-3a)\subset \operatorname{Spec}\mathbb{C}[a,b]$. 
\end{proof}
For $\cR(3,1,-2,-2)/S_2$, it requires a bit more work. Let $\cR^{u}(3,1,-2,-2)\coloneqq \cR(3,1,-2,-2)/S_2$ be the stratum without labeling of the poles and $\overline\cR^u(3,1,-2,-2)$ its multi-scale compactification. For brevity, we will use the notation $\cR^u\coloneqq \cR^u(3,1,-2,-2)$ and $\overline\cR^u\coloneqq \overline\cR^u(3,1,-2,-2)$. We first check that the horizontal divisor $D_h\subset \overline\cR^u$ is a smooth rational curve using an argument similar to preceding proposition:
\begin{lemma}
    The horizontal divisor $D_h$ of $\overline\cR^u$ is a smooth rational curve.
\end{lemma}
\begin{proof}
    Note that $D_h$ is an image of $\overline\cH(3,1\mid -2\mid -2\mid -1,-1)/(S_2\times S_2)$, where the first factor of $S_2$ symmeterizes the two double poles and the second factor of $S_2$ symmeterizes the two simple poles. Then a differential in $\cH(3,1\mid-2\mid-2\mid-1,-1)$ can  be written as:
    \[\omega=\frac{(1-z)^3(a-z)}{z^2(b-z)(c-z)}dz,\]
    for distinct $a,b\in \bP^1\setminus\{0,\infty,1\}$. Since residue at the pole $z=0$ is trivial, we get $\frac{1}{b}+\frac{1}{c}=\frac{1}{a}+3$. Similarly, triviality of residue at the pole $z=\infty$ gives $b+c=a+3$. These two equations are symmetric in the simple poles $b$ and $c$, so forgetting the simple poles yields a one-to-one morphism $\overline\cH(3,1\mid -2\mid -2\mid -1,-1)/(S_2\times S_2)\rightarrow\overline{M}_{1,4}$.

    Consequently, based on the discussion in \cite[\S 5]{CoMoZaEU}, to prove $D_h$ is a smooth rational curve, it suffices to show that it has no self-intersection, that is, $\overline\cH(3,1\mid -2\mid -2\mid -1,-1)/(S_2\times S_2)$ does not have a horizontal degeneration. Such a degeneration can only occur if the triple zero and the two double poles collide (or equivalently, the simple zero and the two simple poles collide), or if the simple zero collides with one of the double poles. But since the double poles are residueless, this forces the simple pole manifesting in the new horizontal edge to be residueless, which is not possible.
\end{proof}
Now, we are ready to prove that $\overline\cR^u$ has the CKgP.

\begin{proposition}
    The stratum $\cR^u(3,1,-2,-2)$ is a rational surface. Moreover, it has the CKgP and its rational Chow ring generated by the boundary divisor classes. 
\end{proposition}

\begin{proof}
    The rationality implies the CKgP in this situation. We prove that the Chow ring is generated by boundary divisor classes. Consider the forgetful map $p:\overline{\cR}^u\to \overline{\mathcal{M}}_{1,1}$. We claim that the general fiber of $p$ is a smooth rational curve. 

    The fiber $F=p^{-1}(\infty)$ is over the rational nodal curve $\infty\in \overline{\mathcal{M}}_{1,1}$ is lying over a nodal curve with four rational irreducible components, denoted by $D_1,D_2,D_3$ and $D_h$.
    
    The dual graph is the chain of length four. That is, $F$ is lying over a rational nodal curve. This is not enough to prove that the general fiber of $p$ is rational, since $p$ is ramified over some components of $F$. 

    The component $D_h$ is horizontal, so not ramified over $p$ (see \cite[\S 4]{CoMoSc_LinearSubmanifolds}). The other components $D_1,D_2$ and $D_3$ correspond to the two-level graphs with two vertical edges. 

    By \cite[Prop 7.2]{CoMoZaEU}, the normal bundle $\mathcal{N}_{\Gamma}$ of a divisor $D_\Gamma$ satisfies $$c_1 (\mathcal{N}_{\Gamma})=-\frac{\kappa_e}{\ell_\Gamma}\left(\psi_{e^+}+\psi_{e^-}  \right)-\frac{1}{\ell_\Gamma}\sum_{\hat\Delta}\ell_{\hat\Delta,a_{\hat\Delta,\Gamma}}[D_{\hat\Delta}]$$

    where $e$ is any edge of $\Gamma$ the sum is taken over all 3-level graph $\hat\Delta$ where the edge $e$ becomes long. That is, $\hat\Delta$ is a degeneration of $\Gamma$ where $e$ becomes an edge joining the level 0 and level -2. Note that while \cite[Prop. 7.2]{CoMoZaEU} is proved for the strata of differentials without residue conditions, by the discussion of \cite{CoMoSc_LinearSubmanifolds} (especially Proposition 3.2 and Proposition 4.4 there), it also holds boundary divisors in residueless strata.

    For each $D_i$, $\psi_{e^+}\neq 0$ or $\psi_{e^-}\neq 0$. That is, $D_i^2 =\deg \mathcal{N}_{D_i}<0$ and each $D_i$ is an exceptional curve on $\overline\cR^u$. So the map $p$ factors through $$q: \cR'\to \overline{\mathcal{M}}_{1,1}$$ where $\cR'$ is obtained by contracting $D_1,D_2$ and $D_3$ from $\cR^u(3,1,-2,-2)$. Note that $q^{-1}(\infty)\cong D_h$, so the general fiber of $q$ is a smooth rational curve. So is the general fiber of $p$.

    Let us now denote by $p_0:\cR^u\to \mathcal{M}_{1,1}$ the restriction of $p$ to the interior. We now claim that the fiber $p_0^{-1}([E])$ is irreducible for any elliptic curve $E\in \mathcal{M}_{1,1}$. This would then imply $p|_{p^{-1}(\cM_{1,1})}:p^{-1}(\cM_{1,1})\rightarrow \cM_{1,1}$ is an honest $\mathbb{P}^1$ fibration over $\cM_{1,1}$. Since the Chow ring of $\cM_{1,1}$ is generated by the fundamental class, this would then imply the Chow ring of $\overline\cR^u$ is generated by the boundary divisors.
    
    Let $u,v\in E$ be distinct points. The space of meromorphic functions on $E$ with at most double poles at $t$ and $s$ has dimension $h^0(\mathcal{O}_E(2t+2s))=1+4-1=4$ by Riemann-Roch. The residueless condition on one pole makes it automatically double pole and also enforces the residueless condition on another pole. So the space of residueless functions at double poles at $t$ and $s$ has dimension three. We may assume that $t=0$.    
    
    Let $\wp(z)$ be the Weierstrass $\wp$-function on $E$. It is a residueless meromorphic function with only double pole at $0\in E$. Thus the space $W_s$ of residueless meromorphic functions with double poles at $0,s\in E$ is generated by $1, \wp(z)$ and $\wp(z-s)$. Let $F(z)=a+b\wp(z)+c\wp(z-s)$ be any such meromorphic function. We need to find a condition that $F$ has a zero of order three (at some $q\in E$ different from $u,s$). That is, $F(q)=F'(q)=F''(q)=0$. Equivalently, we have $b\wp'(q)+c\wp'(q-s) = b\wp''(q)+c\wp''(q-s)=0$. That is, $\frac{\wp'(q)}{\wp'(q-s)}=\frac{\wp''(q)}{\wp''(q-s)}$ and $$\frac{\wp''(q)}{\wp'(q)} = \frac{\wp''(q-s)}{\wp'(q-s)}.$$
    
    The curve $C\subset E^2$ consisting of such $(s,q)$ is the closure of a variety defined by the equations $m(q)=m(q-s)$ and $q,s,0$ distinct, where $m:E\to \mathbb{P}^1$ is a meromorphic function defined by $$m(z)=\frac{\wp''(z)}{\wp'(z)}.$$ The function $m$ has four simple poles at each point of $E[2]$, so $\deg m=4$. 
    
    The map $(q,s)\mapsto (q,q-s)$ is an automorphism of $E^2$, sending $C$ to the closure $C'$ of the set of points $(u,v)\in E^2$ satisfying $m(u)=m(v)$, $0,u,v$ are distinct. Since $\deg m=4$, the degree of the projection to first component $\pi_1: C'\to E$, $(u,v)\mapsto u$ is three. Assume the contrary that $C'$ is reducible and there exists a component $R$ of $C'$ such that $\deg \pi_1|_R =1$. That is, there exists a birational map $\phi:E\dashrightarrow R$. Then $\Phi=\pi_2\circ \phi : E\dashrightarrow E$ is extends to an automorphism of $E$ satisfying $m(z)=m(\Phi(z))$. The automorphism $\Phi$ preserves $E[2]$ and fixes 0, since $m$ has simple poles at each point of $E[2]$ and $\operatorname{res}_{z=0}m(z)=-3$. Thus $\Phi=\operatorname{Id}_E$ and $R$ is the diagonal $\Delta\in E^2$. However, $C'$ does not contain $\Delta$, a contradiction. Therefore $C$ is irreducible. 

\begin{comment}
    Another proof of rationality:
Now we need to compute $g(C)$. We need another construction of $C$ here. For fixed $q$, we can consider a linear map $j^s_q:W_s\to \mathbb{C}^3$ defined by $f\mapsto  (f(q),f'(q),f''(q))$. If $q$ is a zero of order three of $f$, then $j^s_q(f)=0$. That is, $f\in \ker j_q$. Thus $C$ is contained in the subvariety defined by $\operatorname{rank} j^s_q <3$ that is, $\det j^s_q(z)=\wp'(z)\wp''(z-s)-\wp'(z-s)\wp''(z)=0$ on $C$. So $C$ is contained in the determinantal subvariety of $E^2$. 

Let $L=\mathcal{O}_{E^2}(2H+2\Delta)$. Then $c_1({\pi_1}_* L) ={\pi_1}_* (\frac{1}{2}c_1(L)^2)$ by Grothendick-Riemann-Roch. Since $c_1(L)=2[H]+2[\Delta]$, we have $c_1(L)^2 = 8H.\Delta$. So $c_1({\pi_1}_* L)=4[pt]$.

Consider the defining sequence of $W_s$ $$0\to W_s\to {\pi_1} _* L\to \mathcal{O}_E(2t)\to 0$$ as a kernel of residue map, where $H=E\times \{pt\}$ and $\Delta=\{(u,v)|u=v\}$. From this, we have $c_1(W_s)=c_1({\pi_1} _*L)-2[pt]=2[pt]$. 

Now consider the map $j^2:\pi^*_1W_s\to J^2(L)$ between the rank 3 bundles on $E^2$. The curve $C$ is contained in the determinantal subvariety $X$ defined by $\det j^2=0$. Since $J^2(L)\cong L^{\oplus3}$, we have $c_1(J^2(L))=3c_1(L)=6[H]+6[\Delta]$. By Porteous formula, $[X]=c_1(J^2(L))-\pi_1^* c_1(W_s)=6[H]+6[\Delta]-2[V]$ where $V=\{pt\}\times E$. However, $X$ also contains $H$ and $\Delta$, so we have $[C]=5[H]+5[\Delta]-2[V]$ and $[C]^2=10$. By adjunction formula, we have $g(C)=1+\frac{[C]^2}{2}=6$. \Myeongjae{Incomplete}
\end{comment}
\end{proof}

\subsection*{Certifying CKgP}
Except for $i=2$, it is clear that $D_{\Gamma_i}^o$ has the CKgP with $CH^*(D_{\Gamma_i}^o)\cong \mathbb{Q}$ for all other $i$. We will now show the same holds also for $i=2$:
\begin{lemma}
  The open subset $D_{\Gamma_2}^o$ of the boundary divisor $D_{\Gamma_2}$ has the CKgP and $CH^*(D_{\Gamma_2}^o)\cong \mathbb{Q}$ generated by the fundamental class.
\end{lemma}
\begin{proof}
  Put $B_{\Gamma_2}\coloneqq \overline\cH(1,1)\times \overline\cH(3,1,-3,-3)$ the product of the strata corresponding to the two levels of $\Gamma_2$. We denote by $B_{\Gamma_2}^o$ (resp. $D_{\Gamma_2}^o$) the open subset of $B_{\Gamma_2}$ (resp. $D_{\Gamma_2}$) obtained  as the complement of the locus where the enhanced level graph degenerates further. Then as discussed in \Cref{sec:boundary_correspondence}, there is a stack $U_{\Gamma_2}^o$ that admits finite morphisms $p_{\Gamma_2}^o:U_{\Gamma_2}^o\rightarrow B_{\Gamma_2}^o$ and $c_{\Gamma_2}^o:U_{\Gamma_2}^o\rightarrow B_{\Gamma_2}^o$, where the former forgets the prong-matching equivalence classes (so it is of degree 2) and the latter is the quotient by $\text{Aut}(\Gamma_2)\cong S_2$. Then, it is enough to show that $U_{\Gamma_2}^o$ has the CKgP.

  First we claim that $p_{\Gamma_2}^o$ can be extended to a finite 2-to-1 morphism $U_{\Gamma_2}^*\rightarrow \overline\cH(1,1)\times \cH(3,1,-3,-3)\subset B_{\Gamma_2}$. For each degeneration $\Delta$ of the top level of $\Gamma_2$, we denote by $B_{\Gamma_2,\Delta}$ (resp. $D_{\Gamma_2,\Delta}$) the open subset of $B_{\Gamma_2}$ (resp. $D_{\Gamma_2}$) parameterizing the differentials compatible with an undegeneration of $\Delta$. Then again by the discussion of \Cref{sec:boundary_correspondence}, there is a stack $U_{\Gamma_2,\Delta}$ that admits finite morphisms $p_{\Gamma_2,\Delta}:U_{\Gamma_2,\Delta}\rightarrow B_{\Gamma_2,\Delta}$ and $c_{\Gamma_2,\Delta}:U_{\Gamma_2,\Delta}\rightarrow D_{\Gamma_2,\Delta}$. Note that the vertical boundary divisors of $\overline\cH(1,1)$ are given by one of the following enhanced level graphs:
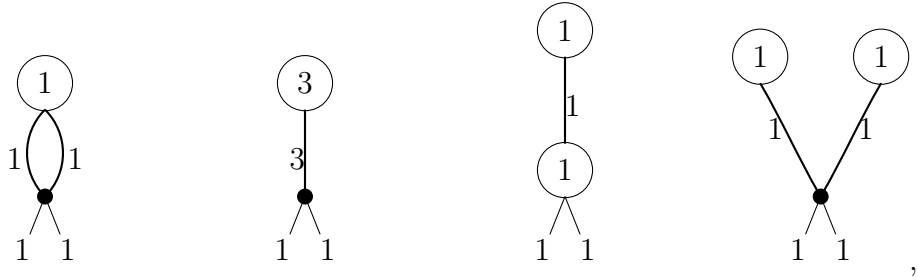
\begin{figure}[H]
  \centering
  \begin{subfigure}{0.2\textwidth}
    \centering
\begin{tikzpicture}[main/.style = {draw, circle}]
    \node[main] at (0,0) {1};
    %\fill[black] (0,0) circle (3pt);
    \fill[black] (0, -1.5) circle (3pt);

    \draw[thick] (0,-.35) to [out=225, in=135] (0, -1.5);
    \draw[thick] (0,-.35) to [out=315, in=45] (0, -1.5);
    \draw (0,-1.5) -- (0.2,-2);
    \draw (0,-1.5) -- (-0.2,-2);

    \draw node at  (0.3,-2.2) {$1$};
    \draw node at  (-0.3,-2.2) {$1$};
    \draw node at  (-0.4,-1) {$1$};
    \draw node at  (0.4,-1) {$1$};
  \end{tikzpicture}
\end{subfigure}
\begin{subfigure}{0.2\textwidth}
    \centering
\begin{tikzpicture}[main/.style = {draw, circle}]
    \node[main] at (0,0) {3};
    %\fill[black] (0,0) circle (3pt);
    \fill[black] (0, -1.5) circle (3pt);

    \draw[thick] (0,-.35) to [out=270, in=90] (0, -1.5);
    \draw (0,-1.5) -- (0.2,-2);
    \draw (0,-1.5) -- (-0.2,-2);

    \draw node at  (0.3,-2.2) {$1$};
    \draw node at  (-0.3,-2.2) {$1$};
    \draw node at  (-0.1,-1) {$3$};
  \end{tikzpicture}
\end{subfigure}
\begin{subfigure}{0.2\textwidth}
    \centering
\begin{tikzpicture}[main/.style = {draw, circle}]
    \node[main] at (0,0) {1};
    %\fill[black] (0,0) circle (3pt);
    \node[main] at (0,-1.85) {1};

    \draw[thick] (0,-.35) to [out=270, in=90] (0, -1.5);
    \draw (0,-2.2) -- (0.2,-2.7);
    \draw (0,-2.2) -- (-0.2,-2.7);

    \draw node at  (0.3,-2.9) {$1$};
    \draw node at  (-0.3,-2.9) {$1$};
    \draw node at  (0.1,-1) {$1$};
  \end{tikzpicture}
\end{subfigure}
\begin{subfigure}{0.2\textwidth}
    \centering
\begin{tikzpicture}[main/.style = {draw, circle}]
    %\node[main] at (0,0) {1};
    \fill[black] (0,0) circle (3pt);
    \node[main] at (-.8,1.85) {1};
    \node[main] at (.8,1.85) {1};

    \draw[thick] (0,0) to [out=120, in=-60] (-.8, 1.5);
    \draw[thick] (0,0) to [out=60, in=-120] (.8, 1.5);
    \draw (0,0) -- (0.2,-.5);
    \draw (0,0) -- (-0.2,-.5);

    \draw node at  (0.3,-.7) {$1$};
    \draw node at  (-0.3,-.7) {$1$};
    \draw node at  (-0.6,0.9) {$1$};
    \draw node at  (0.6,0.9) {$1$};
  \end{tikzpicture},
\end{subfigure}
\end{figure}
all of which have the unique prong-matching equivalence class. Note that only the second level graphs has one edge with an enhancement greater than 2, so for any intersection of these boundary divisors, we will also have unique prong-matching equivalence class. Consequently, for any degeneration $\Delta$ of the top level of $\Gamma_2$, the compatible twisted differentials will continue to have two prong matching equivalence classes coming from the bottom level passage. As a result, the morphism $p_{\Gamma_2,\Delta}$ is still 2-to-1 and restricts exactly to $p_{\Gamma_2}^o$ over $B_{\Gamma_2}^o\subset B_{\Gamma_2,\Delta}$. Thus, $p_{\Gamma_2,\Delta}$'s glue to give a 2-to-1 morphism $U_{\Gamma_2}^*\rightarrow \overline\cH(1,1)\times \cH(3,1,-3,-3)$, which we denote by $p_{\Gamma_2}$.

  Then the projection $U_{\Gamma_2}^*\rightarrow \cH(3,1,-3,-3)$ is a proper morphism whose any of its fiber has two connected components by \cite[Thm. 1.3]{boissy_rauzy}. If $U_{\Gamma_2}^*\rightarrow C \rightarrow \cH(3,1,-3,-3)$ is the Stein factorization then $C\rightarrow \cH(3,1,-3,-3)\cong M_{0,4}$ is a 2-to-1 unramified covering. Since the branching can be along at most three boundary points of $M_{0,4}$ (in fact, branching occurs only along the two points in the boundary where one of the zeros collides with one of the nodes/poles), by Riemann-Hurwitz, the cover is again a rational curve (i.e. $C$ is again an open subset of $\mathbb{P}^1$). So, by the universal property of product, $U_{\Gamma_2}^*\cong \overline\cH(1,1)\times C$, and thus $D_{\Gamma_2}^o$ is a finite image of $\cH(1,1)\times C$ which has CKgP, with the Chow ring isomorphic to $\mathbb{Q}$ generated by the fundamental class.
\end{proof}
In fact even though the above claim is stated for $D_{\Gamma_2}^o$, the proof implies that $\overline\cH(1,1)\times \mathbb{P}^1$ dominates the closure $D_{\Gamma_2}$, so it also has the CKgP with the Chow ring generated by the boundary strata.

With the only boundary divisor $D_{\Gamma_2}$ having non-trivial prong-matching equivalence classes out of the way, the rest of the boundary strata of $\overline\cH(3,1)$ can be dealt with analogously to the $\mu=(4)$ case in the previous section. Indeed the boundary divisors $D_{\Gamma_i}$ for $i=7,8,12$ can be shown to have CKgP with the Chow ring generated by the boundary strata via the application of \Cref{thm:ckgp_fibration}. And for the remaining vertical boundary strata, they will be finitely covered by the product of the generalized strata for each levels.

As for the codimension two vertical boundary strata, they can be enumerated by the 3-level vertical graphs. The entire list of graphs are depicted in Figure~\ref{fig:h4} and Figure~\ref{fig:h13}. Every codimension two boundary strata $D(\Gamma)$ has CKpG. See Appendix~\ref{appendix} for the detail. 

The horizontal boundary divisor $D_h$ of $\overline\cH(3,1)$, on the other hand, is finitely covered by $\overline\cH(3,1,-1,-1)/S_2$, which, in turn, admits a birational morphism onto $\overline{\mathcal{M}}_{2,2}$. Thus, we can prove $D_h$ also has the CKgP via similar argument to $\mu=(4)$ case, by analyzing the exceptional divisors of this birational morphism. Thus, we have:
\begin{proposition}
  The compactified stratum $\overline\cH(3,1)$ has the CKgP and the cohomology (as well as the Chow) ring is tautological, generated by the boundary strata.
\end{proposition}
This concludes the proof of \Cref{thm:main_thm} and \Cref{cor:main_cor}.

\newpage
\appendix
\section{Vertical 3-level graphs}\label[appendix]{appendix}

In this appendix, we will list the vertical 3-level graphs corresponding to codimension two boundary strata of $\cR(\mu)$, $\mu=(4),(3,1),(2,2)$. In general, we can have more than one 3-level graphs degenerating the pair $\Gamma_i,\Gamma_j$. But for $\mu=(4),(3,1),(2,2)$, we only have one for each pair. So we denote $\Gamma_{i,j}$ the 3-level vertical graph obtained by degenerating $\Gamma_i$ and $\Gamma_j$.

\tikzset{
  vgraph/.style={
    x=1cm,
    y=1cm,
    every node/.style={font=\small},
    line cap=round,
    line join=round
  },
  component/.style={
    draw,
    circle,
    fill=white,
    minimum size=6.8mm,
    inner sep=0pt,
    font=\small,
    line width=0.7pt
  },
  blacknode/.style={circle,fill=black,inner sep=1.8pt},
  opennode/.style={
    circle,
    draw=black,
    fill=black,
    minimum size=2.2mm,
    inner sep=0pt
  },
  edgelabel/.style={font=\scriptsize,fill=white,inner sep=0.8pt},
  halfedgelabel/.style={font=\scriptsize,fill=white,inner sep=0.8pt},
  leglabel/.style={font=\scriptsize,fill=white,inner sep=0.6pt},
  edge/.style={draw,line width=0.75pt}
}

% Every graph uses the same three heights:
% top level y=0, middle level y=-1.05, bottom level y=-2.10.
% The fixed bounding box keeps all subfigures and captions aligned.
\newcommand{\vgbbox}{%
  \path[use as bounding box] (-1.12,0.42) rectangle (1.12,-2.94);%
}

\newcommand{\sgref}[2]{\textnormal{(\ref{sg:#1:#2})}}
% Explicit visible graph labels, independent of the subfigure counter.
% Examples: \sgcaption{A}{\Gamma_{1,3}},
%           \sgcaption{a}{\Gamma_{5,10}},
%           \sgcaption{Q}{\Gamma_{7,13}}.
\newcommand{\sgtag}[1]{\textnormal{(#1)}}
\newcommand{\sgcaption}[2]{%
  \caption*{\sgtag{#1}\enspace $#2$}%
}

% ------------------------------------------------------------------
% H(4), codimension 2
% Visible labels: (A),...,(P)
% All vertices lie on the common heights y=0,-1.05,-2.10.
% ------------------------------------------------------------------
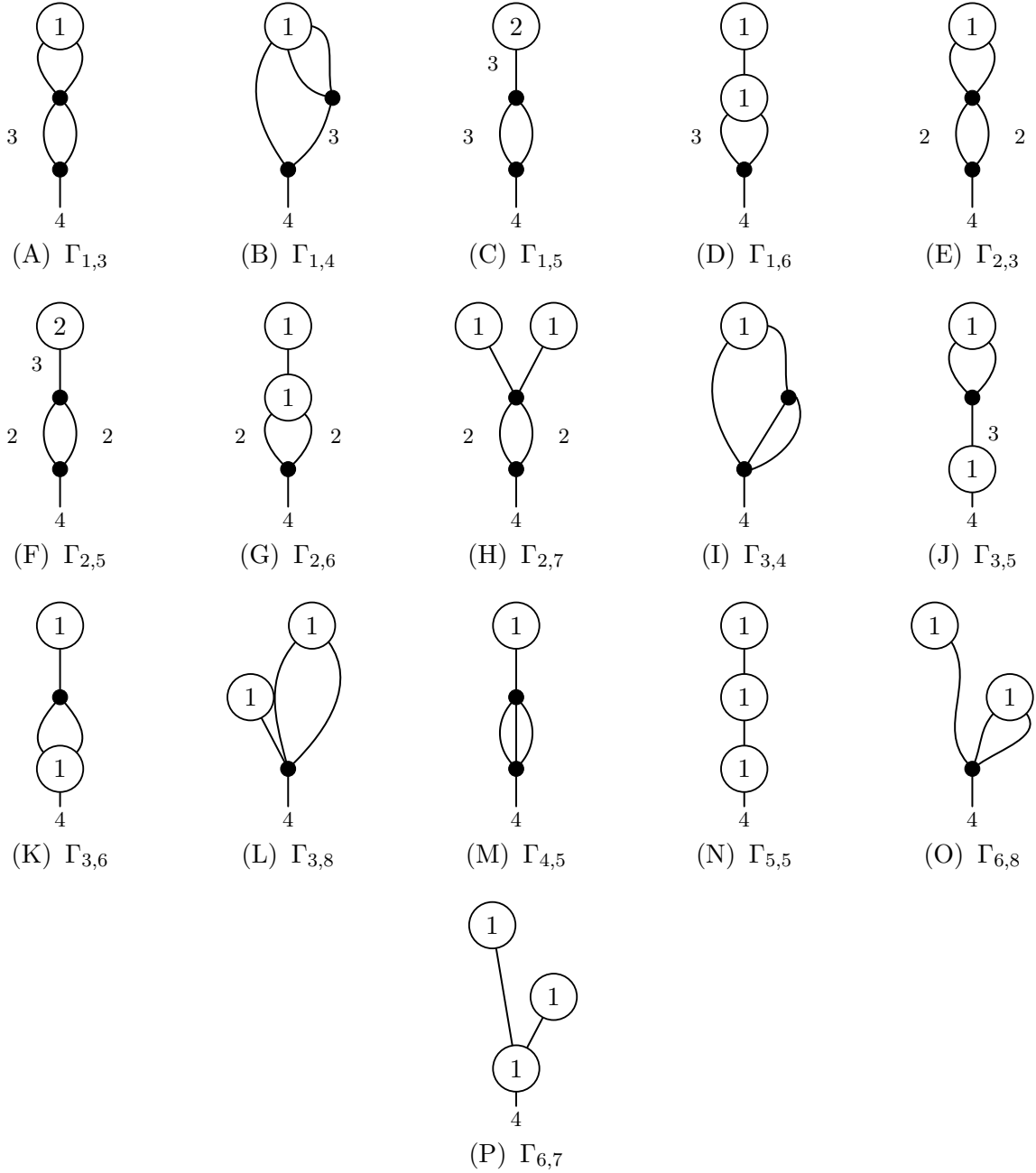
\begin{figure}[H]
\centering
\begin{subfigure}[t]{0.19\textwidth}\centering
\begin{tikzpicture}[vgraph]
  \vgbbox
  \node[component] (T) at (0,0) {1};
  \node[opennode] (M) at (0,-1.05) {};
  \node[opennode] (B) at (0,-2.10) {};
  \draw[edge] (T.south west) to[out=230,in=130] (M);
  \draw[edge] (T.south east) to[out=310,in=50] (M);
  \draw[edge] (M) to[out=230,in=130] (B);
  \draw[edge] (M) to[out=310,in=50] (B);
  \draw[edge] (B)--(0,-2.65);
  \node[edgelabel] at (-0.70,-1.62) {$3$};
  \node[leglabel] at (0,-2.84) {$4$};
\end{tikzpicture}
\sgcaption{A}{\Gamma_{1,3}}
\end{subfigure}\hfill
\begin{subfigure}[t]{0.19\textwidth}\centering
\begin{tikzpicture}[vgraph]
  \vgbbox
  \node[component] (T) at (0,0) {1};
  \node[opennode] (R) at (0.65,-1.05) {};
  \node[opennode] (B) at (0,-2.10) {};
  \draw[edge] (T.south west) to[out=235,in=125] (B);
  \draw[edge] (T.south) to[out=285,in=165] (R);
  \draw[edge] (T.east) to[out=-15,in=100] (R);
  \draw[edge] (R) to[out=-105,in=40] (B);
  \draw[edge] (B)--(0,-2.65);
  \node[edgelabel] at (0.67,-1.62) {$3$};
  \node[leglabel] at (0,-2.84) {$4$};
\end{tikzpicture}
\sgcaption{B}{\Gamma_{1,4}}
\end{subfigure}\hfill
\begin{subfigure}[t]{0.19\textwidth}\centering
\begin{tikzpicture}[vgraph]
  \vgbbox
  \node[component] (T) at (0,0) {2};
  \node[opennode] (M) at (0,-1.05) {};
  \node[opennode] (B) at (0,-2.10) {};
  \draw[edge] (T)--(M);
  \draw[edge] (M) to[out=230,in=130] (B);
  \draw[edge] (M) to[out=310,in=50] (B);
  \draw[edge] (B)--(0,-2.65);
  \node[edgelabel] at (-0.34,-0.55) {$3$};
  \node[edgelabel] at (-0.70,-1.62) {$3$};
  \node[leglabel] at (0,-2.84) {$4$};
\end{tikzpicture}
\sgcaption{C}{\Gamma_{1,5}}
\end{subfigure}\hfill
\begin{subfigure}[t]{0.19\textwidth}\centering
\begin{tikzpicture}[vgraph]
  \vgbbox
  \node[component] (T) at (0,0) {1};
  \node[component] (M) at (0,-1.05) {1};
  \node[opennode] (B) at (0,-2.10) {};
  \draw[edge] (T)--(M);
  \draw[edge] (M.south west) to[out=225,in=135] (B);
  \draw[edge] (M.south east) to[out=315,in=45] (B);
  \draw[edge] (B)--(0,-2.65);
  \node[edgelabel] at (-0.70,-1.62) {$3$};
  \node[leglabel] at (0,-2.84) {$4$};
\end{tikzpicture}
\sgcaption{D}{\Gamma_{1,6}}
\end{subfigure}\hfill
\begin{subfigure}[t]{0.19\textwidth}\centering
\begin{tikzpicture}[vgraph]
  \vgbbox
  \node[component] (T) at (0,0) {1};
  \node[opennode] (M) at (0,-1.05) {};
  \node[opennode] (B) at (0,-2.10) {};
  \draw[edge] (T.south west) to[out=230,in=130] (M);
  \draw[edge] (T.south east) to[out=310,in=50] (M);
  \draw[edge] (M) to[out=230,in=130] (B);
  \draw[edge] (M) to[out=310,in=50] (B);
  \draw[edge] (B)--(0,-2.65);
  \node[edgelabel] at (-0.70,-1.62) {$2$};
  \node[edgelabel] at (0.70,-1.62) {$2$};
  \node[leglabel] at (0,-2.84) {$4$};
\end{tikzpicture}
\sgcaption{E}{\Gamma_{2,3}}
\end{subfigure}

\par\vspace{0.8em}

\begin{subfigure}[t]{0.19\textwidth}\centering
\begin{tikzpicture}[vgraph]
  \vgbbox
  \node[component] (T) at (0,0) {2};
  \node[opennode] (M) at (0,-1.05) {};
  \node[opennode] (B) at (0,-2.10) {};
  \draw[edge] (T)--(M);
  \draw[edge] (M) to[out=230,in=130] (B);
  \draw[edge] (M) to[out=310,in=50] (B);
  \draw[edge] (B)--(0,-2.65);
  \node[edgelabel] at (-0.34,-0.55) {$3$};
  \node[edgelabel] at (-0.70,-1.62) {$2$};
  \node[edgelabel] at (0.70,-1.62) {$2$};
  \node[leglabel] at (0,-2.84) {$4$};
\end{tikzpicture}
\sgcaption{F}{\Gamma_{2,5}}
\end{subfigure}\hfill
\begin{subfigure}[t]{0.19\textwidth}\centering
\begin{tikzpicture}[vgraph]
  \vgbbox
  \node[component] (T) at (0,0) {1};
  \node[component] (M) at (0,-1.05) {1};
  \node[opennode] (B) at (0,-2.10) {};
  \draw[edge] (T)--(M);
  \draw[edge] (M.south west) to[out=225,in=135] (B);
  \draw[edge] (M.south east) to[out=315,in=45] (B);
  \draw[edge] (B)--(0,-2.65);
  \node[edgelabel] at (-0.70,-1.62) {$2$};
  \node[edgelabel] at (0.70,-1.62) {$2$};
  \node[leglabel] at (0,-2.84) {$4$};
\end{tikzpicture}
\sgcaption{G}{\Gamma_{2,6}}
\end{subfigure}\hfill
\begin{subfigure}[t]{0.19\textwidth}\centering
\begin{tikzpicture}[vgraph]
  \vgbbox
  \node[component] (L) at (-0.55,0) {1};
  \node[component] (R) at (0.55,0) {1};
  \node[opennode] (M) at (0,-1.05) {};
  \node[opennode] (B) at (0,-2.10) {};
  \draw[edge] (L)--(M);
  \draw[edge] (R)--(M);
  \draw[edge] (M) to[out=230,in=130] (B);
  \draw[edge] (M) to[out=310,in=50] (B);
  \draw[edge] (B)--(0,-2.65);
  \node[edgelabel] at (-0.70,-1.62) {$2$};
  \node[edgelabel] at (0.70,-1.62) {$2$};
  \node[leglabel] at (0,-2.84) {$4$};
\end{tikzpicture}
\sgcaption{H}{\Gamma_{2,7}}
\end{subfigure}\hfill
\begin{subfigure}[t]{0.19\textwidth}\centering
\begin{tikzpicture}[vgraph]
  \vgbbox
  \node[component] (T) at (0,0) {1};
  \node[opennode] (R) at (0.65,-1.05) {};
  \node[opennode] (B) at (0,-2.10) {};
  \draw[edge] (T.south west) to[out=235,in=125] (B);
  \draw[edge] (T.east) to[out=-15,in=100] (R);
  \draw[edge] (R)--(B);
  \draw[edge] (R.east) to[out=-65,in=20] (B.east);
  \draw[edge] (B)--(0,-2.65);
  \node[leglabel] at (0,-2.84) {$4$};
\end{tikzpicture}
\sgcaption{I}{\Gamma_{3,4}}
\end{subfigure}\hfill
\begin{subfigure}[t]{0.19\textwidth}\centering
\begin{tikzpicture}[vgraph]
  \vgbbox
  \node[component] (T) at (0,0) {1};
  \node[opennode] (M) at (0,-1.05) {};
  \node[component] (B) at (0,-2.10) {1};
  \draw[edge] (T.south west) to[out=225,in=135] (M);
  \draw[edge] (T.south east) to[out=315,in=45] (M);
  \draw[edge] (M)--(B);
  \draw[edge] (B.south)--(0,-2.65);
  \node[edgelabel] at (0.31,-1.57) {$3$};
  \node[leglabel] at (0,-2.84) {$4$};
\end{tikzpicture}
\sgcaption{J}{\Gamma_{3,5}}
\end{subfigure}

\par\vspace{0.8em}

\begin{subfigure}[t]{0.19\textwidth}\centering
\begin{tikzpicture}[vgraph]
  \vgbbox
  \node[component] (T) at (0,0) {1};
  \node[opennode] (M) at (0,-1.05) {};
  \node[component] (B) at (0,-2.10) {1};
  \draw[edge] (T)--(M);
  \draw[edge] (M) to[out=230,in=130] (B.north west);
  \draw[edge] (M) to[out=310,in=50] (B.north east);
  \draw[edge] (B.south)--(0,-2.65);
  \node[leglabel] at (0,-2.84) {$4$};
\end{tikzpicture}
\sgcaption{K}{\Gamma_{3,6}}
\end{subfigure}\hfill
\begin{subfigure}[t]{0.19\textwidth}\centering
\begin{tikzpicture}[vgraph]
  \vgbbox
  \node[component] (T) at (0.35,0) {1};
  \node[component] (L) at (-0.55,-1.05) {1};
  \node[opennode] (B) at (0,-2.10) {};
  \draw[edge] (L)--(B);
  \draw[edge] (T.south west) to[out=230,in=105] (B);
  \draw[edge] (T.south east) to[out=305,in=45] (B);
  \draw[edge] (B)--(0,-2.65);
  \node[leglabel] at (0,-2.84) {$4$};
\end{tikzpicture}
\sgcaption{L}{\Gamma_{3,8}}
\end{subfigure}\hfill
\begin{subfigure}[t]{0.19\textwidth}\centering
\begin{tikzpicture}[vgraph]
  \vgbbox
  \node[component] (T) at (0,0) {1};
  \node[opennode] (M) at (0,-1.05) {};
  \node[opennode] (B) at (0,-2.10) {};
  \draw[edge] (T)--(M);
  \draw[edge] (M) to[out=230,in=130] (B);
  \draw[edge] (M)--(B);
  \draw[edge] (M) to[out=310,in=50] (B);
  \draw[edge] (B)--(0,-2.65);
  \node[leglabel] at (0,-2.84) {$4$};
\end{tikzpicture}
\sgcaption{M}{\Gamma_{4,5}}
\end{subfigure}\hfill
\begin{subfigure}[t]{0.19\textwidth}\centering
\begin{tikzpicture}[vgraph]
  \vgbbox
  \node[component] (T) at (0,0) {1};
  \node[component] (M) at (0,-1.05) {1};
  \node[component] (B) at (0,-2.10) {1};
  \draw[edge] (T)--(M)--(B);
  \draw[edge] (B.south)--(0,-2.65);
  \node[leglabel] at (0,-2.84) {$4$};
\end{tikzpicture}
\sgcaption{N}{\Gamma_{5,5}}
\end{subfigure}\hfill
\begin{subfigure}[t]{0.19\textwidth}\centering
\begin{tikzpicture}[vgraph]
  \vgbbox
  \node[component] (L) at (-0.55,0) {1};
  \node[component] (R) at (0.55,-1.05) {1};
  \node[opennode] (B) at (0,-2.10) {};
  \draw[edge] (L.south east) to[out=300,in=130] (B);
  \draw[edge] (R.south west) to[out=230,in=70] (B);
  \draw[edge] (R.south east) to[out=305,in=35] (B);
  \draw[edge] (B)--(0,-2.65);
  \node[leglabel] at (0,-2.84) {$4$};
\end{tikzpicture}
\sgcaption{O}{\Gamma_{6,8}}
\end{subfigure}

\par\vspace{0.8em}

\begin{subfigure}[t]{0.19\textwidth}\centering
\begin{tikzpicture}[vgraph]
  \vgbbox
  \node[component] (T) at (-0.35,0) {1};
  \node[component] (R) at (0.55,-1.05) {1};
  \node[component] (B) at (0,-2.10) {1};
  \draw[edge] (T)--(B);
  \draw[edge] (R)--(B);
  \draw[edge] (B.south)--(0,-2.65);
  \node[leglabel] at (0,-2.84) {$4$};
\end{tikzpicture}
\sgcaption{P}{\Gamma_{6,7}}
\end{subfigure}

\caption{3-level graphs in $\overline{\mathcal H}(4)$.}

\label{fig:h4}
\end{figure}
\newpage

\newpage
% ------------------------------------------------------------------
% H(3,1), codimension 2
% Visible labels: (a),...,(p)
% Half-edges are labeled a_2 and a_1.  When they occur on different
% levels, the higher half-edge is a_2.
% The selected middle-to-bottom edge is labeled 2-a_2.
% ------------------------------------------------------------------
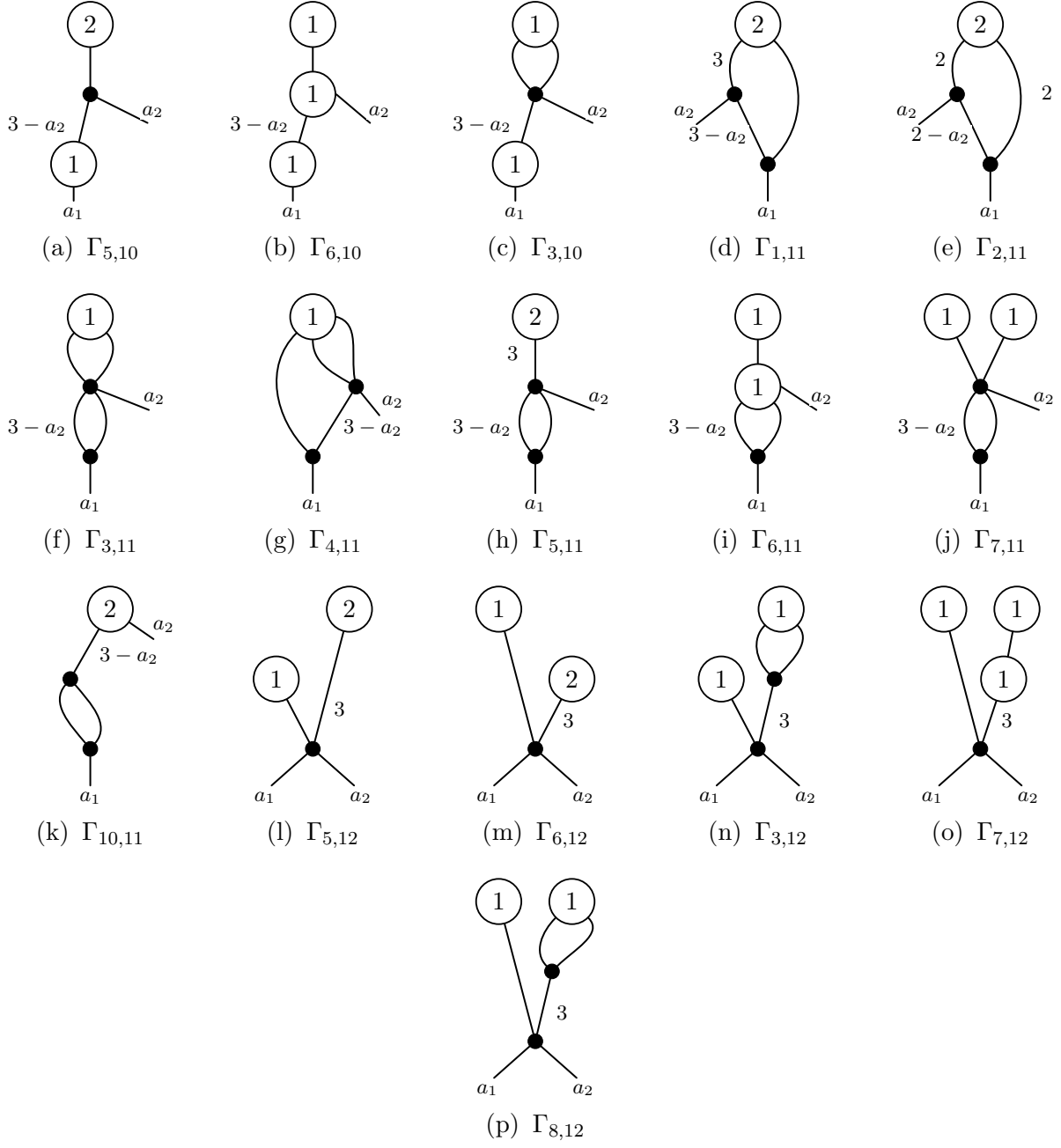
\begin{figure}[H]
\centering

\begin{subfigure}[t]{0.19\textwidth}\centering
\begin{tikzpicture}[vgraph]
  \vgbbox
  \node[component] (T) at (0,0) {2};
  \node[opennode] (M) at (0,-1.05) {};
  \node[component] (B) at (-0.25,-2.10) {1};
  \draw[edge] (T)--(M);
  \draw[edge] (M)--(B);
  \draw[edge] (M)--(0.86,-1.48);
  \draw[edge] (B.south)--(-0.25,-2.65);
  \node[halfedgelabel] at (0.94,-1.31) {$a_2$};
  \node[halfedgelabel] at (-0.25,-2.84) {$a_1$};
  \node[edgelabel] at (-0.8,-1.5) {$3-a_2$};
\end{tikzpicture}
\sgcaption{a}{\Gamma_{5,10}}
\end{subfigure}\hfill
\begin{subfigure}[t]{0.19\textwidth}\centering
\begin{tikzpicture}[vgraph]
  \vgbbox
  \node[component] (T) at (0,0) {1};
  \node[component] (M) at (0,-1.05) {1};
  \node[component] (B) at (-0.30,-2.10) {1};
  \draw[edge] (T)--(M);
  \draw[edge] (M)--(B);
  \draw[edge] (M.east)--(0.86,-1.48);
  \draw[edge] (B.south)--(-0.30,-2.65);
  \node[halfedgelabel] at (1,-1.31) {$a_2$};
  \node[halfedgelabel] at (-0.30,-2.84) {$a_1$};
  \node[edgelabel] at (-0.8,-1.5) {$3-a_2$};
\end{tikzpicture}
\sgcaption{b}{\Gamma_{6,10}}
\end{subfigure}\hfill
\begin{subfigure}[t]{0.19\textwidth}\centering
\begin{tikzpicture}[vgraph]
  \vgbbox
  \node[component] (T) at (0,0) {1};
  \node[opennode] (M) at (0,-1.05) {};
  \node[component] (B) at (-0.30,-2.10) {1};
  \draw[edge] (T.south west) to[out=225,in=135] (M);
  \draw[edge] (T.south east) to[out=315,in=45] (M);
  \draw[edge] (M)--(B);
  \draw[edge] (M)--(0.86,-1.48);
  \draw[edge] (B.south)--(-0.30,-2.65);
  \node[halfedgelabel] at (0.94,-1.31) {$a_2$};
  \node[halfedgelabel] at (-0.30,-2.84) {$a_1$};
  \node[edgelabel] at (-0.8,-1.5) {$3-a_2$};
\end{tikzpicture}
\sgcaption{c}{\Gamma_{3,10}}
\end{subfigure}\hfill
\begin{subfigure}[t]{0.19\textwidth}\centering
\begin{tikzpicture}[vgraph]
  \vgbbox
  \node[component] (T) at (0,0) {2};
  \node[opennode] (M) at (-0.35,-1.05) {};
  \node[opennode] (B) at (0.15,-2.10) {};
  \draw[edge] (T.south west) to[out=225,in=105] (M);
  \draw[edge] (T.south east) to[out=315,in=45] (B);
  \draw[edge] (M)--(-0.92,-1.50);
  \draw[edge] (M)--(B);
  \draw[edge] (B)--(0.15,-2.65);
  \node[edgelabel] at (-0.6,-0.52) {$3$};
  \node[halfedgelabel] at (-1.1,-1.32) {$a_2$};
  \node[halfedgelabel] at (0.15,-2.84) {$a_1$};
  \node[edgelabel] at (-0.6,-1.67) {$3-a_2$};
\end{tikzpicture}
\sgcaption{d}{\Gamma_{1,11}}
\end{subfigure}\hfill
\begin{subfigure}[t]{0.19\textwidth}\centering
\begin{tikzpicture}[vgraph]
  \vgbbox
  \node[component] (T) at (0,0) {2};
  \node[opennode] (M) at (-0.35,-1.05) {};
  \node[opennode] (B) at (0.15,-2.10) {};
  \draw[edge] (T.south west) to[out=225,in=105] (M);
  \draw[edge] (T.south east) to[out=315,in=45] (B);
  \draw[edge] (M)--(-0.92,-1.50);
  \draw[edge] (M)--(B);
  \draw[edge] (B)--(0.15,-2.65);
  \node[edgelabel] at (-0.6,-0.52) {$2$};
  \node[edgelabel] at (1,-1.02) {$2$};
  \node[halfedgelabel] at (-1.1,-1.32) {$a_2$};
  \node[halfedgelabel] at (0.15,-2.84) {$a_1$};
  \node[edgelabel] at (-0.6,-1.67) {$2-a_2$};
\end{tikzpicture}
\sgcaption{e}{\Gamma_{2,11}}
\end{subfigure}

\par\vspace{0.8em}

\begin{subfigure}[t]{0.19\textwidth}\centering
\begin{tikzpicture}[vgraph]
  \vgbbox
  \node[component] (T) at (0,0) {1};
  \node[opennode] (M) at (0,-1.05) {};
  \node[opennode] (B) at (0,-2.10) {};
  \draw[edge] (T.south west) to[out=225,in=135] (M);
  \draw[edge] (T.south east) to[out=315,in=45] (M);
  \draw[edge] (M)--(0.88,-1.40);
  \draw[edge] (M) to[out=230,in=130] (B);
  \draw[edge] (M) to[out=310,in=50] (B);
  \draw[edge] (B)--(0,-2.65);
  \node[halfedgelabel] at (0.96,-1.22) {$a_2$};
  \node[halfedgelabel] at (0,-2.84) {$a_1$};
  \node[edgelabel] at (-0.8,-1.67) {$3-a_2$};
\end{tikzpicture}
\sgcaption{f}{\Gamma_{3,11}}
\end{subfigure}\hfill
\begin{subfigure}[t]{0.19\textwidth}\centering
\begin{tikzpicture}[vgraph]
  \vgbbox
  \node[component] (T) at (0,0) {1};
  \node[opennode] (M) at (0.65,-1.05) {};
  \node[opennode] (B) at (0,-2.10) {};
  \draw[edge] (T.south west) to[out=225,in=135] (B);
  \draw[edge] (T.south) to[out=270,in=150] (M);
  \draw[edge] (T.east) to[out=-15,in=100] (M);
  \draw[edge] (M)--(B);
  \draw[edge] (M)--(1.00,-1.48);
  \draw[edge] (B)--(0,-2.65);
  \node[halfedgelabel] at (1.2,-1.28) {$a_2$};
  \node[halfedgelabel] at (0,-2.84) {$a_1$};
  \node[edgelabel] at (0.9,-1.67) {$3-a_2$};
\end{tikzpicture}
\sgcaption{g}{\Gamma_{4,11}}
\end{subfigure}\hfill
\begin{subfigure}[t]{0.19\textwidth}\centering
\begin{tikzpicture}[vgraph]
  \vgbbox
  \node[component] (T) at (0,0) {2};
  \node[opennode] (M) at (0,-1.05) {};
  \node[opennode] (B) at (0,-2.10) {};
  \draw[edge] (T)--(M);
  \draw[edge] (M)--(0.88,-1.40);
  \draw[edge] (M) to[out=230,in=130] (B);
  \draw[edge] (M) to[out=310,in=50] (B);
  \draw[edge] (B)--(0,-2.65);
  \node[edgelabel] at (-0.34,-0.55) {$3$};
  \node[halfedgelabel] at (0.96,-1.22) {$a_2$};
  \node[halfedgelabel] at (0,-2.84) {$a_1$};
  \node[edgelabel] at (-0.8,-1.67) {$3-a_2$};
\end{tikzpicture}
\sgcaption{h}{\Gamma_{5,11}}
\end{subfigure}\hfill
\begin{subfigure}[t]{0.19\textwidth}\centering
\begin{tikzpicture}[vgraph]
  \vgbbox
  \node[component] (T) at (0,0) {1};
  \node[component] (M) at (0,-1.05) {1};
  \node[opennode] (B) at (0,-2.10) {};
  \draw[edge] (T)--(M);
  \draw[edge] (M.east)--(0.88,-1.40);
  \draw[edge] (M.south west) to[out=225,in=135] (B);
  \draw[edge] (M.south east) to[out=315,in=45] (B);
  \draw[edge] (B)--(0,-2.65);
  \node[halfedgelabel] at (0.96,-1.22) {$a_2$};
  \node[halfedgelabel] at (0,-2.84) {$a_1$};
  \node[edgelabel] at (-0.9,-1.67) {$3-a_2$};
\end{tikzpicture}
\sgcaption{i}{\Gamma_{6,11}}
\end{subfigure}\hfill
\begin{subfigure}[t]{0.19\textwidth}\centering
\begin{tikzpicture}[vgraph]
  \vgbbox
  \node[component] (L) at (-0.50,0) {1};
  \node[component] (R) at (0.50,0) {1};
  \node[opennode] (M) at (0,-1.05) {};
  \node[opennode] (B) at (0,-2.10) {};
  \draw[edge] (L)--(M);
  \draw[edge] (R)--(M);
  \draw[edge] (M)--(0.88,-1.40);
  \draw[edge] (M) to[out=230,in=130] (B);
  \draw[edge] (M) to[out=310,in=50] (B);
  \draw[edge] (B)--(0,-2.65);
  \node[halfedgelabel] at (0.96,-1.22) {$a_2$};
  \node[halfedgelabel] at (0,-2.84) {$a_1$};
  \node[edgelabel] at (-0.8,-1.67) {$3-a_2$};
\end{tikzpicture}
\sgcaption{j}{\Gamma_{7,11}}
\end{subfigure}

\par\vspace{0.8em}

\begin{subfigure}[t]{0.19\textwidth}\centering
\begin{tikzpicture}[vgraph]
  \vgbbox
  \node[component] (T) at (0.30,0) {2};
  \node[opennode] (M) at (-0.30,-1.05) {};
  \node[opennode] (B) at (0,-2.10) {};
  \draw[edge] (T)--(M);
  \draw[edge] (T)--(0.95,-0.45);
  \draw[edge] (M) to[out=230,in=130] (B);
  \draw[edge] (M) to[out=310,in=50] (B);
  \draw[edge] (B)--(0,-2.65);
  \node[halfedgelabel] at (1.1,-0.26) {$a_2$};
  \node[halfedgelabel] at (0,-2.84) {$a_1$};
  \node[edgelabel] at (0.58,-0.7) {$3-a_2$};
\end{tikzpicture}
\sgcaption{k}{\Gamma_{10,11}}
\end{subfigure}\hfill
\begin{subfigure}[t]{0.19\textwidth}\centering
\begin{tikzpicture}[vgraph]
  \vgbbox
  \node[component] (R) at (0.55,0) {2};
  \node[component] (L) at (-0.55,-1.05) {1};
  \node[opennode] (M) at (0,-2.10) {};
  \draw[edge] (R)--(M);
  \draw[edge] (L)--(M);
  \draw[edge] (M)--(-0.62,-2.65);
  \draw[edge] (M)--(0.62,-2.65);
  \node[edgelabel] at (0.4,-1.5) {$3$};
  \node[halfedgelabel] at (-0.72,-2.83) {$a_1$};
  \node[halfedgelabel] at (0.72,-2.83) {$a_2$};
\end{tikzpicture}
\sgcaption{l}{\Gamma_{5,12}}
\end{subfigure}\hfill
\begin{subfigure}[t]{0.19\textwidth}\centering
\begin{tikzpicture}[vgraph]
  \vgbbox
  \node[component] (L) at (-0.55,0) {1};
  \node[component] (R) at (0.55,-1.05) {2};
  \node[opennode] (M) at (0,-2.10) {};
  \draw[edge] (L)--(M);
  \draw[edge] (R)--(M);
  \draw[edge] (M)--(-0.62,-2.65);
  \draw[edge] (M)--(0.62,-2.65);
  \node[edgelabel] at (0.5,-1.67) {$3$};
  \node[halfedgelabel] at (-0.72,-2.83) {$a_1$};
  \node[halfedgelabel] at (0.72,-2.83) {$a_2$};
\end{tikzpicture}
\sgcaption{m}{\Gamma_{6,12}}
\end{subfigure}\hfill
\begin{subfigure}[t]{0.19\textwidth}\centering
\begin{tikzpicture}[vgraph]
  \vgbbox
  \node[component] (T) at (0.35,0) {1};
  \node[component] (L) at (-0.55,-1.05) {1};
  \node[opennode] (U) at (0.25,-1.05) {};
  \node[opennode] (M) at (0,-2.10) {};
  \draw[edge] (T.south west) to[out=225,in=135] (U);
  \draw[edge] (T.south east) to[out=315,in=45] (U);
  \draw[edge] (L)--(M);
  \draw[edge] (U)--(M);
  \draw[edge] (M)--(-0.62,-2.65);
  \draw[edge] (M)--(0.62,-2.65);
  \node[edgelabel] at (0.4,-1.67) {$3$};
  \node[halfedgelabel] at (-0.72,-2.83) {$a_1$};
  \node[halfedgelabel] at (0.72,-2.83) {$a_2$};
\end{tikzpicture}
\sgcaption{n}{\Gamma_{3,12}}
\end{subfigure}\hfill
\begin{subfigure}[t]{0.19\textwidth}\centering
\begin{tikzpicture}[vgraph]
  \vgbbox
  \node[component] (L) at (-0.55,0) {1};
  \node[component] (R) at (0.55,0) {1};
  \node[component] (C) at (0.35,-1.05) {1};
  \node[opennode] (M) at (0,-2.10) {};
  \draw[edge] (L)--(M);
  \draw[edge] (R)--(C);
  \draw[edge] (C)--(M);
  \draw[edge] (M)--(-0.62,-2.65);
  \draw[edge] (M)--(0.62,-2.65);
  \node[edgelabel] at (0.4,-1.67) {$3$};
  \node[halfedgelabel] at (-0.72,-2.83) {$a_1$};
  \node[halfedgelabel] at (0.72,-2.83) {$a_2$};
\end{tikzpicture}
\sgcaption{o}{\Gamma_{7,12}}
\end{subfigure}

\par\vspace{0.8em}

\begin{subfigure}[t]{0.19\textwidth}\centering
\begin{tikzpicture}[vgraph]
  \vgbbox
  \node[component] (L) at (-0.55,0) {1};
  \node[component] (R) at (0.55,0) {1};
  \node[opennode] (U) at (0.25,-1.05) {};
  \node[opennode] (M) at (0,-2.10) {};
  \draw[edge] (L)--(M);
  \draw[edge] (R.south west) to[out=225,in=135] (U);
  \draw[edge] (R.south east) to[out=315,in=45] (U);
  \draw[edge] (U)--(M);
  \draw[edge] (M)--(-0.62,-2.65);
  \draw[edge] (M)--(0.62,-2.65);
  \node[edgelabel] at (0.4,-1.67) {$3$};
  \node[halfedgelabel] at (-0.72,-2.83) {$a_1$};
  \node[halfedgelabel] at (0.72,-2.83) {$a_2$};
\end{tikzpicture}
\sgcaption{p}{\Gamma_{8,12}}
\end{subfigure}

\caption{Additional 3-level graphs in
$\overline{\mathcal H}(a_1,a_2)$, where
$(a_1,a_2)=(3,1)$ or $(2,2)$.}
\label{fig:h13}

\label{fig:h13}
\end{figure}

% ------------------------------------------------------------------
% H(2,2)
% Visible labels: (Q), (R)
% ------------------------------------------------------------------
\begin{figure}[H]
\centering

\begin{subfigure}[t]{0.30\textwidth}
\centering
\begin{tikzpicture}[vgraph]
  \path[use as bounding box] (-1.20,0.42) rectangle (1.20,-2.94);
  \node[component] (A) at (-0.75,0) {1};
  \node[component] (B) at (0.10,0) {1};
  \node[component] (C) at (0.72,-1.05) {1};
  \node[opennode] (M) at (0,-2.10) {};
  \draw[edge] (A)--(M);
  \draw[edge] (B)--(M);
  \draw[edge] (C)--(M);
  \draw[edge] (M)--(-0.55,-2.65);
  \draw[edge] (M)--(0.55,-2.65);
  \node[halfedgelabel] at (-0.68,-2.52) {$2$};
  \node[halfedgelabel] at (0.68,-2.52) {$2$};
\end{tikzpicture}
\sgcaption{Q}{\Gamma_{7,13}}
\end{subfigure}
\hfill
\begin{subfigure}[t]{0.30\textwidth}
\centering
\begin{tikzpicture}[vgraph]
  \path[use as bounding box] (-1.20,0.42) rectangle (1.20,-2.94);
  \node[component] (A) at (-0.75,0) {1};
  \node[component] (B) at (-0.10,-1.05) {1};
  \node[component] (C) at (0.72,-1.05) {1};
  \node[opennode] (M) at (0,-2.10) {};
  \draw[edge] (A.south west) to[out=235,in=135] (M);
  \draw[edge] (B)--(M);
  \draw[edge] (C)--(M);
  \draw[edge] (M)--(-0.55,-2.65);
  \draw[edge] (M)--(0.55,-2.65);
  \node[halfedgelabel] at (-0.68,-2.52) {$2$};
  \node[halfedgelabel] at (0.68,-2.52) {$2$};
\end{tikzpicture}
\sgcaption{R}{\Gamma_{6,13}}
\end{subfigure}

\caption{Additional 3-level graphs in $\overline{\mathcal H}(2,2)$.}
\label{fig:h22}
\end{figure}

\subsection{Boundary of $\cH(4)$}

If $\mu=(4)$, then $\partial \overline{\cH}(\mu)$ has codimension-two boundary strata corresponding to level graphs depicted in Figure~\ref{fig:h4}. The middle-level components are contained in the following generalized strata, up to forgetting the labeling of some zeroes and poles:
\begin{center}
\renewcommand{\arraystretch}{1.25}
\begin{tabular}{@{}r@{\;}l@{\qquad\qquad}r@{\;}l@{}}
\sgtag{A} & $\cH(2,0,-2,-2)$
    & \sgtag{I} & $\cH(0,0,-2)$ \\

\sgtag{B} & $\cH(2,-2,-2)$
    & \sgtag{J} & $\cH(2,-2,-2)$ \\

\sgtag{C} & $\cH(2,0,-4)$
    & \sgtag{K} & $\cH(0,0,-2)$ \\

\sgtag{D} & $\cH(2,0,-2)$
    & \sgtag{L} & $\cH(0)$ \\

\sgtag{E} & $\cH(1,1,-2,-2)$
    & \sgtag{M} & $\cH(0,0,0,-2)$ \\

\sgtag{F} & $\cH(1,1,-4)$
    & \sgtag{N} & $\cH(2,-2)$ \\

\sgtag{G} & $\cH(1,1,-2)$
    & \sgtag{O} & $\cH(0,0)$ \\

\sgtag{H} & $\cR(1,1,-2,-2)$
    & \sgtag{P} & $\cH(0)$
\end{tabular}
\end{center}

\begin{itemize}
    \item The strata in \sgtag{A}, \sgtag{W}, \sgtag{M} are isomorphic to $\mathcal{M}_{0,4}\cong \mathbb{C}\setminus\{0,1\}$.
    \item The strata in \sgtag{C}, \sgtag{B}, \sgtag{D}, \sgtag{I}, \sgtag{K}, \sgtag{F}, \sgtag{J} are isomorphic to $\mathcal{M}_{0,3}=\{\ast\}$.
    \item The strata in \sgtag{G}, \sgtag{L}, \sgtag{P} are $\mathcal{M}_{1,1}$.
    \item The stratum in \sgtag{O} is $\mathcal{M}_{1,2}$.
    \item The stratum in \sgtag{H} is a singleton $\{\ast\}$. 
    \item The stratum in \sgtag{N} is isomorphic to a rational curve. 
\end{itemize}

\subsection{Boundary of $\cH(3,1)$}

If $\mu=(3,1)$, then $\partial \overline{\cR}(\mu)$ has boundary divisors corresponding to level graphs depicted in Figure~\ref{fig:h4} with the zero of order 4 is replaced by two zeroes of orders $3$ and $1$, respectively. The middle level strata are unchanged. 

Additionally, it has boundary divisors corresponding to the level graphs depicted in Figure~\ref{fig:h13} with $(a_1,a_2)=(3,1)$. The middle-level components are contained in the following generalized strata, up to forgetting the labeling of some zeroes and poles:

\begin{center}
\renewcommand{\arraystretch}{1.25}
\begin{tabular}{@{}r@{\;}l@{\qquad\qquad}r@{\;}l@{}}
\sgtag{a} & $\cH(1,1,-4)$
    & \sgtag{i} & $\cH(1,1,0,-2)$ \\

\sgtag{b} & $\cH(1,1,-2)$
    & \sgtag{j} & $\cR(1,1,0,-2,-2)$ \\

\sgtag{c} & $\cH(1,1,-2,-2)$
    & \sgtag{k} & $\cH(1,0,-3)$ \\

\sgtag{d} & $\cH(1,1,-4)$
    & \sgtag{l} & $\cH(0)$ \\

\sgtag{e} & $\cH(1,0,-3)$
    & \sgtag{m} & $\cH(2)$ \\

\sgtag{f} & $\cH(1,1,0,-2,-2)$
    & \sgtag{n} & $\cH(0)\times\cH(2,-2,-2)\times\mathbb{C}^{\ast}$ \\

\sgtag{g} & $\cH(1,1,-2,-2)$
    & \sgtag{o} & $\cH(2,-2)$ \\

\sgtag{h} & $\cH(1,1,0,-4)$
    & \sgtag{p} & $\cH(2,-2,-2)$
\end{tabular}
\end{center}

\begin{itemize}
    \item The strata in \sgtag{g}, \sgtag{c}, \sgtag{h} are isomorphic to $\mathcal{M}_{0,4}$.
    \item The strata in \sgtag{a}, \sgtag{e}, \sgtag{d}, \sgtag{k}, \sgtag{p} are isomorphic to $\mathcal{M}_{0,3}=\{\ast\}$.
    \item The stratum in \sgtag{f} is isomorphic to $\mathcal{M}_{0,5}$
    \item The stratum in \sgtag{j} is isomorphic to $\mathbb{P}^1\setminus \{\text{4 pts}\}$. The stratum $\cR(1,1,-2 ,-2)$ is a singleton, so $\cR(1,1,0,-2,-2)$ is rational as it is the set of marked point on the unique element of $\cR(1,1,-2,-2)$.   
    \item The stratum in \sgtag{b} is an open subvariety of $\cM_{1,2}$
    \item The stratum in \sgtag{i} is an open subvariety of $\cM_{1,3}$
    \item The strata in \sgtag{l},\sgtag{m},\sgtag{n} has CKgP.
    \item The stratum in \sgtag{o} is a rational curve. 
\end{itemize}

\subsection{Boundary of $\cH(2,2)$}

If $\mu=(2,2)$, then $\partial \overline{\cH}(\mu)$ has boundary divisors corresponding to level graphs depicted in Figure~\ref{fig:h4} with the zero of order 4 is replaced by two double zeroes. The middle level strata are unchanged. 

Additionally, it has (possibly empty) boundary divisors corresponding to the level graphs depicted in Figure~\ref{fig:h13} with $(a_1,a_2)=(2,2)$. The middle-level components are contained in the following generalized strata, up to forgetting the labeling of some zeroes and poles:

\begin{center}
\renewcommand{\arraystretch}{1.25}
\begin{tabular}{@{}r@{\;}l@{\qquad\qquad}r@{\;}l@{}}
\sgtag{a} & $\cH(2,0,-4)$
    & \sgtag{i} & $\cH(2,0,0,-2)$ \\

\sgtag{b} & $\cH(2,0,-2)$
    & \sgtag{j} & \textit{Empty} \\

\sgtag{c} & $\cH(2,0,-2,-2)$
    & \sgtag{k} & $\cH(0,0,-2)$ \\

\sgtag{d} & $\cH(2,0,-4)$
    & \sgtag{l} & $\cH(0)$ \\

\sgtag{e} & \textit{Empty}
    & \sgtag{m} & $\cH(2)$ \\

\sgtag{f} & $\cH(2,0,0,-2,-2)$
    & \sgtag{n} &
      $\cH(0)\times\cH(2,-2,-2)\times\mathbb{C}^{\ast}$ \\

\sgtag{g} & $\cH(2,0,-2,-2)$
    & \sgtag{o} & $\cH(2,-2)$ \\

\sgtag{h} & $\cH(2,0,0,-4)$
    & \sgtag{p} & $\cH(2,-2,-2)$
\end{tabular}
\end{center}

\begin{itemize}
    \item The strata in \sgtag{c}, \sgtag{g}, \sgtag{h} are isomorphic to $\mathcal{M}_{0,4}$.
    \item The strata in \sgtag{a}, \sgtag{d}, \sgtag{k}, \sgtag{p} are isomorphic to $\mathcal{M}_{0,3}=\{\ast\}$.
    \item \sgtag{f} is isomorphic to $\mathcal{M}_{0,5}$
    \item The stratum in \sgtag{b} is a rational surface. 
    \item The stratum in \sgtag{i} has CKgP by Proposition~\ref{prop:h200-2}.
    \item The stratum in \sgtag{l} is isomorphic to $\mathcal{M}_{1,1}$.
    \item The stratum in \sgtag{m} is isomorphic to a rational curve.
    \item The strata in \sgtag{n} is a $\mathbb{G}_m$-bundle over $\cH(0)\times\cH(2,-2,-2)$, so it has CKgP.
    \item The stratum in \sgtag{o} is isomorphic to a rational curve. 
\end{itemize}

Finally, $\overline{\cH}(2,2)$ has two more additional divisors, corresponding to the level graphs depicted in Figure~\ref{fig:h22}. The middle-level components are contained in the following generalized strata, up to forgetting the labeling of some zeroes and poles:

\begin{center}
\renewcommand{\arraystretch}{1.25}
\begin{tabular}{@{}r@{\;}l@{\qquad\qquad}r@{\;}l@{}}
\sgtag{Q} & $\cH(0)$
    & \sgtag{R} & $\cH(0) \times \cH(0)\times \mathbb{C}^\ast$ \\

\end{tabular}
\end{center}

\begin{itemize}
    \item The strata in \sgtag{Q} is isomorphic to a rational curve.
    \item The strata in \sgtag{R} is a $\mathbb{G}_m$-bundle over $\cH(0)\times\cH(0)$, thus has the CKgP. 
\end{itemize}

\bibliography{bib}
\bibliographystyle{siam}

\end{document}